\documentclass{amsart}

\usepackage{graphicx} 
\usepackage{longtable}
\usepackage{array}

\usepackage{amssymb} 
\usepackage{amsrefs}

\usepackage[all]{xy}
\newdir{ >}{{}*!/-5pt/@{>}} 
\SelectTips{cm}{} 

\usepackage{tikz}
\usetikzlibrary{arrows}

\DeclareMathOperator{\Ext}{Ext}
\DeclareMathOperator{\Hom}{Hom}

\newcommand{\F}{\mathbb{F}}
\newcommand{\RP}{\mathbb{R}P}
\newcommand{\cA}{\mathcal{A}}
\newcommand{\cE}{\mathcal{E}}
\newcommand{\cW}{\mathcal{W}}

\newcommand{\lra}{\longrightarrow}
\newcommand{\lla}{\longleftarrow}

\newcommand{\llla}[1]{\stackrel{#1}{\longleftarrow}}
\newcommand{\<}{\langle}
\renewcommand{\>}{\rangle}

\newtheorem{theorem}{Theorem}
\newtheorem{proposition}[theorem]{Proposition}
\newtheorem{lemma}[theorem]{Lemma}
\newtheorem{corollary}[theorem]{Corollary}
\theoremstyle{definition}

\theoremstyle{remark}
\newtheorem{remark}[theorem]{Remark}

\author{Robert R. Bruner}
\address{Department of Mathematics\\
         Wayne State University\\
         Detroit, Michigan  48202\\
         USA
         }
\email{robert.bruner@wayne.edu}

\author{Christian Nassau}
\address{Jenaer Weg 31\\
65931 Frankfurt\\
Germany}
\email{nassau@nullhomotopie.de}

\author{Sean Tilson}
\address{Mathematik und Informatik\\
Bergische Universit\"at Wuppertal\\
42119 Wuppertal\\
Germany}
\email{Sean.Tilson@gmail.com}

\thanks{Thanks to the Simons Foundation, and to the
Isaac Newton Institute for
Mathematical Sciences for support and hospitality during the program
{\em Homotopy  Harnessing Higher Structures} during which some of this
work was done.
This program was supported by EPSRC grant  number
EP/R014604/1.}

\title{Steenrod operations and $\cA$-module extensions}

\begin{document}
\maketitle
\tableofcontents

\section{Introduction}

Explicit extensions
\[
\cE_x : 0 \lla \F_2 \lla M_0 \lla \cdots \lla M_{s-1} \lla \Sigma^t \F_2 \lla 0
\]
representing cocycles  $x \in \Ext_{\cA}^{s,t}(\F_2,\F_2)$ can be used to 
calculate Steenrod operations $Sq^i : \Ext_{\cA}^{s,t}(\F_2,\F_2) \lra
\Ext_{\cA}^{s+i,2t}(\F_2,\F_2)$ by a method devised by Christian Nassau
and described in Section~\ref{sect:efficient}.   
To be effective, the modules $M_i$
need to be small.   A practical method for finding small extensions is
described in Section~\ref{sect:practical}, 
and used in Sections~\ref{sect:e0} and \ref{sect:d0}.
These extensions can be used to identify explicit cocycles representing
the values of Steenrod operations in the
minimal resolutions produced by the first author's computer programs.
This information is useful in determining
differentials in the Adams spectral sequence.

The extension for $f_0$ was found by the third author as a part of
an undergraduate research project at Wayne State many years ago.   Students of
Agn\`es Beaudry at the University of Colorado have produced an extension
for $d_0$ as part of a Research Experiences for
Undergraduates project in the summer of 2019. We give an alternate extension
realizing $d_0$ in Section~\ref{sect:d0}.   In response to the first
version of this paper, Dexter Chua has done extensive computer calculations
of extensions.    It seems that an algorithm to produce minimal 
extensions may not be far off.

Andy Baker has also used these extensions to investigate the realizability of
$\cA$-modules.

\section{Extensions for the subalgebra generated by the $h_i$}

The cocycle $h_i \in \Ext_{\cA}^{1,2^i}(\F_2,\F_2)$ is represented by an
extension
\[
0 \lla  \F_2 \lla M \lla \Sigma^{2^i} \F_2 \lla 0
\]
which we represent diagrammatically as follows:

\begin{figure}[h]
\begin{tikzpicture}

\draw [] (0,0) circle [radius=0.07];
\draw [] (4,0) circle [radius=0.07];
\draw [] (4,2) circle [radius=0.07];
\draw [] (8,2) circle [radius=0.07];

\draw [<-] (0.1,0) -- (3.9,0);
\draw [<-] (4.1,2) -- (7.9,2);
\draw []   (4,0.1) .. controls (3.7,0.7) and (3.7,1.3) .. (4,1.9);

\node [left] at  (3.8, 1) {$Sq^{2^i}$};

\end{tikzpicture}
\end{figure}

Extensions representing elements in the subalgebra generated
by the $h_i$ can then be represented simply by splicing. However, it can
be difficult to recognize when two such are equivalent.   For example,
the extension in Figure~\ref{fig:h0h1} representing $h_0 h_1$
\begin{figure}[h]
\begin{tikzpicture}

\foreach \s in {0,1,2,3}
{
\node [left] at (-0.4,\s) {$\s$};
}

\draw [] (0,0) circle [radius=0.07];
\draw [] (3,0) circle [radius=0.07];
  \node [below right] at (3,0) {\small{$x_0$}};
\draw [] (3,1) circle [radius=0.07];
\draw [] (6,1) circle [radius=0.07];
  \node [below right] at (6,1) {\small{$x_1$}};
  \node [below right] at (3,0) {\small{$x_0$}};
\draw [] (6,3) circle [radius=0.07];
\draw [] (9,3) circle [radius=0.07];
  \node [below right] at (9,3) {\small{$x_2$}};

\draw [<-] (0.1,0) -- (2.9,0);
\draw [<-] (3.1,1) -- (5.9,1);
\draw [<-] (6.1,3) -- (8.9,3);

\draw      (3,0.1) -- (3,0.9);
\draw []   (6,1.1) .. controls (5.7,1.7) and (5.7,2.3) .. (6,2.9);

\node at (0,-1) {$\F_2$};
\node at (3,-1) {$M_0$};
\node at (6,-1) {$M_1$};
\node at (9,-1) {$\Sigma^3 \F_2$};

\draw [<-] (0.3,-1) -- (2.7,-1);
\draw [<-] (3.3,-1) -- (5.7,-1);
\draw [<-] (6.3,-1) -- (8.6,-1);

\end{tikzpicture}
\caption{The splice of $\cE_{h_0}$ and $\cE_{h_1}$}
\label{fig:h0h1}
\end{figure}
is equivalent to a split extension, but the easiest way to check
this is to compute the cocycle $c_2 : C_2 \lra \Sigma^3 \F_2$ which appears
in a chain map from a resolution $C_* \lra \F_2$ to the extension.
\begin{figure}[h]
\begin{tikzpicture}

\node at (0,1) {$\F_2$};
\node at (3,1) {$C_0$};
\node at (6,1) {$C_1$};
\node at (9,1) {$C_2$};

\draw [<-] (0.3,1) -- (2.7,1);
\draw [<-] (3.3,1) -- (5.7,1);
\draw [<-] (6.3,1) -- (8.6,1);

\draw [] (-0.05,.8) -- (-0.05,-.7);
\draw [] (0.02,.8) -- (0.02,-.7);
\draw [->] (3,.8) -- (3,-.7);
\draw [->] (6,.8) -- (6,-.7);
\draw [->] (9,.8) -- (9,-.7);

\node at (0,-1) {$\F_2$};
\node at (3,-1) {$M_0$};
\node at (6,-1) {$M_1$};
\node at (9,-1) {$\Sigma^3 \F_2$};

\node at (3.3,0) {$c_0$};
\node at (6.3,0) {$c_1$};
\node at (9.3,0) {$c_2$};

\draw [<-] (0.3,-1) -- (2.7,-1);
\draw [<-] (3.3,-1) -- (5.7,-1);
\draw [<-] (6.3,-1) -- (8.6,-1);

\end{tikzpicture}
\caption{The chain  map to the splice of $\cE_{h_0}$ and $\cE_{h_1}$}
\label{fig:h0h1ch}
\end{figure}
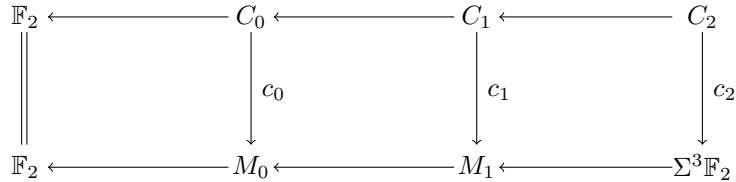
In the notation of Section~\ref{sec:minimal}, such a chain map 
has the following nonzero values:
\begin{itemize}
\item[0:] $c_0(0_0) = x_0$
\item[1:] $c_1(1_0) = x_1$
\end{itemize}
Since the term $Sq^{2}(1_0)$ does not appear in $d(2_g)$ for any $g$,
the cocycle $c_2 $ is zero, and the extension is equivalent to one which is
split.

\section{Steenrod operations in $\Ext$}

Since $\cA$ is a cocommutative Hopf algebra, we have a symmetric monoidal product
\[
M,N \mapsto M \otimes_{\F_2} N
\]
on the category of $\cA$-modules by
pulling back the natural $\cA \otimes \cA$-module structure on $M \otimes N$
along the coproduct $\cA \lra \cA \otimes \cA$.   The comparison theorem then gives
an $\cA$-linear diagonal map $\Delta : C_* \lra C_* \otimes C_*$ which induces
a product 
\[
\Hom_{\cA}(C_*,\F_2) \otimes \Hom_{\cA}(C_*,\F_2) \lra \Hom_{\cA}(C_*,\F_2).
\]
This must, by general nonsense, agree with the Yoneda product in 
$\Ext_{\cA}(\F_2,\F_2)$.   
This implies that this product is independent of the coproduct of $\cA$.  

The cocommutativity of $\cA$ implies that $\tau \Delta \simeq \Delta$,
where $\tau : C_* \otimes C_*  \lra C_*  \otimes C_* $ is the transposition,
showing that this product is commutative.   A chain homotopy $\Delta_1 : \tau \Delta \simeq
\Delta$ gives a {\em `cup-1' product} $x \cup_1 y := (x \otimes y) \Delta_1$.    We can
iterate this construction, getting maps
\[
\Delta_i : C_\sigma \lra (C \otimes C)_{\sigma+i},
\]
with $\Delta_0 = \Delta$, satisfying
\begin{equation}
(d\otimes 1 + 1 \otimes d) \Delta_i + \Delta_i d = \Delta_{i-1} - \tau \Delta_{i-1}
\label{eqn:chmap}
\end{equation}
for each $i>0$.    Using these we define Steenrod operations\\[-2ex]
\begin{figure}[h]
\begin{tikzpicture}
\node at (0,0) {$Sq^i(x) : C_{s+i}$};
\draw [->] (1.3,0) -- (3,0);
\node at (4,0) {$(C\otimes C)_{2s}$};
\draw [->] (4.9,0) -- (6.2,0);
\node at (8,0) {$\Sigma^t \F_2 \otimes \Sigma^t \F_2 \cong \Sigma^{2t} \F_2$};
\node at (2.2,0.2) {$\Delta_{s-i}$};
\node at (5.5,0.2) {$x \otimes x$};
\end{tikzpicture}
\end{figure}\\
for $0 \leq i \leq s$.  See~\cite{V168} and \cite{Hinf}*{Ch.IV, Sec.2} for details.

In practice, a major obstacle to computing these is the size of the
modules $(C \otimes C)_{2s}$.  They are far too large for hand calculation and
are even too large for practical machine calculations.   Conceptually, the maps
$\Delta_i$ carry within them all possible decompositions, but we are interested 
in only that small part detected by the cocycle $x \otimes x$.
This is like writing out the entire multiplication table, where all we really need
is to know how to multiply by one indecomposable.

\section{An efficient method of computing Steenrod operations}
\label{sect:efficient}

The second author proposed a method for making these calculations practical in
an email to the first author \cite{Nas97}.  
He observed that if $\cE_x$ is an extension corresponding to $x$,

\begin{figure}[h]
\begin{tikzpicture}[scale=0.8]

\node at (0,2) {$0$};
\node at (2,2) {$\F_2$};
\node at (4,2) {$C_0$};
\node at (6,2) {$C_1$};
\node at (8,2) {$\cdots$};
\node at (10,2) {$C_{s-1}$};
\node at (12,2) {$C_{s}$};
\node at (14,2) {$\cdots$};

\foreach \s in {0,2,...,12}
\foreach \i in {0,2}
\draw [<-] (\s+0.5,\i) -- (\s+1.5,\i);

\node at (-1,0) {$\cE_x:$};
\node at (0,0) {$0$};
\node at (2,0) {$\F_2$};
\node at (4,0) {$M_0$};
\node at (6,0) {$M_1$};
\node at (8,0) {$\cdots$};
\node at (10,0) {$M_{s-1}$};
\node at (12,0) {$\Sigma^t\F_2$};
\node at (14,0) {$0$};

\draw [] (1.95,1.7) -- (1.95,0.4);
\draw [] (2.02,1.7) -- (2.02,0.4);

\foreach \s in {4,6,10,12}
{
  \draw [->] (\s,1.7) -- (\s,0.4);
}
\node [right] at (4,1) {$x_0$};
\node [right] at (6,1) {$x_1$};
\node [right] at (10,1) {$x_{s-1}$};
\node [right] at (12,1) {$x_{s}=x$};

\end{tikzpicture}
\end{figure}

\noindent
then the composites
$(x \otimes x)\Delta_i : C_*  \lra C_*  \otimes C_*  \lra M_* \otimes M_*$
can be computed directly, without passing through $C_* \otimes C_*$, and that,
    if the modules $M_i$ are small, then this calculation is computationally feasible.

We can make this precise as follows.   Let $\cW$ be the usual $\F_2[C_2]$ free resolution
of $\F_2$:   $\cW_i$ is free on one generator $e_i$ and the differential is 
$d(e_i) = (1+\tau)e_{i-1}$.   
Give elements of   $\cW_i \otimes C_s$  homological degree $s+i$ and internal
degree the same as in $C_s$.   

\begin{lemma}
\label{lem:main}
The collection of chain homotopies $\Delta_i$ 
gives a $C_2$-equivariant chain map $D: \cW \otimes C_* \lra C_* \otimes C_*$ by
\[
D(e_i \otimes x) = \Delta_i(x)
\]
and vice versa.   
\end{lemma}

\begin{proof}
The equations (\ref{eqn:chmap}) which say that $\Delta_i$ is a chain homotopy 
between $\Delta_{i-1}$ and $\tau \Delta_{i-1}$ 
equivalently say that $D$ is a chain map.
\end{proof}

\begin{theorem}
\label{thm:main}
Let $\widetilde{\Delta}_0 : C_* \lra M_* \otimes M_*$
be a chain map lifting the isomorphism $\F_2 \lra \F_2 \otimes \F_2$ and,
for $i > 0$, let 
$\widetilde{\Delta}_{i} : 
 \widetilde{\Delta}_{i-1} \simeq \tau \widetilde{\Delta}_{i-1}$
be a chain homotopy as in equation~(\ref{eqn:chmap}).
Then $Sq^i(x)$ is the cohomology class of the cocycle
$\widetilde{\Delta}_{s-i} : C_{s+i} \lra 
\Sigma^t \F_2 \otimes \Sigma^t \F_2  \cong \Sigma^{2t} \F_2$.
\end{theorem}

\begin{proof}
A chain map $x : C_* \lra M_*$ gives a $C_2$-equivariant chain map 
$x \otimes x : C_* \otimes C_* \lra M_* \otimes M_*$.   
Composing, we get a $C_2$-equivariant chain map
\[
\widetilde{D} : \cW \otimes C_* \lra C_* \otimes C_* \lra M_* \otimes M_*
\]
covering the isomorphism 
$\F_2 \lra \F_2 \otimes \F_2$
in the category of $\cA$-modules.

By the comparison theorem, any two such $\widetilde{D}$
are chain homotopic.   Thus, rather than computing it as 
the composite $(x \otimes x) D$, we may
compute one directly by lifting
the map from $\cW_0 \otimes \F_2 $ to $ \F_2 \otimes  \F_2$
which sends $e_0 \otimes 1$ to $1 \otimes 1$ to a chain map.

Lemma~\ref{lem:main} relating $D$ and $\Delta$
applies equally well to the relation between $\widetilde{D}$
and $\widetilde{\Delta}$, proving the theorem.
\end{proof}

Calculation of the Steenrod operations on $x$ now depends on finding sufficiently
amenable extensions $\cE_x$.

\section{The canonical extension}

By taking pushouts, starting from the cocycle $x : C_s \lra M_s = \Sigma^t \F_2$,
we obtain an extension $\cE_x$ associated to $x$.   Precisely, $M_i$ is the pushout
of $C_{i+1} \lra C_i$ and $C_{i+1} \lra M_{i+1}$.    However, this canonical
extension is no better than the resolution $C$ itself, in that $M_i \cong C_i$ for
$i < s-1$, so the size problem associated to $C \otimes C$ is not eliminated.

\section{An extension for $c_0$ and the $Sq^i(c_0)$}

The lowest degree class not in the subalgebra generated by the $h_i$ is 
the class $c_0 \in \Ext_{\cA}(\F_2,\F_2)$.   Let $\cE_{c_0}$ be the extension
in Figure~\ref{fig:Ec0}.

\begin{figure}[h]
\begin{tikzpicture}[scale=0.5]

\draw [] (0,0) circle [radius=0.07];
\node [below right] at (4,0) {\tiny{$k_0$}};
\draw [] (4,0) circle [radius=0.07];
\draw [] (4,1) circle [radius=0.07];
\draw [] (4,3) circle [radius=0.07];
\draw [] (4,7) circle [radius=0.07];
\draw (4,0.1) -- (4,0.9);
\draw [] (3.9,1.1) .. controls (3.7,1.7) and (3.7,2.3) .. (3.9,2.9);
\draw [] (3.9,3) -- (3.7,3) -- (3.7,7) -- (3.9,7);
\node [below right] at (8,1) {\tiny{$k_1$}};
\draw [] (8,1) circle [radius=0.07];
\draw [] (8,3) circle [radius=0.07];
\draw [] (8,5) circle [radius=0.07];
\draw [] (8,7) circle [radius=0.07];
\draw [] (8,9) circle [radius=0.07];
\draw [] (7.9,1.1) .. controls (7.7,1.7) and (7.7,2.3) .. (7.9,2.9);
\draw [] (7.9,3) -- (7.7,3) -- (7.7,7) -- (7.9,7);
\draw [] (7.9,7.1) .. controls (7.7,7.7) and (7.7,8.3) .. (7.9,8.9);
\draw [] (8.1,1) -- (8.3,1) -- (8.3,4.8) -- (8.1,4.95);
\draw [] (8.1,5.05) -- (8.3,5.2) -- (8.3,9) -- (8.1,9);
\node [below right] at (12,5) {\tiny{$k_2$}};
\draw [] (12,5) circle [radius=0.07];
\draw [] (12,9) circle [radius=0.07];
\draw [] (12,11) circle [radius=0.07];
\draw [] (12.1,5) -- (12.3,5) -- (12.3,9) -- (12.1,9);
\draw [] (12.1,9.1) .. controls (12.3,9.7) and (12.3,10.3) .. (12.1,10.9);
\node [below right] at (16,11) {\tiny{$k_3$}};
\draw [] (16,11) circle [radius=0.07];

\draw [<-] (0.2,0) -- (3.8,0);
\draw [<-] (4.2,1) -- (7.8,1);
\draw [<-] (4.2,3) -- (7.5,3);
\draw [<-] (4.2,7) -- (7.5,7);
\draw [<-] (8.3,5) -- (11.8,5);
\draw [<-] (8.5,9) -- (11.8,9);
\draw [<-] (12.2,11) -- (15.8,11);

\foreach \s in {0,1,3,5,7,9,11}
  \node at (-1,\s) {\tiny{$\s$}};

\node at (0,-2) {\small{$\F_2$}};
\draw [<-] (0.5,-2) -- (3.5,-2);
\node at (4,-2) {\small{$C_0$}};
\draw [<-] (4.5,-2) -- (7.5,-2);
\node at (8,-2) {\small{$C_1$}};
\draw [<-] (8.5,-2) -- (11.5,-2);
\node at (12,-2) {\small{$C_2$}};
\draw [<-] (12.5,-2) -- (15.3,-2);
\node at (16.2,-2) {\small{$\Sigma^{11}\F_2$}};

\end{tikzpicture}
\caption{The extension $\cE_{c_0}$}
\label{fig:Ec0}
\end{figure}

\begin{proposition}
The extension $\cE_{c_0}$  represents the cocycle $c_0 = 3_3 \in \Ext_{\cA}^{3,11}(\F_2,\F_2)$.
\end{proposition}

\begin{proof}
It suffices to exhibit the following  chain map $c:C_* \lra \cE_{c_0}$
from the minimal resolution in Section~\ref{sec:minimal}
to $\cE_{c_0}$.  The nonzero values of $c$ are
\begin{itemize}
\item[0:] $c(0_0) = k_0$
\item[1:] $c(1_0) = k_1$
\item[2:] $c(2_2) = k_2$ 
\item[\phantom{2:}] $c(2_4) = Sq^4 k_2$
\item[3:] $c(3_3) = k_3$
\end{itemize}
\end{proof}

\begin{proposition}
\label{prop:sqc0}
The squaring operations on $c_0$ are
$Sq^*(c_0) = (6_5, 5_6, 4_6, 3_9) = (c_0^2, h_0 e_0, f_0, c_1)$.
\end{proposition}

This allows us to determine which one of the two indecomposable elements in
$\Ext^{4,22}_{\cA}(\F_2,\F_2)$ is $f_0$, defined as $Sq^1(c_0)$.   (Note that
definitions of $f_0$ by Toda brackets cannot distinguish between 
$f_0 = 4_6$ and $f_0 + h_1^3 h_4 = 4_6 + 4_7$ because of their indeterminacy.)

\begin{corollary}
The element $f_0 = Sq^1(c_0)$ is dual to $4_6$ in the minimal resolution of Section~\ref{sec:minimal}.
\end{corollary}

\begin{remark}
We identify $\cA$ generators of the minimal resolution with their duals in $\Ext_{\cA}$
to avoid having to explicitly note the duality.  For example, $4_6 + 4_7$ denotes the cocycle
which evaluates to $1$ on each of $4_6$ and $4_7$, while $4_6$ denotes the cocycle which
evaluates to $1$ on $4_6$ and to $0$ on $4_7$.
\end{remark}

\begin{proof}[Proof of Proposition~\ref{prop:sqc0}]
It suffices to record the values of $\Delta_i$, where $\Delta_0 = \Delta$ is the chain map
$C \lra \cE_{c_0} \otimes \cE_{c_0}$ lifting the identity map of $\F_2$, and the $\Delta_i$ for
$i>0$ satisfy equation~(\ref{eqn:chmap}).  
We do this in Tables~\ref{table:Dc0}
and \ref{table:Dic0}.

\begin{longtable}{|>{$} l <{$} | >{$} l <{$} | }
\caption{Nonzero values of the diagonal map on $c_0$
\label{table:Dc0}}\\
\hline
s_g & \Delta  \\
\hline
\endfirsthead
\caption{Nonzero values of the diagonal map on $c_0$ (cont).}\\
\hline
s_g & \Delta(s_g) \\
\hline
\endhead
\hline
\endfoot

0_0 & k_0 \otimes k_0  \\
\hline
1_0 & k_0 \otimes k_1 + k_1 \otimes k_0  \\
1_1 & Sq^1 k_0 \otimes k_1   \\
\hline
2_0 & k_1 \otimes k_1   \\
2_1 & Sq^2 k_1 \otimes k_1   \\
2_2 & k_0 \otimes k_2 + k_2 \otimes k_0   \\
2_3 & Sq^2 Sq^1 k_0 \otimes k_2  \\
2_4 & k_0 \otimes Sq^4 k_2 + Sq^4 k_2 \otimes k_0  \\
2_5 & k_1 \otimes Sq^4 Sq^4 k_1 + Sq^4 Sq^4 k_1 \otimes k_1  \\
\hline
3_1 & k_1 \otimes k_2 + k_2 \otimes k_1 \\
3_2 & k_1 \otimes Sq^4 k_2 + Sq^4 k_2 \otimes k_1 + k_2 \otimes Sq^4 k_1 \\
3_3 & k_0 \otimes k_3 + k_3 \otimes k_0 \\
3_4 & Sq^1 k_0 \otimes k_3 + k_3 \otimes Sq^1 k_0 +
      Sq^4 Sq^2 k_1 \otimes k_2 + Sq^2 k_1 \otimes Sq^4 k_2 \\
3_6 & Sq^4 k_2 \otimes Sq^4 Sq^4 k_1 \\
\hline
4_3 & Sq^4 Sq^2 k_1 \otimes k_3 + k_3 \otimes Sq^4 Sq^2 k_1 + 
      Sq^4 k_2 \otimes Sq^4 k_2 \\
\hline
5_4 & Sq^4 k_2 \otimes k_3 + k_3 \otimes Sq^4 k_2 \\ 
5_6 & Sq^2 Sq^4 k_2 \otimes k_3  \\ 
\hline
6_5 & k_3 \otimes k_3 \\

\end{longtable}

\begin{longtable}{| >{$} l <{$} |  >{$} l <{$} |  >{$} l <{$} |  >{$} l <{$} | }
\caption{Nonzero values of the higher diagonal maps on $c_0$
\label{table:Dic0}}\\
\hline
s_g  &  \Delta_1(s_g) & \Delta_2(s_g) & \Delta_3(s_g) \\
\hline
\endfirsthead
\caption{Nonzero values of the higher diagonal maps on $c_0$ (cont).}\\
\hline
s_g &  \Delta_1(s_g) & \Delta_2(s_g) & \Delta_3(s_g) \\
\hline
\endhead
\hline
\endfoot

1_1 & k_1 \otimes k_1 & & \\
\hline
2_3 & 
      Sq^2 k_1 \otimes k_2 + k_2 \otimes Sq^2 k_1 &
      & \\
2_5 & 
      Sq^4 k_1 \otimes k_2 &
      k_2 \otimes k_2 & \\
2_8 & 
      Sq^4 Sq^4 k_1 \otimes Sq^4 k_2 &
      Sq^4 k_2 \otimes Sq^4 k_2 & \\
\hline
3_2 &
     k_2 \otimes k_2 &
     & \\
3_6 &
     Sq^4 k_2 \otimes Sq^4 k_2 &
     & \\
3_7 &
     Sq^4 k_2 \otimes Sq^2 Sq^4 k_2 &
     & \\
3_9 &
    & Sq^2 Sq^4 k_2 \otimes k_3 
    & k_3 \otimes k_3 \\
\hline
4_6 & Sq^2 Sq^4 k_2 \otimes k_3 
    & k_3 \otimes k_3 & \\
    
\hline
5_6 & k_3 \otimes k_3 & & \\

\end{longtable}
\end{proof}

\section{An extension for $c_1$ and the $Sq^i(c_1)$}

The extension for $c_1$ is the `double' of that for $c_0$:   we simply
replace every $Sq^i$ by $Sq^{2i}$ in $\cE_{c_0}$ to get $\cE_{c_1}$.

\begin{proposition}
The extension $\cE_{c_1}$  represents the cocycle 
$c_1 = 3_9 \in \Ext_{\cA}^{3,22}(\F_2,\F_2)$.
\end{proposition}

\begin{proof}
It suffices to exhibit this chain map $c:C_* \lra \cE_{c_1}$
from the minimal resolution in Section~\ref{sec:minimal}
to $\cE_{c_1}$.  The nonzero values of $c$ are
\begin{itemize}
\item[0:] $c(0_0) = k_0$
\item[1:] $c(1_1) = k_1$
\item[2:] $c(2_5) = k_2$ 
\item[\phantom{2:}] $c(2_8) = Sq^8 k_2$
\item[3:] $c(3_9) = k_3$
\end{itemize}
\end{proof}

\begin{proposition}
\label{prop:sqc1}
The squaring operations on $c_1$ are
$Sq^*(c_1) = (6_{17}, 5_{19}, 4_{19}, 3_{19}) =
(c_1^2, h_1 e_1, f_1, c_2)$.
\end{proposition}

This allows us to determine which element in the
minimal resolution of Section~\ref{sec:minimal} is $f_1 = Sq^1(c_1)$.

\begin{corollary}
The element $f_1 = Sq^1(c_1)$ is dual to $4_{19}$ in the minimal 
resolution of Section~\ref{sec:minimal}.
\end{corollary}

\begin{proof}[Proof of Proposition~\ref{prop:sqc1}]
It suffices to record the values of $\Delta_i$, where
$\Delta_0 = \Delta$ is the chain map
$C \lra \cE_{c_1} \otimes \cE_{c_1}$ lifting the identity
map of $\F_2$, and the $\Delta_i$ for
$i>0$ satisfy equation~(\ref{eqn:chmap}).
We do this in Tables~\ref{table:Dc1}
and \ref{table:Dic1}.

\begin{longtable}{|>{$} l <{$} | >{$} l <{$} | }
\caption{Nonzero values of the diagonal map on $c_1$
\label{table:Dc1}}\\
\hline
s_g & \Delta  \\
\hline
\endfirsthead
\caption{Nonzero values of the diagonal map on $c_1$ (cont).}\\
\hline
s_g & \Delta(s_g) \\
\hline
\endhead
\hline
\endfoot

0_0 & k_0 \otimes k_0  \\
\hline
1_1 & k_0 \otimes k_1 + k_1 \otimes k_0  \\
1_2 & Sq^2 k_0 \otimes k_1   \\
\hline
2_1 & k_1 \otimes k_1   \\
2_3 & Sq^4 k_1 \otimes k_1   \\
2_5 & k_0 \otimes k_2 + k_2 \otimes k_0   \\
2_6 & Sq^4 Sq^2 k_0 \otimes k_2  \\
2_8 & k_0 \otimes Sq^8 k_2 + Sq^8 k_2 \otimes k_0  \\
2_9 & k_1 \otimes Sq^8 Sq^8 k_1 + Sq^8 Sq^8 k_1 \otimes k_1  \\
\hline
3_4 & k_1 \otimes k_2 + k_2 \otimes k_1 \\
3_7 & k_1 \otimes Sq^8 k_2 + Sq^8 k_2 \otimes k_1 + k_2 \otimes Sq^8 k_1 \\
3_9 & k_0 \otimes k_3 + k_3 \otimes k_0 \\
3_{10} & Sq^2 k_0 \otimes k_3 + k_3 \otimes Sq^2 k_0 +
      Sq^8 Sq^4 k_1 \otimes k_2 + Sq^4 k_1 \otimes Sq^8 k_2 \\
3_{14} & Sq^8 Sq^8 k_1 \otimes Sq^8 k_2 \\
\hline
4_{13} & Sq^8 Sq^4 k_1 \otimes k_3 + k_3 \otimes Sq^8 Sq^4 k_1 + 
      Sq^8 k_2 \otimes Sq^8 k_2 \\
\hline
5_{16} & Sq^8 k_2 \otimes k_3 + k_3 \otimes Sq^8 k_2 \\ 
5_{19} & Sq^4 Sq^8 k_2 \otimes k_3  \\ 
\hline
6_{17} & k_3 \otimes k_3 \\

\end{longtable}

\begin{longtable}{| >{$} l <{$} |  >{$} l <{$} |  >{$} l <{$} |  >{$} l <{$} | }
\caption{Nonzero values of the higher diagonal maps on $c_1$
\label{table:Dic1}}\\
\hline
s_g  &  \Delta_1(s_g) & \Delta_2(s_g) & \Delta_3(s_g) \\
\hline
\endfirsthead
\caption{Nonzero values of the higher diagonal maps on $c_1$ (cont).}\\
\hline
s_g &  \Delta_1(s_g) & \Delta_2(s_g) & \Delta_3(s_g) \\
\hline
\endhead
\hline
\endfoot

1_2 & k_1 \otimes k_1 & & \\
\hline
2_6 & 
      Sq^4 k_1 \otimes k_2 + k_2 \otimes Sq^4 k_1 &
      & \\
2_9 & 
      Sq^8 k_1 \otimes k_2 &
      k_2 \otimes k_2 & \\
2_{13} & 
      Sq^8 Sq^8 k_1 \otimes Sq^8 k_2 &
      Sq^8 k_2 \otimes Sq^8 k_2 & \\
\hline
3_7 &
     k_2 \otimes k_2 &
     & \\
3_{14} &
     Sq^8 k_2 \otimes Sq^8 k_2 &
     & \\
3_{16} &
     Sq^8 k_2 \otimes Sq^4 Sq^8 k_2 &
     & \\
3_{19} &
    & Sq^4 Sq^8 k_2 \otimes k_3 
    & k_3 \otimes k_3 \\
\hline
4_{19} & Sq^4 Sq^8 k_2 \otimes k_3 
       & k_3 \otimes k_3 & \\
\hline
5_{19} & k_3 \otimes k_3 & & \\

\end{longtable}
\end{proof}

\section{An extension for $f_0$ and the $Sq^i(f_0)$}

Let $\cE_{f_0}$ be the extension
\[
0 \lla \F_2 \lla N_0 \lla N_1 \lla N_2 \lla N_3 \lla \Sigma^{22} \F_2 \lla 0
\]
with 
\[
N_0 = \frac{A}{(Sq^1,Sq^2,Sq^{12}+Sq^{0,4}, Sq^{16})}[0,15]
\]
\[
N_1 = \frac{\Sigma^4 A \oplus \Sigma^8 A}
  {\left( (Sq^4,0), (Sq^{2,1}, Sq^1), (Sq^{0,2},Sq^2), (Sq^8,0), 
           (0,Sq^4 Sq^1), (Sq^{12},Sq^8) \right) } [0,17]
\]
\[
N_2 = \frac{\Sigma^5 A}{(Sq^1,Sq^6,Sq^{8}, Sq^{2,2})}[0,18]
\]
\[
N_3 = \frac{\Sigma^{10} A}{(Sq^1,Sq^2,Sq^{4})}[0,22]
\]
Here, $M[0,n]$ denotes the quotient of $M$ by classes in degrees greater than $n$.
The maps in the extension are
\begin{itemize}
\item $d(k_1) = Sq^4 k_0$
\item $d(k_1') = Sq^8 k_0$
\item $d(k_2) = Sq^1 k_1$
\item $d(k_3) = Sq^5 k_2$
\item $d(k_4) = Sq^{12} k_3$
\end{itemize}

\begin{proposition}
The extension $\cE_{f_0}$  represents the cocycle $f_0 = 4_6 \in \Ext_{\cA}^{4,22}(\F_2,\F_2)$.
\end{proposition}

\begin{proof}
It suffices to exhibit the following chain map $f:C_* \lra \cE_{f_0}$
from the minimal resolution in Section~\ref{sec:minimal}
to $\cE_{f_0}$.  The nonzero values of $f$ are
\begin{itemize}
\item[0:] $f(0_0) = k_0$
\item[1:] $f(1_2) = k_1$
\item[\phantom{1:}] $f(1_3) = k_1'$
\item[2:] $f(2_2) = k_2$ 
\item[\phantom{2:}] $f(2_4) = Sq^4 k_2$
\item[3:] $f(3_2) = k_3$
\item[4:] $f(4_6) = k_4$
\end{itemize}
\end{proof}

\begin{proposition}
\label{prop:sqf0}
The squaring operations on $f_0$ are
\[
Sq^*(f_0) = (0, 7_{13}+7_{14}, 6_{16}, 0, 4_{19}) = 
(0, h_3 r_0, y_0, 0, f_1).
\]
\end{proposition}

\begin{proof}[Proof of Proposition~\ref{prop:sqf0}]
It suffices to record the values of $\Delta_i$, where $\Delta_0 = \Delta$ is the chain map
$C \lra \cE_{f_0} \otimes \cE_{f_0}$ lifting the identity map of $\F_2$, and the $\Delta_i$ for
$i>0$ satisfy equation~(\ref{eqn:chmap}).
We do this in Tables~\ref{table:Df0}
and \ref{table:Dif0}.

\begin{longtable}{|>{$} l <{$} | >{$} l <{$} | }
\caption{Nonzero values of the diagonal map on $f_0$
\label{table:Df0}}\\
\hline
s_g & \Delta  \\
\hline
\endfirsthead
\caption{Nonzero values of the diagonal map on $f_0$ (cont).}\\
\hline
s_g & \Delta(s_g) \\
\hline
\endhead
\hline
\endfoot

0_0 & k_0 \otimes k_0  \\
\hline
1_2 & k_0 \otimes k_1 + k_1 \otimes k_0  \\
1_3 & k_0 \otimes k_1' + k_1' \otimes k_0 + Sq^4 k_0 \otimes k_1  \\
1_4 & Sq^4 k_0 \otimes Sq^4 k_1' +
      Sq^4 k_1' \otimes Sq^4 k_0 + \\
    & Sq^6 k_0 \otimes Sq^2 k_1' +
      Sq^2 k_1' \otimes Sq^6 k_0 + \\
    & Sq^8 k_0 \otimes k_1' \\
\hline
2_2 & k_0 \otimes k_2 + k_2 \otimes k_0   \\
2_3 & k_1 \otimes k_1   \\
2_4 & k_0 \otimes Sq^4 k_2 +
   Sq^4 k_2 \otimes k_0   +
   Sq^4 k_0 \otimes k_2 
   \\
2_5 & k_1 \otimes Sq^2 k_1  \\
2_6 & k_1 \otimes Sq^4 k_1' +
         Sq^4 k_1' \otimes k_1 + \\
    & Sq^1 k_1 \otimes Sq^3 k_1' +
         Sq^3 k_1' \otimes Sq^1 k_1 + \\
    & Sq^2 k_1 \otimes Sq^2 k_1' +
      Sq^2 k_1' \otimes Sq^2 k_1 + \\
    & Sq^3 k_1 \otimes Sq^1 k_1' + 
       Sq^1 k_1' \otimes Sq^3 k_1 +
       k_1' \otimes k_1'  \\
\pagebreak
2_7 & k_1 \otimes Sq^5 k_1' +
        + Sq^5 k_1' \otimes k_1 + \\
    & Sq^1 k_1 \otimes Sq^4 k_1' +
          Sq^4 k_1' \otimes Sq^1 k_1 + \\
    & Sq^2 k_1 \otimes Sq^3 k_1' +
       Sq^3 k_1' \otimes Sq^2 k_1 + \\
    & Sq^3 k_1 \otimes Sq^2 k_1' + 
       Sq^2 k_1' \otimes Sq^3 k_1 +  k_1' \otimes Sq^1 k_1' \\
2_8 & k_1 \otimes Sq^6 k_1' +
      Sq^6 k_1' \otimes k_1 \\
    & Sq^1 k_1 \otimes Sq^5 k_1' +
      Sq^5 k_1' \otimes Sq^1 k_1 + \\
    & Sq^2 k_1 \otimes Sq^4 k_1' + 
           Sq^4 k_1' \otimes Sq^2 k_1  + \\
    & Sq^{(0,1)} k_1 \otimes Sq^3 k_1' +
      Sq^3 k_1' \otimes Sq^{(0,1)} k_1  + \\
    & Sq^3 k_1 \otimes Sq^4 Sq^2 Sq^1 k_1 +
      Sq^4Sq^2Sq^1 k_1 \otimes Sq^3 k_1 + \\
    & Sq^2 k_1' \otimes k_1' \\
2_9 & k_1' \otimes Sq^4 k_1' + Sq^1 k_1' \otimes Sq^3 k_1' + \\
    & Sq^5 k_1' \otimes Sq^3 k_1 + Sq^2Sq^1 k_1 \otimes Sq^5 k_1'\\
\hline
3_2 & k_0 \otimes k_3 + k_3 \otimes k_0 \\
3_3 & k_1 \otimes Sq^2 k_2 +
      Sq^2 k_2 \otimes k_1 +
      Sq^2 k_1 \otimes k_2 +
      k_2 \otimes Sq^2 k_1 \\
3_4 & Sq^3 k_1 \otimes k_2 + k_2 \otimes Sq^3 k_1 \\
3_5 & k_1' \otimes Sq^4 k_2 +
      Sq^4 k_2 \otimes k_1' + \\
    & Sq^1 k_1' \otimes Sq^3 k_2 +
      Sq^2 k_1' \otimes Sq^2 k_2 + \\
    & Sq^4 k_1' \otimes k_2 + 
      k_2 \otimes Sq^4 k_1' + \\
    & Sq^4 Sq^2 k_2 \otimes Sq^2 k_1 \\
3_6 & Sq^1 k_1' \otimes Sq^4 k_2 \\
3_7 & Sq^1 k_1' \otimes Sq^4 Sq^2 k_2 + 
      Sq^4 Sq^2 k_2 \otimes Sq^1 k_1' + \\
    & Sq^2 k_1' \otimes Sq^5 k_2 +
      Sq^5 k_2 \otimes Sq^2 k_1' + \\
    & Sq^3 k_1' \otimes Sq^4 k_2 +
      Sq^4 k_2 \otimes Sq^3 k_1'  \\
3_8 & k_1 \otimes Sq^8 Sq^4 k_2 + 
      Sq^4 Sq^2 k_2 \otimes Sq^2 k_1' + \\
    & Sq^3 k_1 \otimes Sq^6 Sq^3 k_2 +
      Sq^6 Sq^3 k_2 \otimes Sq^3 k_1 + \\
    & Sq^3 k_1' \otimes Sq^5 k_2 +
      Sq^5 k_2 \otimes Sq^3 k_1' + \\
    & Sq^5 k_1' \otimes Sq^3 k_2 +
      Sq^3 k_2 \otimes Sq^5 k_1' \\
\pagebreak
3_9 & k_1 \otimes Sq^9 Sq^4 k_2 +
      Sq^9 Sq^4 k_2 \otimes k_1 + \\
    & Sq^1 k_1 \otimes Sq^8 Sq^4 k_2 +
      Sq^8 Sq^4 k_2 \otimes Sq^1 k_1 + \\
    & (Sq^3 Sq^1 k_1 + k_1') \otimes Sq^6 Sq^3 k_2 +
      Sq^6 Sq^3 k_2 \otimes (Sq^3 Sq^1 k_1  + k_1')+ \\
    & Sq^4 Sq^2 Sq^1 k_1 \otimes Sq^4 Sq^2 k_2 + \\
    & Sq^2 k_1' \otimes Sq^5 Sq^2 k_2 +
      Sq^5 Sq^2 k_2 \otimes Sq^2 k_1' + \\
    & Sq^7 k_1' \otimes Sq^2 k_2 +
      Sq^2 k_2 \otimes Sq^7 k_1'  \\
3_{10} & Sq^3 k_1 \otimes Sq^8 Sq^4 k_2 +
       Sq^8 Sq^4 k_2 \otimes Sq^3 k_1 + \\
     & Sq^8 Sq^4 Sq^1 k_1 \otimes Sq^2 k_2 +
       Sq^5 Sq^2 Sq^1 k_1 \otimes Sq^5 Sq^2 k_2 + \\
     & Sq^2 k_1' \otimes Sq^6 Sq^3 k_2 +
       Sq^5 k_1' \otimes Sq^4 Sq^2 k_2 + \\
     & Sq^6 k_1' \otimes Sq^5 k_2 +
       Sq^5 k_2 \otimes Sq^6 k_1' + \\
     & Sq^7 k_1' \otimes Sq^4 k_2 +
       Sq^4 k_2 \otimes Sq^7 k_1' + \\
     & Sq^9 Sq^4 k_2 \otimes Sq^2 k_1 \\
\hline
4_2 & k_2 \otimes Sq^3 k_2 + Sq^3 k_2 \otimes k_2 \\ 
4_3 & Sq^2 k_2 \otimes Sq^4 Sq^2 k_2 + Sq^4 Sq^2 k_2 \otimes Sq^2 k_2 + \\
    & Sq^3 k_2 \otimes Sq^5 k_2 + \\
    & Sq^4 k_2 \otimes Sq^4 k_2  \\ 
4_4 & Sq^4 k_2 \otimes Sq^5 k_2 \\
4_5 & Sq^3 k_1' \otimes k_3 + 
      k_3 \otimes Sq^3 k_1' + \\ 
    & Sq^4 Sq^2 k_2 \otimes Sq^5 k_2 \\
4_6 & k_0 \otimes k_4 + k_4 \otimes k_0 + \\
    & k_1 \otimes Sq^8 k_3 +
      Sq^8 k_3 \otimes k_1 + \\
    & Sq^4 k_1' \otimes k_3 +
      k_3 \otimes Sq^4 k_1' + \\
    & Sq^4 Sq^2 k_2 \otimes Sq^4 Sq^2 k_2 \\
4_7 & k_1 \otimes Sq^8 k_3 +
      Sq^8 k_3 \otimes k_1 + \\
    & k_2 \otimes Sq^8 Sq^4 k_2 + \\
    & Sq^4 Sq^2 k_2 \otimes Sq^4 Sq^2 k_2 + \\
    & Sq^6 Sq^3 k_2 \otimes Sq^3 k_2  \\
4_8 & Sq^2 k_1 \otimes Sq^8 k_3 +
      Sq^8 k_3 \otimes Sq^2 k_1 + \\
    & Sq^6 k_1' \otimes k_3 + 
      k_3 \otimes Sq^6 k_1' + \\
    & Sq^5 Sq^2 k_2 \otimes Sq^5 Sq^2 k_2 + \\
    & Sq^6 Sq^3 k_2 \otimes Sq^5 k_2 +
      Sq^5 k_2 \otimes Sq^6 Sq^3 k_2 \\
\pagebreak
4_9 & k_1 \otimes Sq^4 Sq^8 k_3 +
      Sq^4 Sq^8 k_3 \otimes k_1 + \\
    & Sq^3 k_2 \otimes Sq^9 Sq^4 k_2 +
      Sq^9 Sq^4 k_2 \otimes Sq^3 k_2 + \\
    & Sq^4 k_2 \otimes Sq^8 Sq^4 k_2 \\
4_{10} & Sq^5 k_2 \otimes Sq^8 Sq^4 k_2 \\
4_{11} & Sq^4 k_1' \otimes Sq^4 Sq^8 k_3 +
         Sq^4 Sq^8 k_3 \otimes Sq^4 k_1' + \\
       & Sq^8 Sq^4 k_2 \otimes Sq^8 Sq^4 k_2 \\
4_{12} & Sq^8 Sq^4 k_2 \otimes Sq^9 Sq^4 k_2 \\
4_{14} & Sq^{15} k_0 \otimes k_4 + k_4 \otimes Sq^{15} k_0 \\
\hline
5_5 & Sq^5 k_2 \otimes k_3  \\ 
5_7 & k_2 \otimes  Sq^8 k_3 + Sq^8 k_3 \otimes k_2 \\ 
5_9 & k_1 \otimes k_4 + k_4 \otimes k_1 + \\
    & Sq^3 k_2 \otimes Sq^8 k_3 + Sq^8 k_3 \otimes Sq^3 k_2 \\
5_{10} & Sq^2 k_1 \otimes k_4 + k_4 \otimes Sq^2 k_1 + \\
    & Sq^5 k_2 \otimes Sq^8 k_3 + Sq^8 k_3 \otimes Sq^5 k_2 \\
5_{11} & Sq^3 k_1 \otimes k_4 + k_4 \otimes Sq^3 k_1 + \\
    & Sq^4 Sq^2 k_2 \otimes Sq^8 k_3 + Sq^8 k_3 \otimes Sq^4 Sq^2 k_2 \\
5_{12} & Sq^8 Sq^4 k_2 \otimes Sq^8 k_3 + Sq^8 k_3 \otimes Sq^8 Sq^4 k_2 \\
5_{13} & Sq^6 k_1' \otimes k_4 + k_4 \otimes Sq^6 k_1' + \\
    & Sq^6 Sq^3 k_2 \otimes Sq^4 Sq^8 k_3 + Sq^4 Sq^8 k_3 \otimes Sq^6 Sq^3 k_2 \\
5_{14} & Sq^9 Sq^4 k_2 \otimes Sq^8 k_3 \\
\hline
6_8 & Sq^2 k_2 \otimes k_4 + k_4 \otimes Sq^2 k_2 \\
6_9 & Sq^4 Sq^8 k_3 \otimes k_3 \\
6_{10} & Sq^6 Sq^3 k_2 \otimes k_4 + k_4 \otimes Sq^6 Sq^3 k_2 + \\
       & Sq^8 k_3 \otimes Sq^8 k_3 \\
\hline
7_{11} & Sq^8 k_3 \otimes k_4 +  k_4 \otimes Sq^8 k_3 \\
7_{13} & Sq^4 Sq^8 k_3 \otimes k_4 \\
7_{14} & Sq^4 Sq^8 k_3 \otimes k_4 \\
\end{longtable}

\pagebreak

\begin{longtable}{| >{$} l <{$} |  >{$} l <{$} |  >{$} l <{$} |  >{$} l <{$} |  >{$} l <{$} | }
\caption{Nonzero values of the higher diagonal maps on $f_0$
\label{table:Dif0}}\\
\hline
s_g  &  \Delta_1(s_g) & \Delta_2(s_g) & \Delta_3(s_g) & \Delta_4(s_g) \\
\hline
\endfirsthead
\caption{Nonzero values of the higher diagonal maps on $f_0$ (cont).}\\
\hline
s_g  &  \Delta_1(s_g) & \Delta_2(s_g) & \Delta_3(s_g) & \Delta_4(s_g) \\
\hline
\endhead
\hline
\endfoot

1_3 & k_1 \otimes k_1 & & & \\
1_4 & k_1' \otimes k_1' & & & \\
\hline
2_4 & 
       k_1 \otimes k_2 + k_2 \otimes k_1 & & & \\
2_5 & 
      Sq^1 k_1 \otimes k_2 & k_2 \otimes k_2 & & \\
2_8 & 
      Sq^1 k_1' \otimes Sq^4 k_2 & Sq^4 k_2 \otimes Sq^4 k_2 & & \\
\hline
3_2 &
     k_2 \otimes k_2 &
     & & \\
3_4 &
     Sq^2 k_2 \otimes k_2 &
     & & \\
3_5 &
     Sq^2 k_2 \otimes Sq^5 k_2 +  & & & \\
     &
     Sq^5 k_2 \otimes Sq^2 k_2 +  & & & \\
     &
     Sq^5 Sq^2 k_2 \otimes k_2  &
     & & \\
3_6 &
     Sq^4 k_2 \otimes Sq^4 k_2 &
     & & \\
3_7 & 0 &
     Sq	^5 k_2 \otimes k_3
     & k_3 \otimes k_3 & \\
3_8 &
     Sq^6 Sq^3 k_2 \otimes Sq^2 k_2 + & & & \\
     &
     Sq^5 Sq^2 k_2 \otimes Sq^4 k_2  &
     & & \\
3_9 &
     Sq^4 Sq^2 k_2 \otimes Sq^4 Sq^2 k_2 &
    &  & \\
3_{10} &
     Sq^8 Sq^4 k_2 \otimes Sq^2 k_2 + & & & \\
     &
     Sq^2 k_2 \otimes Sq^8 Sq^4 k_2 + 
       & & & \\
       &
     Sq^5 Sq^2 k_2 \otimes Sq^5 Sq^2 k_2 + & & & \\
     &
     Sq^6 Sq^3 k_2 \otimes Sq^5 k_2 &
     & & \\
\hline
4_{19} &  &
    & Sq^4 Sq^8 k_3 \otimes k_4 & k_4 \otimes k_4 \\
    
\hline
5_5 & k_3 \otimes k_3 & & & \\
5_{14} & Sq^8 k_3 \otimes Sq^8 k_3 & & & \\
\hline
6_{16} &
      Sq^4 Sq^8 k_3 \otimes k_4
      & k_4 \otimes k_4 & & \\
\hline
7_{11} &
      & & & \\
7_{13} &
      k_4 \otimes k_4
      & & & \\
7_{14} &
      k_4 \otimes k_4
      & & & \\

\end{longtable}
\end{proof}

\section{A  practical way to find small extensions}
\label{sect:practical}

Here is a systematic way of finding a small extension $\cE_x$.   
Given any extension $\cE_x$, we can factor $\cE_x$ into short exact sequences
\[
 0 \lla \F_2 \llla{p_0} M_0 \lla M_1' \lla 0
\]
\[
 0 \lla M_i' \llla{p_i} M_i \lla M_{i+1}' \lla 0
\]
\[
 0 \lla M_{s-1}' \llla{p_{s-1}} M_{s-1} \lla \Sigma^t \F_2 \lla 0.
\]
In the long exact sequences these induce, we have
\[
\xymatrix{
\Ext_\cA^{s,t}(M_0,\F_2)
&
\Ext_\cA^{s-1,t}(M_1,\F_2)
&
\cdots
&
\Ext_\cA^{1,t}(M_{s-1},\F_2)
&
\\
\Ext_\cA^{s,t}(\F_2,\F_2)
\ar_{p_0^*}[u]
&
\Ext_\cA^{s-1,t}(M_1',\F_2)
\ar_(0.55){\partial}[l]
\ar_{p_1^*}[u]
&
\cdots
\ar_(0.3){\partial}[l]
&
\Ext_\cA^{1,t}(M_{s-1}',\F_2)
\ar_(0.7){\partial}[l]
\ar_{p_{s-1}^*}[u]
&
\Ext_\cA^{0,t}(\Sigma^t \F_2,\F_2)
\ar_(0.5){\partial}[l]
\\
}
\]

The cocycle $x \in \Ext_\cA^{s,t}(\F_2,\F_2)$ satisfies
\[
x = \partial(y_1) = \partial\partial(y_2) = \cdots = \partial^{s-1}(y_{s-1})
= \partial^{s}(1_{\Sigma^t \F_2})
\]
for cocycles $y_i \in \Ext_\cA^{s-i,t}(M_i',\F_2)$.   

This has the virtue of simplifying our task dramatically.  Rather than
needing to find all the $M_i$  realizing $x$ at once, we can produce them
one at a time, in order.

To find $M_0$, it suffices to find any epimorphism  $p_0$ such that $p_0^*(x) = 0$.
We then let $M_1' = \ker(p_0)$ and choose any lift $y_1 \in \Ext_\cA(M_1',\F_2)$
of $x$.  We then repeat the process inductively: choose an epimorphism  $p_1$ such that
$p_1^*(y_1) = 0$, set $M_2' = \ker(p_1)$, and choose a lift $y_2$ of $y_1$, etcetera.

Further, we can search for such a $p_i$ systematically, since the natural map
$\cA \otimes M_i' \lra M_i'$ will certainly work, and we can usually find a
small subquotient of $\cA \otimes M_i'$ which maps onto $M_i'$ and
has $p_i^*(y_i) = 0$.

\section{An extension for $e_0$ and the $Sq^i(e_0)$}
\label{sect:e0}

We will work through the method of the preceding section to obtain an extension
realizing $e_0 = 4_5 \in \Ext_\cA^{4,21}(\F_2,\F_2)$.  

We start by observing that we may choose $M_0$ to be the
(desuspension of) the subquotient of $F_2[x] = H^* \RP^\infty$ spanned
by $x^i$ for $i=1,2,4,8,16$.    Computing the image of $e_0$ under the projection
$M_0 \lra \F_2$  is the same as computing the action of $\Ext_\cA(\F_2,\F_2)$
on the cocycle $0_0$ generating $\Ext_\cA^0(M_0,\F_2)$, and this is easily
checked to be $0$ since $\Ext_\cA^{4,21}(M_0,\F_2) = 0$.   
See Figure~\ref{chart:M0}.

\begin{figure}
{\fbox{\includegraphics[scale=0.5,trim=2cm 1cm 1cm 1cm, clip ]{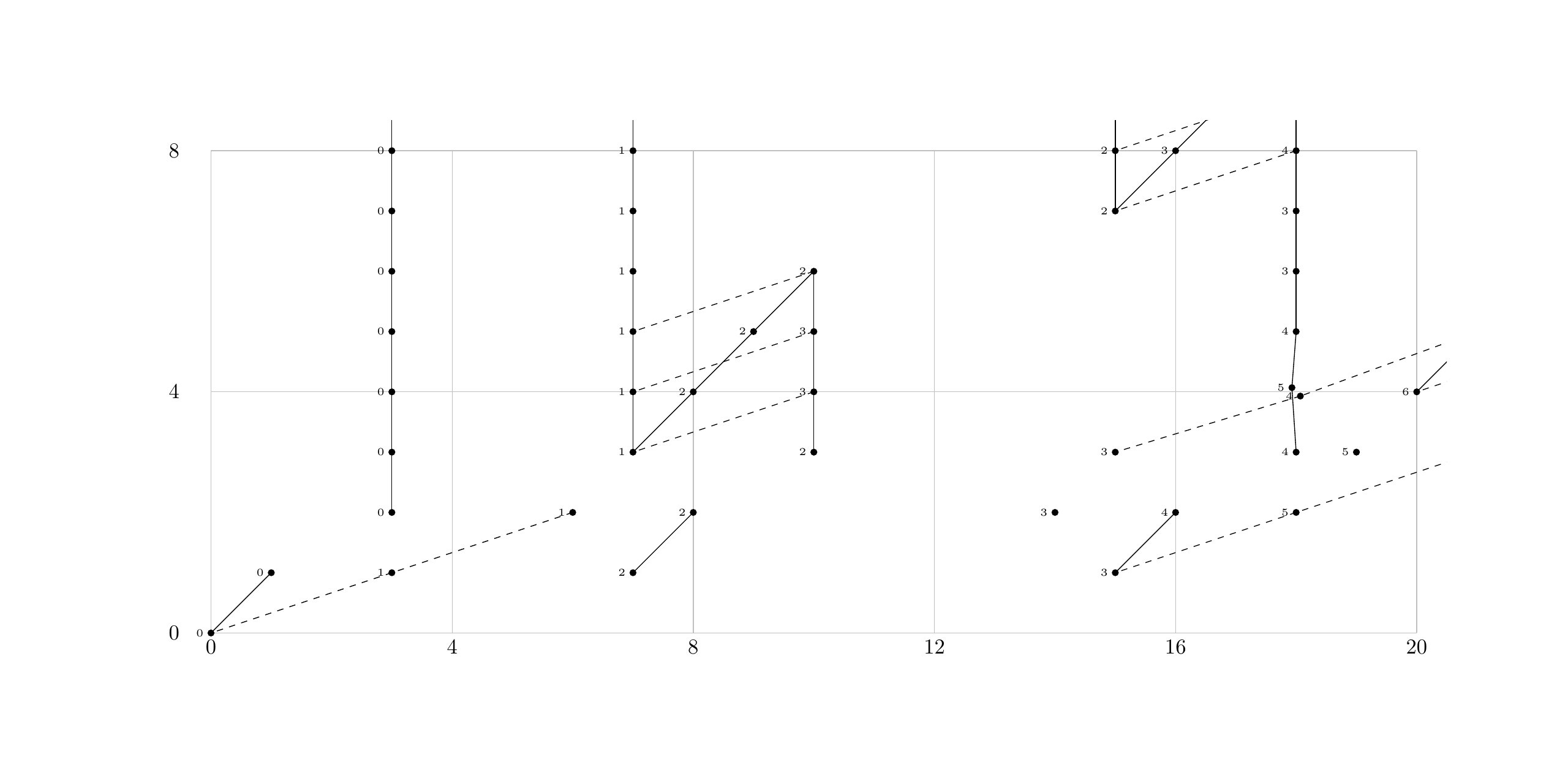}}}
\caption{$\Ext_\cA(M_0,\F_2)$}
\label{chart:M0}
\end{figure}

The kernel $M_1'$ is the submodule of $M_0$
spanned by $x^i$ for $i=2,4,8,16$.
We find $\Ext_\cA^{3,21}(M_1',\F_2) = \< 3_9, 3_{10} \>$ with
$\partial(3_9) = e_0$ and $\partial(3_{10}) = 0$.  We choose $y_1 = 3_9$.
This completes the step involving $\F_2 \lla M_0 \lla M_1'$.
See Figure~\ref{chart:M1p}.

\begin{figure}
{\fbox{\includegraphics[scale=0.5,trim=2cm 1cm 1cm 1cm, clip ]{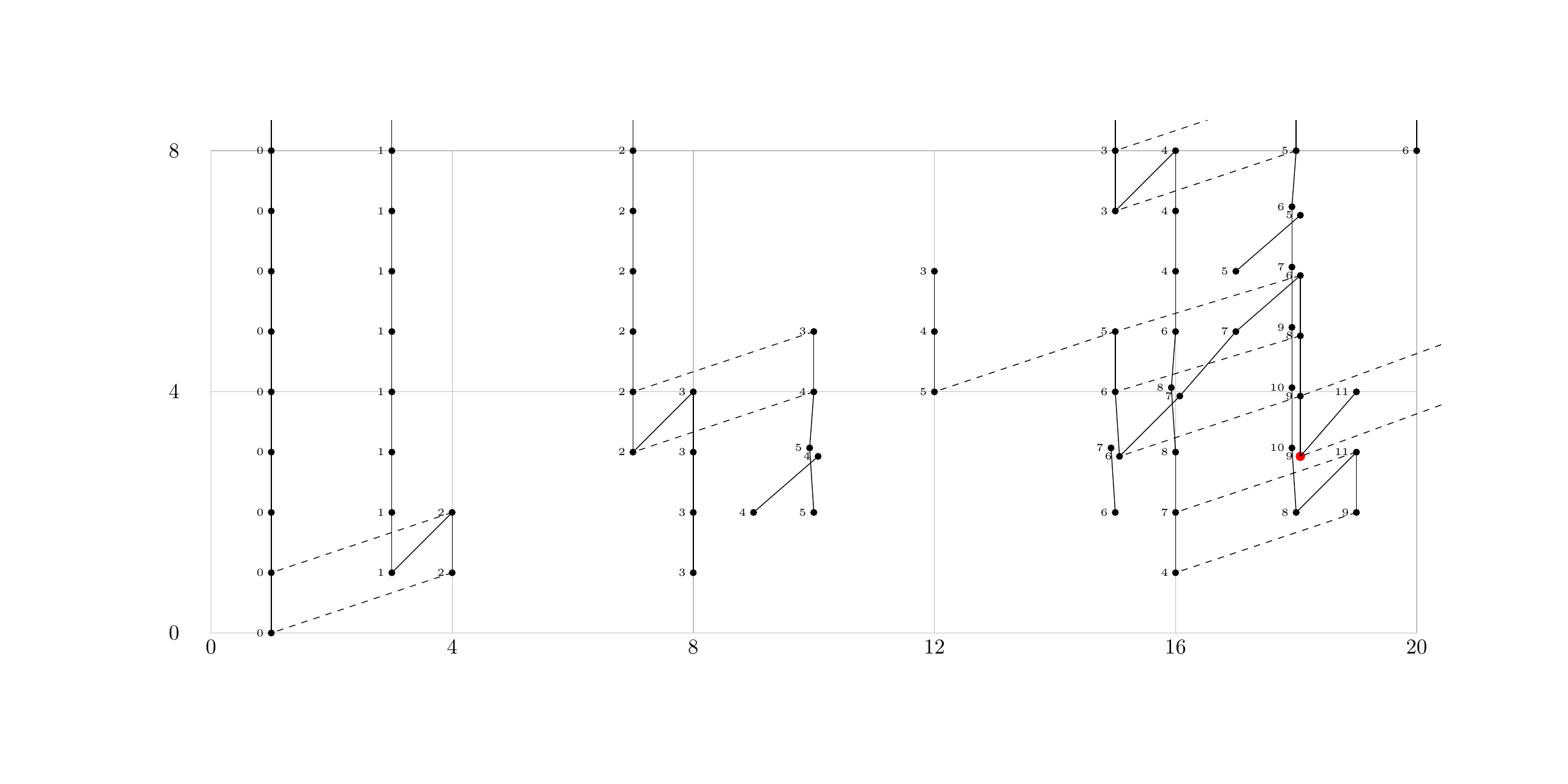}}}
\caption{$\Ext_\cA(M_1',\F_2)$}
\label{chart:M1p}
\end{figure}

Since $M_1'$ is concentrated in odd degrees, its $\cA$ action factors through
the even degree quotient, $\cA \lra \Phi\cA$.   A small piece of this will suffice.
Let $\Phi\cA(1)$ be the double of $\cA(1)$ with $\cA$-action in which $Sq^8$
is nonzero only on the class in degree $2$.  We can compute that 
the map in $\Ext$ induced by the tensor product of $M_1'$ with the
quotient map $\Phi\cA(1) \lra F_2$ sends $3_9 \in \Ext_\cA(M_1',\F_2)$
to 0 in $\Ext_\cA(\Phi\cA(1) \otimes M_1',\F_2)$.  The submodule  of
$\Phi\cA(1) \otimes M_1'$ generated by the bottom class, truncated above degree $21$
is $20$ dimensional over $\F_2$, and could be used as $M_1$.   However, studying
the kernel of the map $M_1 \lra M_1'$, we find that the following subquotient,
which is only $10$ dimensional over $\F_2$ suffices:
\[
M_1 = \frac{\Sigma \cA}{(Sq^1, Sq^{(0,1)},Sq^{(0,0,1)},Sq^8,Sq^8Sq^4, Sq^{16})}[1,17].
\]
This is the truncation above degree 17 of the displayed cyclic module.
A chart of $\Ext_\cA(M_1,\F_2)$ is shown in Figure~\ref{chart:M1}.
Since $h_0\cdot 3_9 = h_2 \cdot 3_6$ in $\Ext_\cA^{4,22}(M_1',\F_2)$,
and no $h_0$-multiple in
$\Ext_\cA^{4,22}(M_1,\F_2)$ is also an $h_2$-multiple, the map in $\Ext$
induced by the evident
epimorphism $M_1 \lra M_1'$ must send $3_9$ to $0$, and
is therefore suitable for our purpose.

\begin{figure}
{\fbox{\includegraphics[scale=0.5,trim=2cm 1cm 1cm 1cm, clip ]{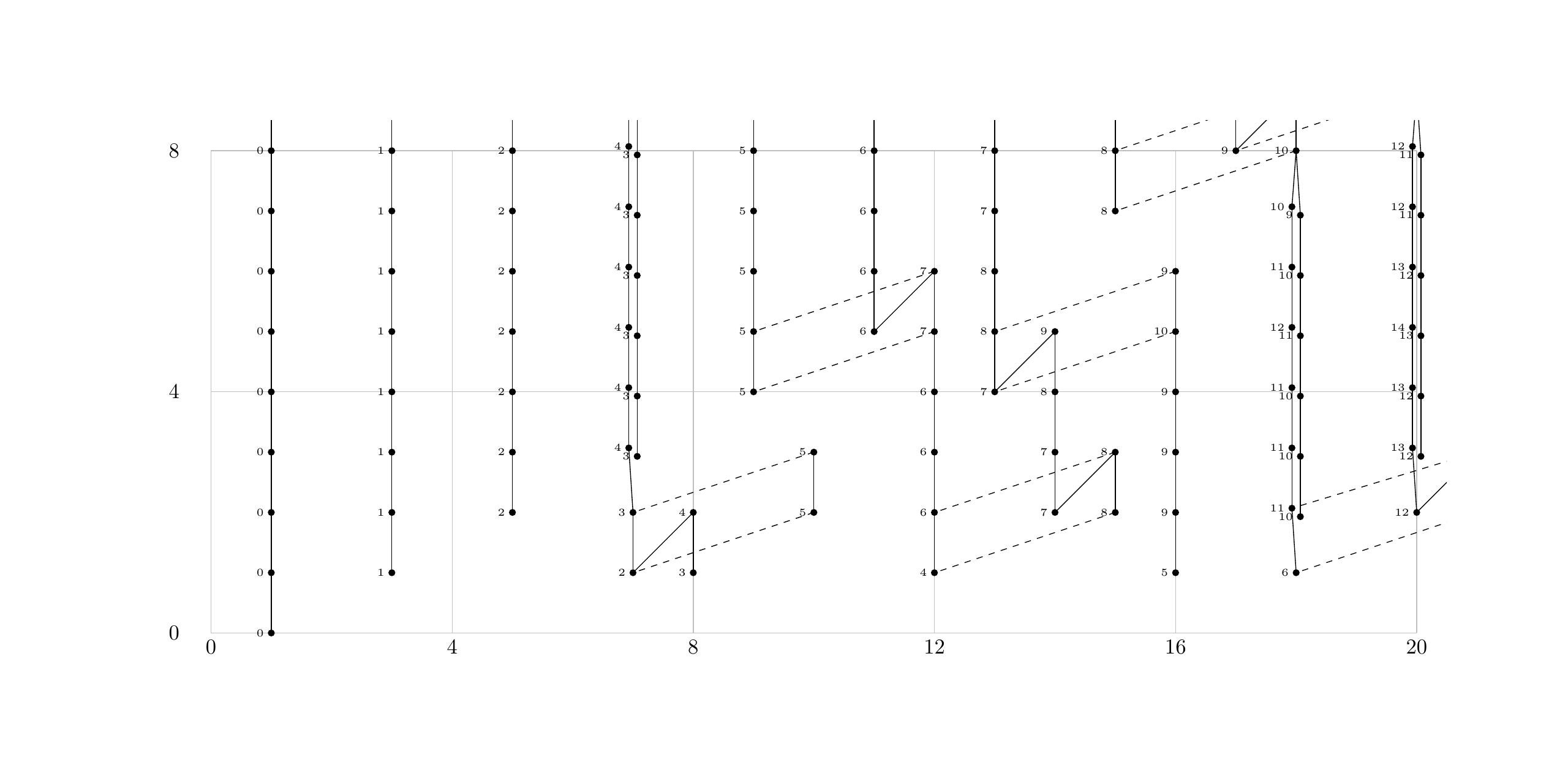}}}
\caption{$\Ext_\cA(M_1,\F_2)$}
\label{chart:M1}
\end{figure}

We then let $M_2' = \ker(M_1 \lra M_1')$.   It is $6$-dimensional over
$\F_2$ and its $\Ext$ chart is shown in Figure~\ref{chart:M2p}.
Since $3_9$ went to $0$ in $\Ext_\cA(M_1,\F_2)$, it must be in the
image of the boundary map from $\Ext_\cA^{2,21}(M_2',\F_2)$, and the
only possibility is that $\partial(2_8) = 3_9$.    Explicit calculation
of the chain map lifting the $1$-cocycle $D_1 \lra M_2'$, using {\tt ext},
where $D_1 \lra D_0 \lra M_1'$ is the start of a resolution of $M_1'$,
verifies this.

A bit of work with {\tt sage} code verifies that
\[
M_2' = \frac{\Sigma^5 \cA}{(Sq^1, Sq^{(0,1)},Sq^6,Sq^8)}[5,17].
\]
We `relax' these relations slightly to find $M_2$, in particular replacing
the relation $Sq^6$ by $Sq^{(0,0,1)}$.   The result is a bit larger than necessary:  adding the
relation $Sq^{0,0,2}$ and truncating above degree $19$ gives a module 
\[
M_2 = \frac{\Sigma^5 \cA}{(Sq^1, Sq^{(0,1)},Sq^{(0,0,1)},Sq^8,Sq^{(0,0,2)})}[5,19]
\]
which is $10$-dimensional over $\F_2$ and surjects to $M_2'$, inducing a map which
sends $2_8$ to 0 since $\Ext_\cA^{(2,21)}(M_2,\F_2) = 0$.   See Figure~\ref{chart:M2}.

\begin{figure}
{\fbox{\includegraphics[scale=0.5,trim=2cm 1cm 1cm 1cm, clip ]{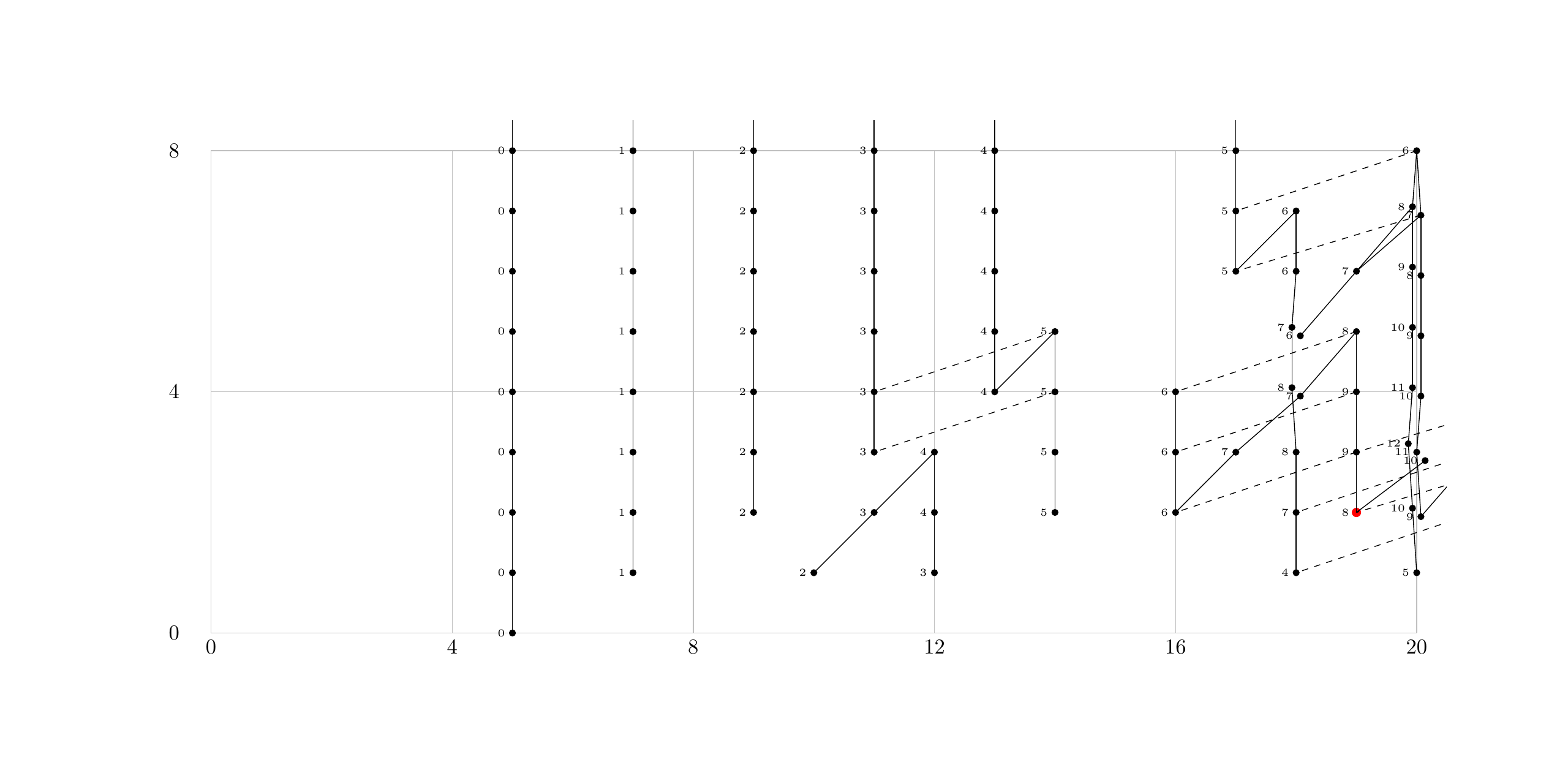}}}
\caption{$\Ext_\cA(M_2',\F_2)$}
\label{chart:M2p}
\end{figure}

\begin{figure}
{\fbox{\includegraphics[scale=0.5,trim=2cm 1cm 1cm 1cm, clip ]{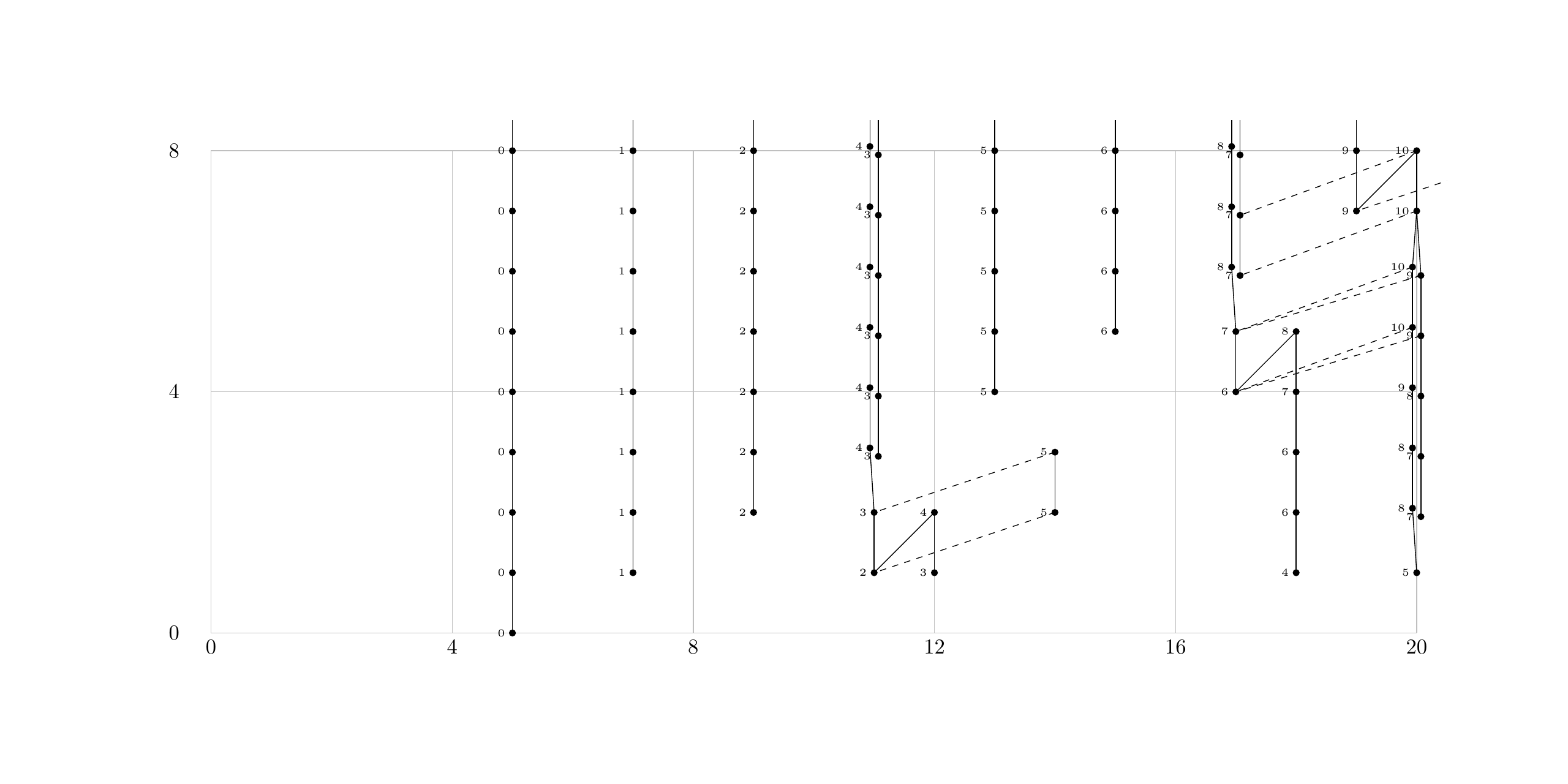}}}
\caption{$\Ext_\cA(M_2,\F_2)$}
\label{chart:M2}
\end{figure}

The kernel, $M_3'$ of $M_2 \lra M_2'$ is $4$-dimensional over $F_2$ and
has $\Ext_\cA^{1,21}(M_3',\F_2) = \langle 1_3 \rangle$.   By exactness of the long
exact sequence in $\Ext$  for $M_2' \lla M_2 \lla M_3'$, we must have
$\partial(1_3) = 2_8 \in \Ext_\cA^{2,21}(M_2',\F_2)$, and calculation of the
chain map lifting the defining $1$-cocycle defining this extension  confirms this.

\begin{figure}
{\fbox{\includegraphics[scale=0.5,trim=2cm 1cm 1cm 1cm, clip ]{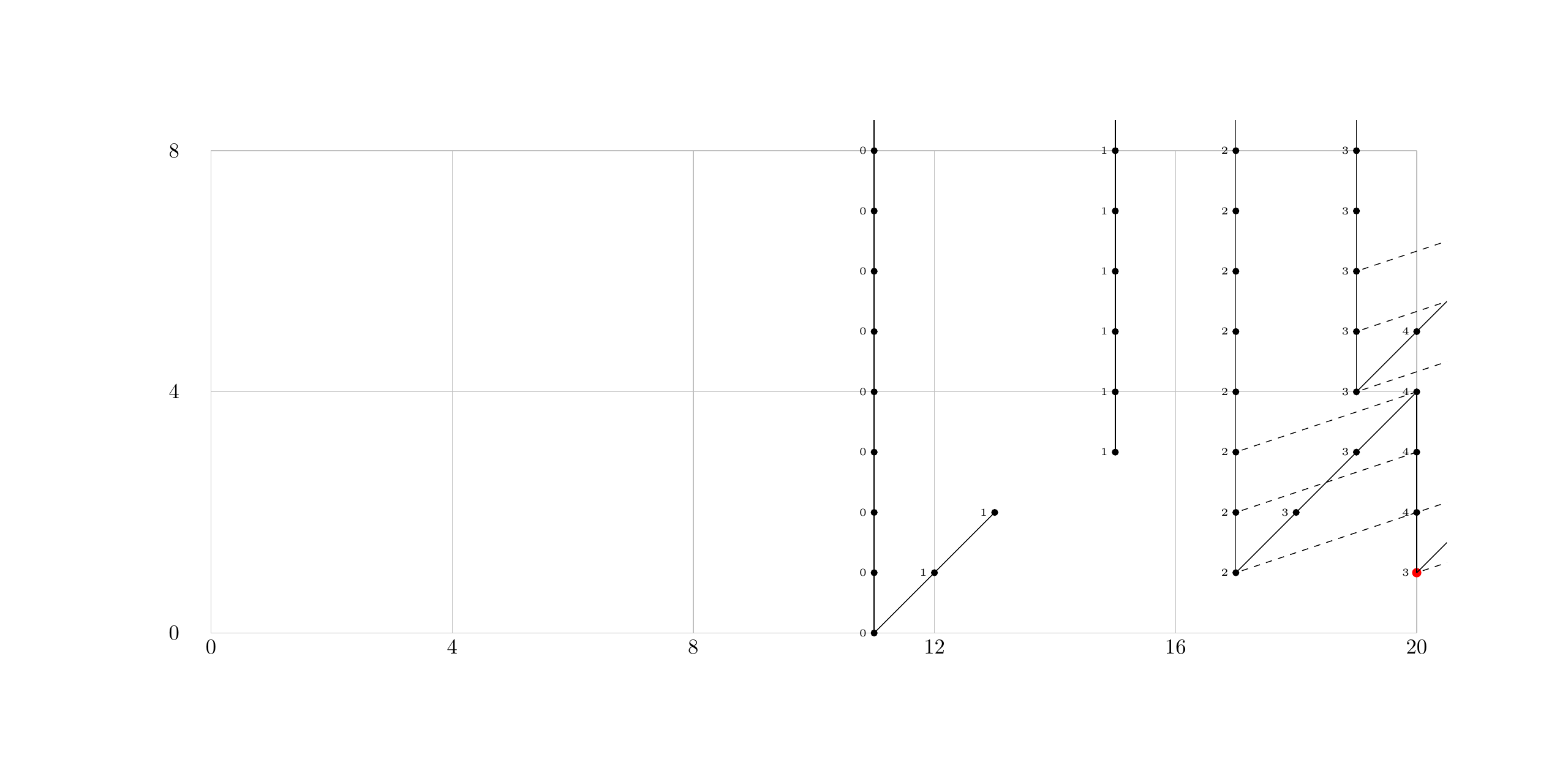}}}
\caption{$\Ext_\cA(M_3',\F_2)$}
\label{chart:M3p}
\end{figure}

We now have an extension $M_3' \lla M_3 \lla \Sigma^{21} \F_2$ defined
by $1_3 \in \Ext_\cA^{1,21}(M_3',\F_2)$.   Using the pushout of the cocycle
$1_3$ and the differential $d : D_1 \lra D_0$ in a minimal resolution of
$M_3'$, it is easy to check that
\[
M_3 = \frac{\Sigma^{11} \cA}{(Sq^1, Sq^{2},Sq^7)}[11,21].
\]
The $\Ext$ chart for $M_3$ in Figure~\ref{chart:M3} shows that 
$\Ext_\cA^{1,21}(M_3,\F_2) = 0$, so that $1_3$ must be $\partial(0_0)$,
where $0_0 \in \Ext_\cA^{0,21}(\Sigma^{21}\F_2,\F_2)$
is the cocycle $\Sigma^{21}\cA \lra \Sigma^{21}\F_2$
lifting the identity map of $\Sigma^{21}\F_2$.

\begin{figure}
{\fbox{\includegraphics[scale=0.5,trim=2cm 1cm 1cm 1cm, clip ]{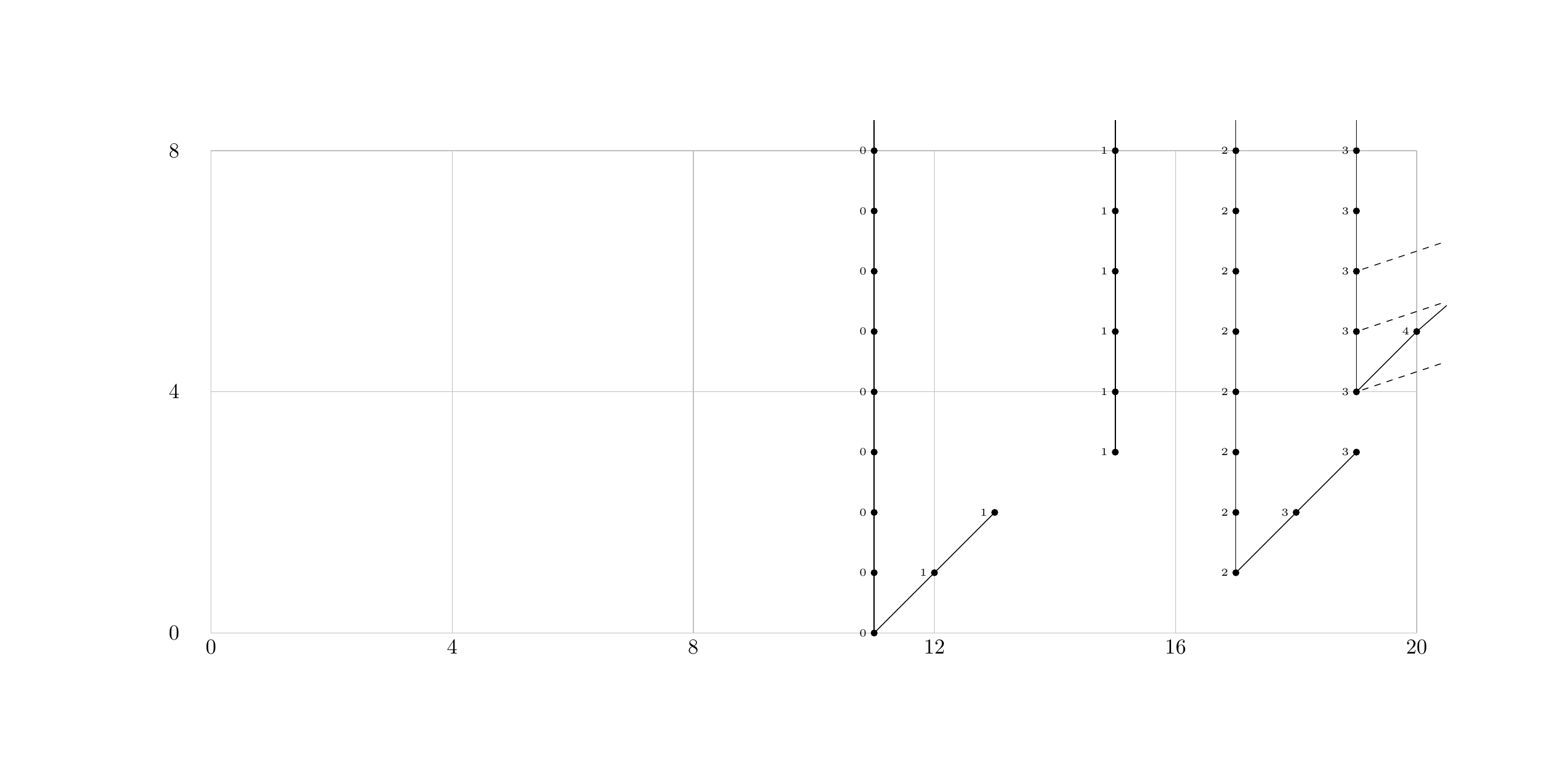}}}
\caption{$\Ext_\cA(M_3,\F_2)$}
\label{chart:M3}
\end{figure}

By construction, the extension
\[
\cE_{e_0} : 0 \lla \F_2 \lla M_0 \lla  M_{1} \lla M_2 \lla M_3 \lla \Sigma^{21} \F_2 \lla 0
\]
realizes $e_0 = 4_5$, but we also wish to put in evidence the chain map 
$e: C_* \lra \cE_{e_0}$ from the minimal
resolution of Section~\ref{sec:minimal} to $\cE_{e_0}$.
Let $k_i$ be the generator of the cyclic $\cA$-module $M_i$, $i=0,1,2,3$
or of $\Sigma^{21}\F_2$ if $i=4$.
The maps in the extension are
\begin{itemize}
\item $\partial(k_1) = Sq^1 k_0$,
\item $\partial(k_2) = Sq^4 k_1$,
\item $\partial(k_3) = Sq^6 k_2$,
\item $\partial(k_4) = Sq^2 Sq^8 k_3$.
\end{itemize}

\begin{proposition}
The extension $\cE_{e_0}$  represents the cocycle 
$e_0 = 4_5 \in \Ext_{\cA}^{4,21}(\F_2,\F_2)$.
\end{proposition}

\begin{proof}
It suffices to exhibit the following  chain map $e:C_* \lra \cE_{e_0}$
from the minimal resolution in Section~\ref{sec:minimal}
to $\cE_{e_0}$.  
The nonzero values of $e$ are
\begin{itemize}
\item[0:] $e(0_0) = k_0$
\item[1:] $e(1_0) = k_1$
\item[2:] $e(2_2) = k_2$ 
\item[\phantom{2:}] $e(2_4) = Sq^4 k_2$
\item[\phantom{2:}] $e(2_7) = Sq^8 Sq^4 k_2$
\item[3:] $e(3_3) = k_3$
\item[\phantom{3:}] $e(3_5) = Sq^6 k_3$
\item[4:] $e(4_5) = k_4 = \Sigma^{21}\iota $
\qedhere
\end{itemize}
\end{proof}

\begin{proposition}
\label{prop:sqe0}
The squaring operations on $e_0$ are
\[
Sq^*(e_0) = (8_{13}, 7_{12}, 6_{14}, 5_{17}, 4_{16}) = 
(e_0^2, m, t, x, e_1).
\]
\end{proposition}

\noindent
The proof is left to the reader.

\section{An extension for $d_0$ and the $Sq^i(d_0)$}
\label{sect:d0}

Examining the extension $\cE_{e_0}$, we find that 
$d_0 = 4_3 \in \Ext_{\cA}^{4,18}(\F_2,\F_2)$
lifts all the way to $M_3'$:
\[
\xymatrix@R=2ex{
\Ext_\cA^{4,18}(\F_2)
&
\Ext_\cA^{3,18}(M_1')
\ar_{\partial}[l]
&
\Ext_\cA^{2,18}(M_2')
\ar_{\partial}[l]
&
\Ext_\cA^{1,18}(M_3')
\ar_{\partial}[l]
\\
d_0 = 4_3
&
3_6
\ar@{|->}[l]
&
2_6
\ar@{|->}[l]
&
1_2
\ar@{|->}[l]
\\}
\]
Since we are considering a class in total degree $18$, we also truncate
the modules in $\cE_{e_0}$ above degree $18$.   This has the effect of
eliminating a single class in degree $19$ in 
$\widetilde{M}_2 = M_2[5,18]$ and
$\widetilde{M}_3' = M_3'[11,18]$.

We then have an extension 
$\widetilde{M}_3' \lla \widetilde{M}_3 \lla \Sigma^{18} \F_2$ defined
by $1_2 \in \Ext_\cA^{1,21}(\widetilde{M}_3',\F_2)$.  It is easy to check that
\[
\widetilde{M}_3 = \frac{\Sigma^{11} \cA}{(Sq^1, Sq^{2})}[11,18].
\]
(This is the module Mahowald would call $M_7$ because it is the smallest
$\cA$-module in which $Sq^7$ is nonzero.)
Calculating $\Ext_\cA(\widetilde{M}_3,\F_2)$ shows that 
$\Ext_\cA^{1,18}(\widetilde{M}_3,\F_2) = 0$, so that $1_2$ must be $\partial(0_0)$,
where $0_0 \in \Ext_\cA^{0,18}(\Sigma^{18}\F_2,\F_2)$
is the cocycle $\Sigma^{18}\cA \lra \Sigma^{18}\F_2$
lifting the identity map of $\Sigma^{18}\F_2$.
By construction, the extension
\[
\cE_{d_0} : 0 \lla \F_2 \lla M_0 \lla  M_{1} \lla \widetilde{M}_2 \lla 
   \widetilde{M}_3 \lla \Sigma^{18} \F_2 \lla 0
\]
realizes $d_0 = 4_3$, but we also wish to put in evidence the chain map 
$d: C_* \lra \cE_{d_0}$ from the minimal
resolution of Section~\ref{sec:minimal} to $\cE_{d_0}$.
Let $k_i$ be the generator of the cyclic $\cA$-module $M_i$ for $i=0,1$,
for $\widetilde{M}_i$ if $i=2,3$,
or for $\Sigma^{18}\F_2$ if $i=4$.
The maps in the extension are
\begin{itemize}
\item $\partial(k_1) = Sq^1 k_0$,
\item $\partial(k_2) = Sq^4 k_1$,
\item $\partial(k_3) = Sq^6 k_2$,
\item $\partial(k_4) = Sq^7 k_3$.
\end{itemize}

\begin{proposition}
The extension $\cE_{d_0}$  represents the cocycle 
$d_0 = 4_3 \in \Ext_{\cA}^{4,18}(\F_2,\F_2)$.
\end{proposition}

\begin{proof}
It suffices to exhibit the following  chain map $d:C_* \lra \cE_{d_0}$
from the minimal resolution in Section~\ref{sec:minimal}
to $\cE_{d_0}$.  
The nonzero values of $d$ are
\begin{itemize}
\item[0:] $d(0_0) = k_0$
\item[1:] $d(1_0) = k_1$
\item[2:] $d(2_2) = k_2$ 
\item[\phantom{2:}] $d(2_4) = Sq^4 k_2$
\item[\phantom{2:}] $d(2_7) = Sq^8 Sq^4 k_2$
\item[3:] $d(3_3) = k_3$
\item[\phantom{3:}] $d(3_5) = Sq^6 k_3$
\item[4:] $d(4_3) = k_4 = \Sigma^{18}\iota $
\end{itemize}
\end{proof}

\begin{proposition}
\label{prop:sqd0}
The squaring operations on $d_0$ are
\[
Sq^*(d_0) = (8_{7}, 0, 6_{10}, 0, 4_{13}) = 
(d_0^2, 0, r, 0, d_1).
\]
\end{proposition}

\noindent
The proof is left to the reader.

\section{An explicit minimal resolution}
\label{sec:minimal}

Here is an explicit minimal resolution of $\F_2$ over the mod 2 Steenrod 
algebra $\cA$, complete  through  homological degree $s=8$ and
internal degree $t = 44$,   computed by
the computer code {\tt ext.1.9.2}.  The same resolution will be
produced by any modern version (post 2000) of the software and is
well defined by the ordering of monomials, which is first by
degree of the $\cA$-module generator, then by a reverse lexicographic ordering
of the Milnor basis.   For example, in degree $10$, for example, we would 
have the ordering
\[
Sq^{10} > Sq^{(7,1)} > Sq^{(4,2)} > Sq^{(1,3)} > Sq^{(3,0,1)} > Sq^{(0,1,1)}
\]

The $\cA$-module generators in homological degree $s$ are called $s_0$, $s_1$,
\ldots, $s_g$, \ldots, in order of internal degree.
The ordering of elements in the terms $C_{s,t}$ described above and the usual
row reduction algorithm in linear algebra breaks the ties when two
generators appear in the same bidegree.

\subsection{Homological degree 1}

Complete through degree $t=127$.

\begin{itemize}
\addtolength{\itemsep}{2.4ex}

\item[1:] $d(1_0) = Sq^{1} 0_0$

\item[2:] $d(1_1) = Sq^{2} 0_0$

\item[4:] $d(1_2) = Sq^{4} 0_0$

\item[8:] $d(1_3) = Sq^{8} 0_0$

\item[16:] $d(1_4) = Sq^{16} 0_0$

\item[32:] $d(1_5) = Sq^{32} 0_0$

\item[64:] $d(1_6) = Sq^{64} 0_0$

\end{itemize}

\subsection{Homological degree 2}

Complete through degree $t=44$.

\begin{itemize}
\addtolength{\itemsep}{2.4ex}

\item[2:] $d(2_0) =
Sq^{{1}} 1_0
$

\item[4:] $d(2_1) =
Sq^{3} 1_0
+Sq^{2} 1_1
$

\item[5:] $d(2_2) =
Sq^{4} 1_0
+Sq^{(0,1)} 1_1
+Sq^{1} 1_2
$

\item[8:] $d(2_3) =
Sq^{7} 1_0
+Sq^{6} 1_1
+Sq^{4} 1_2
$

\item[9:] $d(2_4) =
(Sq^{8}+Sq^{(2,2)}) 1_0
+(Sq^{7}+Sq^{(4,1)}+Sq^{(0,0,1)}) 1_1
+Sq^{1} 1_3
$

\item[10:] $d(2_5) =
(Sq^{9}+Sq^{(3,2)}) 1_0
+(Sq^{8}+Sq^{(5,1)}) 1_1
+Sq^{(0,2)} 1_2
+Sq^{2} 1_3
$

\item[16:] $d(2_6) =
Sq^{15} 1_0
+Sq^{14} 1_1
+Sq^{12} 1_2
+Sq^{8} 1_3
$

\item[17:] $d(2_7) =
(Sq^{16}+Sq^{(10,2)}+Sq^{(7,3)}+Sq^{(4,4)}+Sq^{(2,0,2)}) 1_0 +\\
(Sq^{(12,1)}+Sq^{(3,4)}+Sq^{(0,5)}+Sq^{(8,0,1)}+Sq^{(0,0,0,1)}) 1_1
+Sq^{13} 1_2
+Sq^{1} 1_4
$

\item[18:] $d(2_8) =
(Sq^{17}+Sq^{(11,2)}+Sq^{(3,0,2)}) 1_0
+(Sq^{16}+Sq^{(4,4)}+Sq^{(1,5)}) 1_1
+(Sq^{14}+Sq^{(8,2)}+Sq^{(0,0,2)}) 1_2
+Sq^{2} 1_4
$

\item[20:] $d(2_9) =
(Sq^{19}+Sq^{(10,3)}+Sq^{(7,4)}) 1_0
+(Sq^{18}+Sq^{(15,1)}+Sq^{(6,4)}) 1_1
+(Sq^{16}+Sq^{(10,2)}) 1_2
+Sq^{(0,4)} 1_3
+Sq^{4} 1_4
$

\item[32:] $d(2_{10}) =
Sq^{31} 1_0
+Sq^{30} 1_1
+Sq^{28} 1_2
+Sq^{24} 1_3
+Sq^{16} 1_4
$

\item[33:] $d(2_{11}) =
(Sq^{32}+Sq^{(26,2)}+Sq^{(20,4)}+Sq^{(17,5)}+Sq^{(18,0,2)}+Sq^{(15,1,2)}+Sq^{(12,2,2)}+Sq^{(6,4,2)}+Sq^{(3,5,2)}+Sq^{(0,6,2)}+Sq^{(11,0,3)}+Sq^{(4,0,4)}+Sq^{(2,0,0,2)}) 1_0 +
(Sq^{(28,1)}+Sq^{(19,4)}+Sq^{(16,5)}+Sq^{(13,6)}+Sq^{(7,8)}+Sq^{(24,0,1)}+Sq^{(9,5,1)}+Sq^{(3,7,1)}+Sq^{(3,0,4)}+Sq^{(0,1,4)}+Sq^{(16,0,0,1)}+Sq^{(4,4,0,1)}+Sq^{(0,0,0,0,1)}) 1_1 +\\
(Sq^{29}+Sq^{(20,3)}+Sq^{(5,8)}+Sq^{(15,0,2)}+Sq^{(12,1,2)}+Sq^{(8,0,3)}) 1_2
+Sq^{1} 1_5
$

\item[34:] $d(2_{12}) =
(Sq^{(27,2)}+Sq^{(15,6)}+Sq^{(12,7)}+Sq^{(9,8)}+Sq^{(10,3,2)}+Sq^{(5,0,4)}+Sq^{(2,1,4)}+Sq^{(3,0,0,2)}) 1_0
+(Sq^{32}+Sq^{(29,1)}+Sq^{(20,4)}+Sq^{(14,6)}+Sq^{(8,8)}+Sq^{(5,9)}+Sq^{(10,5,1)}+Sq^{(7,6,1)}+Sq^{(4,7,1)}+Sq^{(1,8,1)}+Sq^{(4,0,4)}+Sq^{(1,1,4)}+Sq^{(17,0,0,1)}) 1_1
+(Sq^{(24,2)}+Sq^{(6,8)}+Sq^{(0,10)}+Sq^{(16,0,2)}+Sq^{(0,0,0,2)}) 1_2
+Sq^{26} 1_3
+Sq^{2} 1_5
$

\item[36:] $d(2_{13}) =
(Sq^{35}+Sq^{(20,5)}+Sq^{(15,2,2)}+Sq^{(4,1,4)}) 1_0
+(Sq^{34}+Sq^{(31,1)}+Sq^{(22,4)}+Sq^{(13,7)}+Sq^{(7,9)}+Sq^{(1,11)}+Sq^{(15,4,1)}+Sq^{(9,6,1)}+Sq^{(6,7,1)}+Sq^{(3,8,1)}+Sq^{(0,9,1)}+Sq^{(6,0,4)}+Sq^{(19,0,0,1)}) 1_1
+(Sq^{32}+Sq^{(23,3)}+Sq^{(8,8)}+Sq^{(2,10)}+Sq^{(15,1,2)}) 1_2
+(Sq^{28}+Sq^{(16,4)}+Sq^{(0,0,4)}) 1_3
+Sq^{4} 1_5
$

\item[40:] $d(2_{14}) =
(Sq^{39}+Sq^{(30,3)}+Sq^{(27,4)}+Sq^{(24,5)}+Sq^{(21,6)}+Sq^{(18,7)}+Sq^{(12,9)}+Sq^{(3,2,0,2)}) 1_0
+(Sq^{38}+Sq^{(23,5)}+Sq^{(20,6)}+Sq^{(14,8)}+Sq^{(10,7,1)}+Sq^{(1,3,4)}+Sq^{(3,0,5)}+Sq^{(7,0,0,0,1)}) 1_1
+(Sq^{36}+Sq^{(30,2)}+Sq^{(27,3)}+Sq^{(12,8)}+Sq^{(3,11)}) 1_2
+(Sq^{32}+Sq^{(20,4)}) 1_3
+Sq^{(0,8)} 1_4
+Sq^{8} 1_5
$

\end{itemize}

\subsection{Homological degree 3}

Complete through degree $t=44$.

\begin{itemize}
\addtolength{\itemsep}{2.4ex}

\item[3:] $d(3_0) =
Sq^{1} 2_0
$

\item[6:] $d(3_1) =
Sq^{4} 2_0
+Sq^{2} 2_1
+Sq^{1} 2_2
$

\item[10:] $d(3_2) =
(Sq^{8}+Sq^{(2,2)}) 2_0
+(Sq^{6}+Sq^{(0,2)}) 2_1
+Sq^{1} 2_4
$

\item[11:] $d(3_3) =
(Sq^{9}+Sq^{(3,2)}) 2_0
+(Sq^{(0,0,1)}) 2_1
+Sq^{6} 2_2
+(Sq^{3}+Sq^{(0,1)}) 2_3
$

\item[12:] $d(3_4) =
Sq^{10} 2_0
+(Sq^{8}+Sq^{(1,0,1)}) 2_1
+Sq^{4} 2_3
+Sq^{3} 2_4
+Sq^{2} 2_5
$

\item[17:] $d(3_5) =
(Sq^{(9,2)}+Sq^{(6,3)}) 2_0
+(Sq^{(7,2)}) 2_1
+(Sq^{(0,4)}) 2_2
+(Sq^{9}+Sq^{(0,3)}) 2_3
+(Sq^{8}+Sq^{(2,2)}) 2_4
+(Sq^{7}+Sq^{(4,1)}+Sq^{(0,0,1)}) 2_5
+Sq^{1} 2_6
$

\item[18:] $d(3_6) =
(Sq^{16}+Sq^{(10,2)}+Sq^{(7,3)}+Sq^{(4,4)}+Sq^{(1,5)}+Sq^{(2,0,2)}) 2_0
+(Sq^{14}+Sq^{(11,1)}+Sq^{(8,2)}+Sq^{(2,4)}+Sq^{(7,0,1)}+Sq^{(4,1,1)}+Sq^{(0,0,2)}) 2_1
+Sq^{1} 2_7
$

\item[20:] $d(3_7) =
(Sq^{18}+Sq^{(3,5)}) 2_0
+(Sq^{16}+Sq^{(4,4)}+Sq^{(6,1,1)}) 2_1
+(Sq^{15}+Sq^{(3,4)}) 2_2
+(Sq^{12}+Sq^{(9,1)}+Sq^{(3,3)}+Sq^{(0,4)}) 2_3
+Sq^{3} 2_7
+Sq^{2} 2_8
$

\item[21:] $d(3_8) =
(Sq^{19}+Sq^{(13,2)}+Sq^{(4,5)}+Sq^{(2,1,2)}) 2_0
+(Sq^{17}+Sq^{(11,2)}+Sq^{(5,4)}+Sq^{(2,5)}+Sq^{(10,0,1)}+Sq^{(7,1,1)}) 2_1
+(Sq^{16}+Sq^{(4,4)}) 2_2
+(Sq^{13}+Sq^{(10,1)}) 2_3
+(Sq^{12}+Sq^{(6,2)}+Sq^{(3,3)}+Sq^{(0,4)}) 2_4
+(Sq^{(8,1)}+Sq^{(4,0,1)}) 2_5
+Sq^{5} 2_6
+Sq^{4} 2_7
+(Sq^{(0,1)}) 2_8
+Sq^{1} 2_9
$

\item[22:] $d(3_9) =
(Sq^{(14,2)}+Sq^{(11,3)}+Sq^{(6,0,2)}) 2_0
+(Sq^{(15,1)}+Sq^{(11,0,1)}+Sq^{(0,1,0,1)}) 2_1
+(Sq^{(11,2)}+Sq^{(8,3)}+Sq^{(2,5)}+Sq^{(3,0,2)}) 2_2
+(Sq^{14}+Sq^{(2,4)}+Sq^{(0,0,2)}) 2_3
+(Sq^{13}+Sq^{(7,2)}) 2_4
+Sq^{12} 2_5
+(Sq^{6}+Sq^{(0,2)}) 2_6
$

\item[24:] $d(3_{10}) =
(Sq^{22}+Sq^{(13,3)}+Sq^{(10,4)}+Sq^{(5,1,2)}) 2_0
+(Sq^{20}+Sq^{(10,1,1)}) 2_1
+(Sq^{(7,4)}+Sq^{(2,1,2)}) 2_2
+(Sq^{16}+Sq^{(1,5)}+Sq^{(2,0,2)}) 2_3
+(Sq^{(3,4)}+Sq^{(0,5)}) 2_4
+(Sq^{(11,1)}+Sq^{(7,0,1)}) 2_5
+Sq^{8} 2_6
+Sq^{7} 2_7
+Sq^{6} 2_8
+Sq^{4} 2_9
$

\item[33:] $d(3_{11}) =
(Sq^{31}+Sq^{(19,4)}+Sq^{(16,5)}+Sq^{(13,6)}+Sq^{(10,7)}+Sq^{(4,9)}+Sq^{(1,10)}+Sq^{(17,0,2)}+Sq^{(14,1,2)}+Sq^{(11,2,2)}+Sq^{(2,5,2)}+Sq^{(1,3,3)}+Sq^{(3,0,4)}) 2_0
+(Sq^{29}+Sq^{(23,2)}+Sq^{(14,5)}+Sq^{(19,1,1)}+Sq^{(10,4,1)}+Sq^{(4,6,1)}+Sq^{(15,0,2)}+Sq^{(3,4,2)}+Sq^{(14,0,0,1)}+Sq^{(2,4,0,1)}) 2_1
+(Sq^{(16,4)}+Sq^{(1,9)}+Sq^{(11,1,2)}+Sq^{(8,2,2)}+Sq^{(4,1,3)}+Sq^{(1,2,3)}+Sq^{(0,0,4)}) 2_2
+(Sq^{25}+Sq^{(16,3)}+Sq^{(13,4)}+Sq^{(7,6)}+Sq^{(1,8)}+Sq^{(12,2,1)}+Sq^{(8,1,2)}) 2_3
+(Sq^{24}+Sq^{(18,2)}+Sq^{(12,4)}+Sq^{(9,5)}+Sq^{(3,7)}) 2_4
+(Sq^{23}+Sq^{(20,1)}+Sq^{(5,6)}+Sq^{(16,0,1)}+Sq^{(1,5,1)}) 2_5
+(Sq^{(11,2)}+Sq^{(2,5)}) 2_6
+(Sq^{16}+Sq^{(10,2)}+Sq^{(7,3)}+Sq^{(4,4)}+Sq^{(2,0,2)}) 2_7
+(Sq^{(12,1)}+Sq^{(3,4)}+Sq^{(0,5)}+Sq^{(8,0,1)}+Sq^{(0,0,0,1)}) 2_8
+Sq^{13} 2_9
+Sq^{1} 2_{10}
$

\item[34:] $d(3_{12}) =
(Sq^{(26,2)}+Sq^{(8,8)}+Sq^{(5,9)}+Sq^{(2,10)}+Sq^{(18,0,2)}+Sq^{(15,1,2)}+Sq^{(12,2,2)}+Sq^{(11,0,3)}+Sq^{(2,0,0,2)}) 2_0
+(Sq^{(18,4)}+Sq^{(12,6)}+Sq^{(9,7)}+Sq^{(6,8)}+Sq^{(0,10)}+Sq^{(23,0,1)}+Sq^{(20,1,1)}+Sq^{(8,5,1)}+Sq^{(5,6,1)}+Sq^{(7,3,2)}+Sq^{(4,4,2)}+Sq^{(2,0,4)}+Sq^{(12,1,0,1)}+Sq^{(0,0,0,2)}) 2_1
+(Sq^{(17,4)}+Sq^{(8,7)}+Sq^{(2,9)}+Sq^{(15,0,2)}+Sq^{(12,1,2)}+Sq^{(3,4,2)}+Sq^{(0,5,2)}+Sq^{(5,1,3)}+Sq^{(2,2,3)}+Sq^{(1,0,4)}) 2_2
+(Sq^{(23,1)}+Sq^{(17,3)}+Sq^{(14,4)}+Sq^{(11,5)}+Sq^{(8,6)}+Sq^{(7,4,1)}+Sq^{(12,0,2)}+Sq^{(9,1,2)}+Sq^{(0,4,2)}) 2_3
+(Sq^{25}+Sq^{(19,2)}+Sq^{(10,5)}+Sq^{(7,6)}+Sq^{(1,8)}) 2_4
+(Sq^{(21,1)}+Sq^{(0,8)}+Sq^{(17,0,1)}+Sq^{(5,4,1)}+Sq^{(2,5,1)}) 2_5
+(Sq^{18}+Sq^{(12,2)}+Sq^{(0,6)}) 2_6
+(Sq^{17}+Sq^{(11,2)}+Sq^{(3,0,2)}) 2_7
+(Sq^{16}+Sq^{(4,4)}+Sq^{(1,5)}) 2_8
+(Sq^{14}+Sq^{(8,2)}+Sq^{(0,0,2)}) 2_9
+Sq^{2} 2_{10}
$

\item[34:] $d(3_{13}) =
(Sq^{32}+Sq^{(26,2)}+Sq^{(20,4)}+Sq^{(17,5)}+Sq^{(11,7)}+Sq^{(5,9)}+Sq^{(2,10)}+Sq^{(18,0,2)}+Sq^{(15,1,2)}+Sq^{(9,3,2)}+Sq^{(3,5,2)}+Sq^{(0,6,2)}+Sq^{(11,0,3)}+Sq^{(8,1,3)}+Sq^{(5,2,3)}+Sq^{(2,3,3)}+Sq^{(4,0,4)}+Sq^{(1,1,4)}) 2_0
+(Sq^{30}+Sq^{(24,2)}+Sq^{(18,4)}+Sq^{(15,5)}+Sq^{(9,7)}+Sq^{(6,8)}+Sq^{(3,9)}+Sq^{(0,10)}+Sq^{(23,0,1)}+Sq^{(20,1,1)}+Sq^{(8,5,1)}+Sq^{(2,7,1)}+Sq^{(16,0,2)}+Sq^{(10,2,2)}+Sq^{(7,3,2)}+Sq^{(12,1,0,1)}+Sq^{(8,0,1,1)}) 2_1
+(Sq^{(14,5)}+Sq^{(5,8)}+Sq^{(2,9)}+Sq^{(15,0,2)}+Sq^{(9,2,2)}+Sq^{(6,3,2)}+Sq^{(3,4,2)}+Sq^{(2,2,3)}+Sq^{(1,0,4)}) 2_2
+(Sq^{(17,3)}+Sq^{(14,4)}+Sq^{(13,2,1)}+Sq^{(7,4,1)}+Sq^{(4,5,1)}+Sq^{(12,0,2)}+Sq^{(9,1,2)}+Sq^{(0,4,2)}) 2_3
+Sq^{1} 2_{11}
$

\item[36:] $d(3_{14}) =
(Sq^{(22,4)}+Sq^{(13,7)}+Sq^{(10,8)}+Sq^{(1,11)}+Sq^{(14,2,2)}+Sq^{(11,3,2)}+Sq^{(10,1,3)}+Sq^{(4,3,3)}+Sq^{(6,0,4)}+Sq^{(3,1,4)}) 2_0
+(Sq^{32}+Sq^{(20,4)}+Sq^{(11,7)}+Sq^{(8,8)}+Sq^{(5,9)}+Sq^{(2,10)}+Sq^{(22,1,1)}+Sq^{(13,4,1)}+Sq^{(10,5,1)}+Sq^{(4,7,1)}+Sq^{(1,8,1)}+Sq^{(12,2,2)}+Sq^{(6,4,2)}+Sq^{(0,6,2)}+Sq^{(4,0,4)}+Sq^{(17,0,0,1)}+Sq^{(2,5,0,1)}) 2_1
+(Sq^{(22,3)}+Sq^{(16,5)}+Sq^{(13,6)}+Sq^{(7,8)}+Sq^{(14,1,2)}+Sq^{(8,3,2)}+Sq^{(10,0,3)}+Sq^{(7,1,3)}+Sq^{(4,2,3)}+Sq^{(3,0,4)}) 2_2
+(Sq^{28}+Sq^{(25,1)}+Sq^{(22,2)}+Sq^{(16,4)}+Sq^{(4,8)}+Sq^{(1,9)}+Sq^{(15,2,1)}+Sq^{(9,4,1)}+Sq^{(6,5,1)}+Sq^{(14,0,2)}+Sq^{(8,2,2)}+Sq^{(0,0,4)}) 2_3
+Sq^{27} 2_4
+Sq^{3} 2_{11}
+Sq^{2} 2_{12}
$

\item[37:] $d(3_{15}) =
(Sq^{(29,2)}+Sq^{(23,4)}+Sq^{(20,5)}+Sq^{(17,6)}+Sq^{(11,8)}+Sq^{(5,10)}+Sq^{(18,1,2)}+Sq^{(15,2,2)}+Sq^{(9,4,2)}+Sq^{(6,5,2)}+Sq^{(3,6,2)}+Sq^{(14,0,3)}+Sq^{(5,3,3)}+Sq^{(7,0,4)}+Sq^{(4,1,4)}+Sq^{(2,1,0,2)}) 2_0
+(Sq^{33}+Sq^{(30,1)}+Sq^{(21,4)}+Sq^{(18,5)}+Sq^{(15,6)}+Sq^{(12,7)}+Sq^{(9,8)}+Sq^{(0,11)}+Sq^{(14,4,1)}+Sq^{(8,6,1)}+Sq^{(2,8,1)}+Sq^{(10,3,2)}+Sq^{(1,6,2)}+Sq^{(2,1,4)}+Sq^{(0,6,0,1)}+Sq^{(11,0,1,1)}+Sq^{(3,0,0,2)}) 2_1
+(Sq^{32}+Sq^{(23,3)}+Sq^{(17,5)}+Sq^{(14,6)}+Sq^{(8,8)}+Sq^{(5,9)}+Sq^{(15,1,2)}+Sq^{(12,2,2)}+Sq^{(9,3,2)}+Sq^{(3,5,2)}+Sq^{(0,6,2)}+Sq^{(11,0,3)}+Sq^{(8,1,3)}+Sq^{(5,2,3)}+Sq^{(4,0,4)}) 2_2
+(Sq^{(23,2)}+Sq^{(14,5)}+Sq^{(2,9)}+Sq^{(10,4,1)}+Sq^{(7,5,1)}+Sq^{(15,0,2)}+Sq^{(12,1,2)}+Sq^{(1,0,4)}) 2_3
+(Sq^{(22,2)}+Sq^{(19,3)}+Sq^{(16,4)}+Sq^{(10,6)}+Sq^{(7,7)}+Sq^{(4,8)}+Sq^{(1,9)}+Sq^{(14,0,2)}+Sq^{(11,1,2)}+Sq^{(8,2,2)}+Sq^{(2,4,2)}+Sq^{(1,2,3)}+Sq^{(0,0,4)}) 2_4
+(Sq^{(24,1)}+Sq^{(15,4)}+Sq^{(12,5)}+Sq^{(9,6)}+Sq^{(6,7)}+Sq^{(3,8)}+Sq^{(0,9)}+Sq^{(20,0,1)}+Sq^{(8,4,1)}+Sq^{(2,6,1)}+Sq^{(12,0,0,1)}+Sq^{(0,4,0,1)}) 2_5
+Sq^{21} 2_6
+Sq^{4} 2_{11}
+(Sq^{(0,1)}) 2_{12}
+Sq^{1} 2_{13}
$

\item[40:] $d(3_{16}) =
(Sq^{(29,3)}+Sq^{(26,4)}+Sq^{(20,6)}+Sq^{(17,7)}+Sq^{(5,11)}+Sq^{(2,12)}+Sq^{(6,6,2)}+Sq^{(3,7,2)}+Sq^{(8,3,3)}+Sq^{(2,5,3)}+Sq^{(7,1,4)}+Sq^{(4,2,4)}+Sq^{(1,3,4)}+Sq^{(3,0,5)}+Sq^{(0,1,5)}) 2_0
+(Sq^{36}+Sq^{(30,2)}+Sq^{(15,7)}+Sq^{(6,10)}+Sq^{(8,7,1)}+Sq^{(2,9,1)}+Sq^{(4,6,2)}+Sq^{(21,0,0,1)}+Sq^{(6,5,0,1)}+Sq^{(14,0,1,1)}+Sq^{(2,4,1,1)}) 2_1
+(Sq^{(26,3)}+Sq^{(23,4)}+Sq^{(14,7)}+Sq^{(3,6,2)}+Sq^{(0,7,2)}+Sq^{(14,0,3)}+Sq^{(8,2,3)}+Sq^{(2,4,3)}+Sq^{(7,0,4)}+Sq^{(4,1,4)}) 2_2
+(Sq^{32}+Sq^{(29,1)}+Sq^{(8,8)}+Sq^{(10,5,1)}+Sq^{(12,2,2)}+Sq^{(6,4,2)}+Sq^{(11,0,3)}+Sq^{(3,0,2,1)}) 2_3
+(Sq^{(22,3)}+Sq^{(19,4)}+Sq^{(16,5)}+Sq^{(10,7)}+Sq^{(4,9)}+Sq^{(11,2,2)}) 2_4
+(Sq^{30}+Sq^{(9,7)}+Sq^{(6,8)}+Sq^{(23,0,1)}+Sq^{(11,4,1)}+Sq^{(5,6,1)}+Sq^{(2,7,1)}+Sq^{(15,0,0,1)}) 2_5
+(Sq^{24}+Sq^{(18,2)}+Sq^{(6,6)}+Sq^{(0,8)}) 2_6
+Sq^{7} 2_{11}
+Sq^{6} 2_{12}
+Sq^{4} 2_{13}
$

\item[41:] $d(3_{17}) =
(Sq^{(30,3)}+Sq^{(21,6)}+Sq^{(18,7)}+Sq^{(12,9)}+Sq^{(9,10)}+Sq^{(25,0,2)}+Sq^{(19,2,2)}+Sq^{(7,6,2)}+Sq^{(4,7,2)}+Sq^{(1,8,2)}+Sq^{(15,1,3)}+Sq^{(0,6,3)}+Sq^{(5,2,4)}+Sq^{(2,3,4)}+Sq^{(4,0,5)}+Sq^{(9,0,0,2)}+Sq^{(3,2,0,2)}+Sq^{(2,0,1,2)}) 2_0
+(Sq^{(34,1)}+Sq^{(19,6)}+Sq^{(16,7)}+Sq^{(13,8)}+Sq^{(27,1,1)}+Sq^{(18,4,1)}+Sq^{(9,7,1)}+Sq^{(0,10,1)}+Sq^{(14,3,2)}+Sq^{(11,4,2)}+Sq^{(8,5,2)}+Sq^{(2,7,2)}+Sq^{(6,1,4)}+Sq^{(0,3,4)}+Sq^{(19,1,0,1)}+Sq^{(10,4,0,1)}+Sq^{(7,5,0,1)}+Sq^{(15,0,1,1)}) 2_1
+(Sq^{(30,2)}+Sq^{(0,12)}+Sq^{(19,1,2)}+Sq^{(16,2,2)}+Sq^{(13,3,2)}+Sq^{(4,6,2)}+Sq^{(1,7,2)}+Sq^{(15,0,3)}+Sq^{(12,1,3)}+Sq^{(9,2,3)}+Sq^{(0,2,0,2)}) 2_2
+(Sq^{33}+Sq^{(30,1)}+Sq^{(21,4)}+Sq^{(18,5)}+Sq^{(3,10)}+Sq^{(0,11)}+Sq^{(11,5,1)}+Sq^{(13,2,2)}+Sq^{(10,3,2)}+Sq^{(3,0,0,2)}) 2_3
+(Sq^{32}+Sq^{(23,3)}+Sq^{(20,4)}+Sq^{(14,6)}+Sq^{(11,7)}+Sq^{(8,8)}+Sq^{(5,9)}+Sq^{(18,0,2)}+Sq^{(15,1,2)}+Sq^{(12,2,2)}+Sq^{(6,4,2)}+Sq^{(3,5,2)}+Sq^{(0,6,2)}+Sq^{(5,2,3)}+Sq^{(2,3,3)}+Sq^{(4,0,4)}+Sq^{(2,0,0,2)}) 2_4
+(Sq^{(19,4)}+Sq^{(16,5)}+Sq^{(13,6)}+Sq^{(10,7)}+Sq^{(4,9)}+Sq^{(1,10)}+Sq^{(24,0,1)}+Sq^{(12,4,1)}+Sq^{(6,6,1)}+Sq^{(3,7,1)}+Sq^{(3,0,4)}+Sq^{(0,1,4)}+Sq^{(16,0,0,1)}+Sq^{(4,4,0,1)}+Sq^{(0,0,0,0,1)}) 2_5
+(Sq^{(19,2)}+Sq^{(16,3)}+Sq^{(7,6)}+Sq^{(1,8)}) 2_6
+(Sq^{(6,6)}+Sq^{(0,8)}+Sq^{(3,0,3)}) 2_7
+(Sq^{(4,4,1)}+Sq^{(1,5,1)}) 2_8
+(Sq^{(15,2)}+Sq^{(12,3)}+Sq^{(0,0,3)}) 2_9
+(Sq^{8}+Sq^{(2,2)}) 2_{11}
+(Sq^{7}+Sq^{(4,1)}+Sq^{(0,0,1)}) 2_{12}
+Sq^{1} 2_{14}
$

\item[42:] $d(3_{18}) =
(Sq^{40}+Sq^{(34,2)}+Sq^{(31,3)}+Sq^{(25,5)}+Sq^{(22,6)}+Sq^{(19,7)}+Sq^{(16,8)}+Sq^{(10,10)}+Sq^{(4,12)}+Sq^{(1,13)}+Sq^{(11,5,2)}+Sq^{(8,6,2)}+Sq^{(19,0,3)}+Sq^{(13,2,3)}+Sq^{(10,3,3)}+Sq^{(4,5,3)}+Sq^{(1,6,3)}+Sq^{(3,3,4)}+Sq^{(0,4,4)}+Sq^{(10,0,0,2)}+Sq^{(1,3,0,2)}+Sq^{(0,1,1,2)}) 2_0
+(Sq^{38}+Sq^{(35,1)}+Sq^{(23,5)}+Sq^{(20,6)}+Sq^{(17,7)}+Sq^{(14,8)}+Sq^{(5,11)}+Sq^{(28,1,1)}+Sq^{(13,6,1)}+Sq^{(10,7,1)}+Sq^{(24,0,2)}+Sq^{(12,4,2)}+Sq^{(10,0,4)}+Sq^{(7,1,4)}+Sq^{(1,3,4)}+Sq^{(3,0,5)}+Sq^{(23,0,0,1)}+Sq^{(8,5,0,1)}+Sq^{(5,6,0,1)}+Sq^{(16,0,1,1)}+Sq^{(1,5,1,1)}+Sq^{(8,0,0,2)}+Sq^{(2,2,0,2)}+Sq^{(4,1,0,0,1)}) 2_1
+(Sq^{(28,3)}+Sq^{(25,4)}+Sq^{(22,5)}+Sq^{(19,6)}+Sq^{(10,9)}+Sq^{(17,2,2)}+Sq^{(11,4,2)}+Sq^{(10,2,3)}+Sq^{(4,4,3)}+Sq^{(1,5,3)}+Sq^{(3,2,4)}+Sq^{(2,0,5)}+Sq^{(7,0,0,2)}+Sq^{(1,2,0,2)}+Sq^{(0,0,1,2)}) 2_2
+(Sq^{34}+Sq^{(28,2)}+Sq^{(22,4)}+Sq^{(19,5)}+Sq^{(16,6)}+Sq^{(10,8)}+Sq^{(1,11)}+Sq^{(15,4,1)}+Sq^{(3,8,1)}+Sq^{(0,9,1)}+Sq^{(14,2,2)}+Sq^{(11,3,2)}+Sq^{(8,4,2)}+Sq^{(0,2,4)}+Sq^{(4,0,0,2)}) 2_3
+(Sq^{33}+Sq^{(15,6)}+Sq^{(9,8)}+Sq^{(19,0,2)}+Sq^{(13,2,2)}) 2_4
+(Sq^{32}+Sq^{(29,1)}+Sq^{(17,5)}+Sq^{(11,7)}+Sq^{(8,8)}+Sq^{(25,0,1)}+Sq^{(13,4,1)}+Sq^{(17,0,0,1)}) 2_5
+(Sq^{26}+Sq^{(20,2)}) 2_6
+(Sq^{25}+Sq^{(13,4)}+Sq^{(4,7)}+Sq^{(11,0,2)}+Sq^{(5,2,2)}+Sq^{(2,3,2)}) 2_7
+(Sq^{24}+Sq^{(12,4)}+Sq^{(9,5)}+Sq^{(6,6)}+Sq^{(3,7)}+Sq^{(0,8)}+Sq^{(17,0,1)}+Sq^{(5,4,1)}+Sq^{(2,5,1)}) 2_8
+(Sq^{(16,2)}+Sq^{(13,3)}+Sq^{(8,0,2)}) 2_9
+Sq^{10} 2_{10}
+(Sq^{9}+Sq^{(3,2)}) 2_{11}
+(Sq^{8}+Sq^{(5,1)}) 2_{12}
+(Sq^{(0,2)}) 2_{13}
+Sq^{2} 2_{14}
$

\item[44:] $d(3_{19}) =
(Sq^{(24,6)}+Sq^{(15,9)}+Sq^{(12,10)}+Sq^{(9,11)}+Sq^{(0,14)}+Sq^{(25,1,2)}+Sq^{(13,5,2)}+Sq^{(10,6,2)}+Sq^{(21,0,3)}+Sq^{(18,1,3)}+Sq^{(9,4,3)}+Sq^{(0,7,3)}+Sq^{(2,4,4)}+Sq^{(1,2,5)}+Sq^{(9,1,0,2)}+Sq^{(6,2,0,2)}+Sq^{(3,3,0,2)}+Sq^{(2,1,1,2)}) 2_0
+(Sq^{(28,4)}+Sq^{(1,13)}+Sq^{(33,0,1)}+Sq^{(30,1,1)}+Sq^{(18,5,1)}+Sq^{(6,9,1)}+Sq^{(3,10,1)}+Sq^{(0,11,1)}+Sq^{(20,2,2)}+Sq^{(17,3,2)}+Sq^{(14,4,2)}+Sq^{(8,6,2)}+Sq^{(5,7,2)}+Sq^{(2,8,2)}+Sq^{(12,0,4)}+Sq^{(9,1,4)}+Sq^{(4,7,0,1)}+Sq^{(18,0,1,1)}+Sq^{(6,4,1,1)}+Sq^{(10,0,0,2)}+Sq^{(4,2,0,2)}+Sq^{(9,0,0,0,1)}) 2_1
+(Sq^{39}+Sq^{(30,3)}+Sq^{(27,4)}+Sq^{(21,6)}+Sq^{(18,7)}+Sq^{(15,8)}+Sq^{(12,9)}+Sq^{(6,11)}+Sq^{(4,7,2)}+Sq^{(12,2,3)}+Sq^{(9,3,3)}+Sq^{(0,6,3)}+Sq^{(6,1,0,2)}+Sq^{(3,2,0,2)}) 2_2
+(Sq^{(33,1)}+Sq^{(30,2)}+Sq^{(21,5)}+Sq^{(3,11)}+Sq^{(23,2,1)}+Sq^{(17,4,1)}+Sq^{(2,9,1)}+Sq^{(22,0,2)}+Sq^{(13,3,2)}+Sq^{(15,0,3)}+Sq^{(9,2,3)}+Sq^{(5,1,4)}+Sq^{(0,2,0,2)}) 2_3
+(Sq^{(23,4)}+Sq^{(17,6)}+Sq^{(11,8)}+Sq^{(8,9)}+Sq^{(5,10)}+Sq^{(2,11)}+Sq^{(15,2,2)}+Sq^{(4,1,4)}) 2_4
+(Sq^{(22,4)}+Sq^{(16,6)}+Sq^{(13,7)}+Sq^{(7,9)}+Sq^{(4,10)}+Sq^{(27,0,1)}+Sq^{(15,4,1)}+Sq^{(12,5,1)}+Sq^{(6,7,1)}+Sq^{(6,0,4)}+Sq^{(19,0,0,1)}) 2_5
+(Sq^{28}+Sq^{(4,8)}+Sq^{(0,0,4)}) 2_6
+(Sq^{27}+Sq^{(12,5)}) 2_7
+(Sq^{26}+Sq^{(23,1)}+Sq^{(14,4)}) 2_8
+Sq^{24} 2_9
+(Sq^{12}+Sq^{(0,4)}) 2_{10}
$

\end{itemize}

\subsection{Homological degree 4}

Complete through degree $t=44$.

\begin{itemize}
\addtolength{\itemsep}{2.4ex}

\item[4:] $d(4_{0}) =
Sq^{1}  3_{0}
$

\item[11:] $d(4_{1}) =
Sq^{8}  3_{0}
+(Sq^{5}+Sq^{(2,1)})  3_{1}
+Sq^{1}  3_{2}
$

\item[13:] $d(4_{2}) =
(Sq^{10}+Sq^{(4,2)})  3_{0}
+(Sq^{7}+Sq^{(1,2)}+Sq^{(0,0,1)})  3_{1}
+Sq^{2}  3_{3}
$

\item[18:] $d(4_{3}) =
(Sq^{15}+Sq^{(9,2)}+Sq^{(6,3)}+Sq^{(0,5)})  3_{0}
+(Sq^{(9,1)}+Sq^{(6,2)})  3_{1}
+(Sq^{7}+Sq^{(4,1)}+Sq^{(0,0,1)})  3_{3}
+Sq^{(3,1)}  3_{4}
$

\item[19:] $d(4_{4}) =
(Sq^{16}+Sq^{(7,3)}+Sq^{(4,4)}+Sq^{(1,5)})  3_{0}
+(Sq^{13}+Sq^{(10,1)}+Sq^{(7,2)}+Sq^{(4,3)}+Sq^{(1,4)}+Sq^{(6,0,1)}+Sq^{(0,2,1)})  3_{1}
+Sq^{1}  3_{6}
$

\item[21:] $d(4_{5}) =
(Sq^{(9,3)}+Sq^{(4,0,2)})  3_{0}
+(Sq^{(12,1)}+Sq^{(1,0,2)}+Sq^{(0,0,0,1)})  3_{1}
+Sq^{11}  3_{2}
+(Sq^{10}+Sq^{(7,1)})  3_{3}
+(Sq^{(6,1)}+Sq^{(3,2)}+Sq^{(0,3)}+Sq^{(2,0,1)})  3_{4}
+Sq^{(1,1)}  3_{5}
$

\item[22:] $d(4_{6}) =
(Sq^{(10,3)}+Sq^{(4,5)}+Sq^{(1,6)}+Sq^{(5,0,2)})  3_{0}
+(Sq^{(7,3)}+Sq^{(4,4)}+Sq^{(9,0,1)}+Sq^{(6,1,1)}+Sq^{(3,2,1)}+Sq^{(1,0,0,1)})  3_{1}
+Sq^{12}  3_{2}
+(Sq^{11}+Sq^{(4,0,1)})  3_{3}
+(Sq^{(7,1)}+Sq^{(0,1,1)})  3_{4}
+(Sq^{5}+Sq^{(2,1)})  3_{5}
$

\item[22:] $d(4_{7}) =
(Sq^{19}+Sq^{(13,2)}+Sq^{(10,3)}+Sq^{(7,4)}+Sq^{(5,0,2)}+Sq^{(2,1,2)})  3_{0}
+(Sq^{16}+Sq^{(4,4)}+Sq^{(9,0,1)}+Sq^{(1,0,0,1)})  3_{1}
+(Sq^{12}+Sq^{(6,2)}+Sq^{(3,3)}+Sq^{(0,4)})  3_{2}
+Sq^{(4,0,1)}  3_{3}
+(Sq^{10}+Sq^{(4,2)}+Sq^{(1,3)}+Sq^{(0,1,1)})  3_{4}
+Sq^{4}  3_{6}
+Sq^{2}  3_{7}
+Sq^{1}  3_{8}
$

\item[24:] $d(4_{8}) =
(Sq^{(3,6)}+Sq^{(0,7)}+Sq^{(4,1,2)})  3_{0}
+(Sq^{(3,5)}+Sq^{(0,6)}+Sq^{(11,0,1)}+Sq^{(8,1,1)}+Sq^{(5,2,1)}+Sq^{(1,1,2)}+Sq^{(0,1,0,1)})  3_{1}
+(Sq^{(8,2)}+Sq^{(5,3)})  3_{2}
+Sq^{(10,1)}  3_{3}
+(Sq^{(6,2)}+Sq^{(5,0,1)})  3_{4}
+(Sq^{(1,2)}+Sq^{(0,0,1)})  3_{5}
$

\item[26:] $d(4_{9}) =
(Sq^{(14,3)}+Sq^{(8,5)}+Sq^{(5,6)}+Sq^{(2,7)}+Sq^{(9,0,2)}+Sq^{(2,0,3)})  3_{0}
+(Sq^{(8,4)}+Sq^{(5,5)}+Sq^{(2,6)}+Sq^{(13,0,1)}+Sq^{(7,2,1)}+Sq^{(1,4,1)}+Sq^{(3,1,2)}+Sq^{(5,0,0,1)})  3_{1}
+(Sq^{(10,2)}+Sq^{(4,4)}+Sq^{(1,5)}+Sq^{(2,0,2)})  3_{2}
+(Sq^{15}+Sq^{(12,1)}+Sq^{(3,4)})  3_{3}
+(Sq^{(11,1)}+Sq^{(8,2)}+Sq^{(2,4)}+Sq^{(0,0,2)})  3_{4}
+(Sq^{9}+Sq^{(6,1)}+Sq^{(0,3)})  3_{5}
+(Sq^{4}+Sq^{(1,1)})  3_{9}
$

\item[27:] $d(4_{10}) =
(Sq^{(18,2)}+Sq^{(9,5)}+Sq^{(0,8)}+Sq^{(10,0,2)}+Sq^{(7,1,2)}+Sq^{(4,2,2)}+Sq^{(1,3,2)})  3_{0}
+(Sq^{(15,2)}+Sq^{(9,4)}+Sq^{(0,7)}+Sq^{(8,2,1)}+Sq^{(2,4,1)}+Sq^{(7,0,2)}+Sq^{(4,1,2)}+Sq^{(1,2,2)}+Sq^{(0,0,3)}+Sq^{(6,0,0,1)}+Sq^{(0,2,0,1)})  3_{1}
+(Sq^{(11,2)}+Sq^{(5,4)})  3_{2}
+(Sq^{16}+Sq^{(4,4)})  3_{3}
+(Sq^{15}+Sq^{(9,2)}+Sq^{(6,3)}+Sq^{(3,4)}+Sq^{(5,1,1)}+Sq^{(1,0,2)}+Sq^{(0,0,0,1)})  3_{4}
+Sq^{10}  3_{5}
+(Sq^{9}+Sq^{(3,2)})  3_{6}
+Sq^{(0,0,1)}  3_{7}
+Sq^{6}  3_{8}
+Sq^{(2,1)}  3_{9}
+(Sq^{3}+Sq^{(0,1)})  3_{10}
$

\item[34:] $d(4_{11}) =
(Sq^{31}+Sq^{(25,2)}+Sq^{(13,6)}+Sq^{(10,7)}+Sq^{(17,0,2)}+Sq^{(14,1,2)}+Sq^{(11,2,2)}+Sq^{(8,3,2)}+Sq^{(4,2,3)}+Sq^{(1,3,3)}+Sq^{(3,0,4)}+Sq^{(0,1,4)})  3_{0}
+(Sq^{(10,6)}+Sq^{(1,9)}+Sq^{(9,4,1)}+Sq^{(3,6,1)}+Sq^{(14,0,2)}+Sq^{(11,1,2)}+Sq^{(8,2,2)}+Sq^{(5,3,2)}+Sq^{(7,0,3)}+Sq^{(1,2,3)}+Sq^{(13,0,0,1)}+Sq^{(10,1,0,1)}+Sq^{(1,4,0,1)})  3_{1}
+(Sq^{(12,4)}+Sq^{(9,5)}+Sq^{(6,6)}+Sq^{(0,8)}+Sq^{(10,0,2)}+Sq^{(4,2,2)})  3_{2}
+(Sq^{(11,4)}+Sq^{(8,5)}+Sq^{(4,4,1)})  3_{3}
+(Sq^{(4,6)}+Sq^{(1,7)}+Sq^{(12,1,1)}+Sq^{(0,5,1)}+Sq^{(8,0,2)}+Sq^{(1,0,3)}+Sq^{(7,0,0,1)}+Sq^{(4,1,0,1)})  3_{4}
+(Sq^{17}+Sq^{(14,1)}+Sq^{(8,3)}+Sq^{(5,4)}+Sq^{(0,1,2)})  3_{5}
+(Sq^{16}+Sq^{(10,2)}+Sq^{(7,3)}+Sq^{(4,4)}+Sq^{(1,5)}+Sq^{(2,0,2)})  3_{6}
+(Sq^{14}+Sq^{(11,1)}+Sq^{(8,2)}+Sq^{(2,4)}+Sq^{(7,0,1)}+Sq^{(4,1,1)}+Sq^{(0,0,2)})  3_{7}
+(Sq^{12}+Sq^{(9,1)}+Sq^{(0,4)})  3_{9}
+Sq^{1}  3_{11}
$

\item[35:] $d(4_{12}) =
(Sq^{32}+Sq^{(23,3)}+Sq^{(20,4)}+Sq^{(17,5)}+Sq^{(11,7)}+Sq^{(8,8)}+Sq^{(15,1,2)}+Sq^{(12,2,2)}+Sq^{(3,5,2)}+Sq^{(0,6,2)}+Sq^{(11,0,3)}+Sq^{(5,2,3)}+Sq^{(2,3,3)}+Sq^{(4,0,4)}+Sq^{(1,1,4)})  3_{0}
+(Sq^{29}+Sq^{(26,1)}+Sq^{(23,2)}+Sq^{(20,3)}+Sq^{(17,4)}+Sq^{(14,5)}+Sq^{(11,6)}+Sq^{(8,7)}+Sq^{(2,9)}+Sq^{(22,0,1)}+Sq^{(16,2,1)}+Sq^{(10,4,1)}+Sq^{(4,6,1)}+Sq^{(1,7,1)}+Sq^{(15,0,2)}+Sq^{(12,1,2)}+Sq^{(9,2,2)}+Sq^{(6,3,2)}+Sq^{(3,4,2)}+Sq^{(0,5,2)}+Sq^{(8,0,3)}+Sq^{(2,2,3)}+Sq^{(1,0,4)}+Sq^{(14,0,0,1)}+Sq^{(8,2,0,1)}+Sq^{(2,4,0,1)}+Sq^{(7,0,1,1)}+Sq^{(4,1,1,1)}+Sq^{(0,0,2,1)})  3_{1}
+(Sq^{(10,5)}+Sq^{(7,6)}+Sq^{(1,8)}+Sq^{(11,0,2)}+Sq^{(5,2,2)}+Sq^{(2,3,2)}+Sq^{(1,1,3)})  3_{2}
+Sq^{(5,4,1)}  3_{3}
+(Sq^{(14,3)}+Sq^{(11,4)}+Sq^{(5,6)}+Sq^{(2,7)}+Sq^{(13,1,1)}+Sq^{(1,5,1)}+Sq^{(9,0,2)}+Sq^{(6,1,2)}+Sq^{(3,2,2)}+Sq^{(5,1,0,1)}+Sq^{(1,0,1,1)})  3_{4}
+Sq^{(15,1)}  3_{5}
+Sq^{1}  3_{13}
$

\item[36:] $d(4_{13}) =
(Sq^{(9,8)}+Sq^{(19,0,2)}+Sq^{(16,1,2)}+Sq^{(13,2,2)}+Sq^{(10,3,2)}+Sq^{(1,6,2)}+Sq^{(12,0,3)}+Sq^{(9,1,3)}+Sq^{(0,4,3)})  3_{0}
+(Sq^{(27,1)}+Sq^{(21,3)}+Sq^{(9,7)}+Sq^{(23,0,1)}+Sq^{(20,1,1)}+Sq^{(5,6,1)}+Sq^{(10,2,2)}+Sq^{(1,5,2)}+Sq^{(9,0,3)}+Sq^{(3,2,3)}+Sq^{(2,0,4)}+Sq^{(15,0,0,1)}+Sq^{(3,4,0,1)}+Sq^{(8,0,1,1)})  3_{1}
+(Sq^{26}+Sq^{(17,3)}+Sq^{(11,5)}+Sq^{(8,6)}+Sq^{(5,7)}+Sq^{(2,8)}+Sq^{(12,0,2)}+Sq^{(6,2,2)}+Sq^{(2,1,3)})  3_{2}
+(Sq^{(22,1)}+Sq^{(13,4)}+Sq^{(10,5)}+Sq^{(7,6)}+Sq^{(4,7)}+Sq^{(18,0,1)}+Sq^{(6,4,1)}+Sq^{(3,5,1)}+Sq^{(0,6,1)})  3_{3}
+(Sq^{(21,1)}+Sq^{(18,2)}+Sq^{(15,3)}+Sq^{(12,4)}+Sq^{(14,1,1)}+Sq^{(5,4,1)}+Sq^{(1,3,2)}+Sq^{(9,0,0,1)}+Sq^{(3,2,0,1)})  3_{4}
+(Sq^{19}+Sq^{(16,1)}+Sq^{(13,2)}+Sq^{(10,3)}+Sq^{(4,5)}+Sq^{(1,6)}+Sq^{(5,0,2)})  3_{5}
+(Sq^{(6,4)}+Sq^{(3,5)}+Sq^{(1,1,2)})  3_{6}
+(Sq^{(1,5)}+Sq^{(2,0,2)})  3_{7}
+(Sq^{(6,3)}+Sq^{(3,4)}+Sq^{(0,5)})  3_{8}
+(Sq^{14}+Sq^{(11,1)}+Sq^{(8,2)}+Sq^{(1,2,1)}+Sq^{(0,0,2)})  3_{9}
+Sq^{(6,2)}  3_{10}
$

\item[37:] $d(4_{14}) =
(Sq^{(25,3)}+Sq^{(10,8)}+Sq^{(1,11)}+Sq^{(17,1,2)}+Sq^{(14,2,2)}+Sq^{(13,0,3)}+Sq^{(7,2,3)}+Sq^{(4,3,3)}+Sq^{(1,4,3)})  3_{0}
+(Sq^{31}+Sq^{(28,1)}+Sq^{(21,1,1)}+Sq^{(6,6,1)}+Sq^{(8,3,2)}+Sq^{(2,5,2)}+Sq^{(10,0,3)}+Sq^{(4,2,3)}+Sq^{(1,3,3)}+Sq^{(0,1,4)}+Sq^{(10,2,0,1)}+Sq^{(9,0,1,1)}+Sq^{(6,1,1,1)})  3_{1}
+(Sq^{27}+Sq^{(18,3)}+Sq^{(15,4)}+Sq^{(12,5)}+Sq^{(9,6)}+Sq^{(0,9)}+Sq^{(13,0,2)}+Sq^{(10,1,2)}+Sq^{(7,2,2)}+Sq^{(4,3,2)}+Sq^{(1,4,2)})  3_{2}
+(Sq^{(23,1)}+Sq^{(14,4)}+Sq^{(11,5)}+Sq^{(7,4,1)})  3_{3}
+(Sq^{(22,1)}+Sq^{(16,3)}+Sq^{(13,4)}+Sq^{(4,7)}+Sq^{(1,8)}+Sq^{(12,2,1)}+Sq^{(6,4,1)}+Sq^{(3,5,1)}+Sq^{(11,0,2)}+Sq^{(8,1,2)}+Sq^{(5,2,2)}+Sq^{(2,3,2)}+Sq^{(10,0,0,1)})  3_{4}
+(Sq^{(14,2)}+Sq^{(8,4)}+Sq^{(5,5)}+Sq^{(2,6)}+Sq^{(6,0,2)}+Sq^{(0,2,2)})  3_{5}
+(Sq^{19}+Sq^{(4,5)})  3_{6}
+(Sq^{17}+Sq^{(5,4)}+Sq^{(7,1,1)}+Sq^{(3,0,2)})  3_{7}
+(Sq^{(7,3)}+Sq^{(1,5)})  3_{8}
+(Sq^{(12,1)}+Sq^{(9,2)}+Sq^{(8,0,1)}+Sq^{(0,0,0,1)})  3_{9}
+(Sq^{13}+Sq^{(4,3)}+Sq^{(1,4)})  3_{10}
+Sq^{3}  3_{12}
$

\item[38:] $d(4_{15}) =
(Sq^{(29,2)}+Sq^{(14,7)}+Sq^{(8,9)}+Sq^{(5,10)}+Sq^{(18,1,2)}+Sq^{(12,3,2)}+Sq^{(6,5,2)}+Sq^{(0,7,2)}+Sq^{(8,2,3)}+Sq^{(2,4,3)}+Sq^{(7,0,4)}+Sq^{(4,1,4)}+Sq^{(0,0,5)}+Sq^{(2,1,0,2)})  3_{0}
+(Sq^{32}+Sq^{(29,1)}+Sq^{(20,4)}+Sq^{(14,6)}+Sq^{(11,7)}+Sq^{(8,8)}+Sq^{(5,9)}+Sq^{(2,10)}+Sq^{(25,0,1)}+Sq^{(19,2,1)}+Sq^{(13,4,1)}+Sq^{(10,5,1)}+Sq^{(7,6,1)}+Sq^{(4,7,1)}+Sq^{(12,2,2)}+Sq^{(9,3,2)}+Sq^{(3,5,2)}+Sq^{(0,6,2)}+Sq^{(11,0,3)}+Sq^{(2,3,3)}+Sq^{(4,0,4)}+Sq^{(14,1,0,1)}+Sq^{(11,2,0,1)}+Sq^{(5,4,0,1)}+Sq^{(2,5,0,1)})  3_{1}
+(Sq^{28}+Sq^{(22,2)}+Sq^{(19,3)}+Sq^{(16,4)}+Sq^{(7,7)}+Sq^{(5,3,2)}+Sq^{(2,4,2)}+Sq^{(4,1,3)}+Sq^{(0,0,4)})  3_{2}
+(Sq^{(24,1)}+Sq^{(9,6)}+Sq^{(6,7)}+Sq^{(3,8)}+Sq^{(20,0,1)}+Sq^{(8,4,1)}+Sq^{(5,5,1)}+Sq^{(12,0,0,1)}+Sq^{(0,4,0,1)})  3_{3}
+(Sq^{26}+Sq^{(23,1)}+Sq^{(20,2)}+Sq^{(17,3)}+Sq^{(5,7)}+Sq^{(2,8)}+Sq^{(13,2,1)}+Sq^{(7,4,1)}+Sq^{(4,5,1)}+Sq^{(1,6,1)}+Sq^{(6,2,2)}+Sq^{(0,4,2)}+Sq^{(2,1,3)}+Sq^{(11,0,0,1)}+Sq^{(8,1,0,1)}+Sq^{(4,0,1,1)})  3_{4}
+(Sq^{21}+Sq^{(18,1)})  3_{5}
+Sq^{4}  3_{13}
+Sq^{2}  3_{14}
+Sq^{1}  3_{15}
$

\item[42:] $d(4_{16}) =
(Sq^{(30,3)}+Sq^{(27,4)}+Sq^{(21,6)}+Sq^{(18,7)}+Sq^{(12,9)}+Sq^{(9,10)}+Sq^{(0,13)}+Sq^{(16,3,2)}+Sq^{(4,7,2)}+Sq^{(1,8,2)}+Sq^{(9,3,3)}+Sq^{(6,4,3)}+Sq^{(2,3,4)})  3_{0}
+(Sq^{(18,6)}+Sq^{(15,7)}+Sq^{(9,9)}+Sq^{(3,11)}+Sq^{(17,4,1)}+Sq^{(14,5,1)}+Sq^{(11,6,1)}+Sq^{(5,8,1)}+Sq^{(13,3,2)}+Sq^{(10,4,2)}+Sq^{(7,5,2)}+Sq^{(15,0,3)}+Sq^{(12,1,3)}+Sq^{(6,3,3)}+Sq^{(8,0,4)}+Sq^{(5,1,4)}+Sq^{(2,2,4)}+Sq^{(21,0,0,1)}+Sq^{(18,1,0,1)}+Sq^{(15,2,0,1)}+Sq^{(6,5,0,1)}+Sq^{(3,6,0,1)}+Sq^{(2,4,1,1)})  3_{1}
+(Sq^{(23,3)}+Sq^{(20,4)}+Sq^{(17,5)}+Sq^{(14,6)}+Sq^{(11,7)}+Sq^{(2,10)}+Sq^{(18,0,2)}+Sq^{(15,1,2)}+Sq^{(3,5,2)}+Sq^{(8,1,3)})  3_{2}
+(Sq^{(28,1)}+Sq^{(19,4)}+Sq^{(7,8)}+Sq^{(4,9)}+Sq^{(12,4,1)}+Sq^{(9,5,1)}+Sq^{(6,6,1)})  3_{3}
+(Sq^{30}+Sq^{(24,2)}+Sq^{(21,3)}+Sq^{(18,4)}+Sq^{(15,5)}+Sq^{(12,6)}+Sq^{(9,7)}+Sq^{(6,8)}+Sq^{(3,9)}+Sq^{(23,0,1)}+Sq^{(11,4,1)}+Sq^{(5,6,1)}+Sq^{(2,7,1)}+Sq^{(16,0,2)}+Sq^{(10,2,2)}+Sq^{(7,3,2)}+Sq^{(15,0,0,1)}+Sq^{(3,4,0,1)}+Sq^{(0,5,0,1)})  3_{4}
+(Sq^{(22,1)}+Sq^{(19,2)}+Sq^{(16,3)}+Sq^{(13,4)}+Sq^{(8,1,2)}+Sq^{(5,2,2)})  3_{5}
+(Sq^{24}+Sq^{(18,2)}+Sq^{(15,3)}+Sq^{(12,4)}+Sq^{(9,5)}+Sq^{(6,6)}+Sq^{(10,0,2)}+Sq^{(1,3,2)})  3_{6}
+(Sq^{22}+Sq^{(19,1)}+Sq^{(10,4)}+Sq^{(7,5)}+Sq^{(4,6)}+Sq^{(1,7)}+Sq^{(15,0,1)}+Sq^{(8,0,2)}+Sq^{(7,0,0,1)}+Sq^{(4,1,0,1)})  3_{7}
+(Sq^{(15,2)}+Sq^{(9,4)}+Sq^{(6,5)}+Sq^{(3,6)}+Sq^{(7,0,2)}+Sq^{(1,2,2)})  3_{8}
+(Sq^{(14,2)}+Sq^{(8,4)}+Sq^{(0,2,2)})  3_{9}
+(Sq^{(15,1)}+Sq^{(12,2)}+Sq^{(6,4)}+Sq^{(0,6)}+Sq^{(4,0,2)}+Sq^{(1,1,2)})  3_{10}
+(Sq^{9}+Sq^{(3,2)}+Sq^{(0,3)})  3_{11}
+(Sq^{(5,1)}+Sq^{(2,2)})  3_{12}
$

\item[42:] $d(4_{17}) =
(Sq^{39}+Sq^{(27,4)}+Sq^{(24,5)}+Sq^{(18,7)}+Sq^{(12,9)}+Sq^{(9,10)}+Sq^{(0,13)}+Sq^{(16,3,2)}+Sq^{(13,4,2)}+Sq^{(10,5,2)}+Sq^{(7,6,2)}+Sq^{(1,8,2)}+Sq^{(12,2,3)}+Sq^{(9,3,3)}+Sq^{(6,4,3)}+Sq^{(11,0,4)}+Sq^{(2,3,4)}+Sq^{(9,0,0,2)}+Sq^{(6,1,0,2)}+Sq^{(3,2,0,2)}+Sq^{(0,3,0,2)}+Sq^{(2,0,1,2)})  3_{0}
+(Sq^{(33,1)}+Sq^{(27,3)}+Sq^{(12,8)}+Sq^{(3,11)}+Sq^{(0,12)}+Sq^{(26,1,1)}+Sq^{(23,2,1)}+Sq^{(14,5,1)}+Sq^{(8,7,1)}+Sq^{(22,0,2)}+Sq^{(10,4,2)}+Sq^{(1,7,2)}+Sq^{(12,1,3)}+Sq^{(0,5,3)}+Sq^{(5,1,4)}+Sq^{(1,0,5)}+Sq^{(21,0,0,1)}+Sq^{(18,1,0,1)}+Sq^{(15,2,0,1)}+Sq^{(0,7,0,1)}+Sq^{(2,4,1,1)}+Sq^{(7,0,2,1)}+Sq^{(1,2,2,1)}+Sq^{(6,0,0,2)}+Sq^{(3,1,0,2)}+Sq^{(5,0,0,0,1)})  3_{1}
+(Sq^{32}+Sq^{(14,6)}+Sq^{(2,10)}+Sq^{(9,3,2)}+Sq^{(11,0,3)}+Sq^{(8,1,3)}+Sq^{(1,1,4)})  3_{2}
+(Sq^{31}+Sq^{(28,1)}+Sq^{(19,4)}+Sq^{(16,5)}+Sq^{(7,8)}+Sq^{(1,10)}+Sq^{(24,0,1)}+Sq^{(12,4,1)}+Sq^{(9,5,1)}+Sq^{(3,0,4)}+Sq^{(0,1,4)}+Sq^{(16,0,0,1)}+Sq^{(0,0,0,0,1)})  3_{3}
+(Sq^{(21,3)}+Sq^{(15,5)}+Sq^{(9,7)}+Sq^{(3,9)}+Sq^{(20,1,1)}+Sq^{(8,5,1)}+Sq^{(5,6,1)}+Sq^{(10,2,2)}+Sq^{(7,3,2)}+Sq^{(1,5,2)}+Sq^{(3,2,3)}+Sq^{(15,0,0,1)}+Sq^{(5,1,1,1)})  3_{4}
+(Sq^{(16,3)}+Sq^{(7,6)}+Sq^{(1,8)}+Sq^{(18,0,1)})  3_{5}
+(Sq^{24}+Sq^{(18,2)}+Sq^{(15,3)}+Sq^{(12,4)}+Sq^{(9,5)}+Sq^{(6,6)}+Sq^{(0,8)}+Sq^{(10,0,2)}+Sq^{(7,1,2)}+Sq^{(4,2,2)}+Sq^{(3,0,3)}+Sq^{(0,1,3)})  3_{6}
+(Sq^{22}+Sq^{(19,1)}+Sq^{(16,2)}+Sq^{(10,4)}+Sq^{(4,6)}+Sq^{(15,0,1)}+Sq^{(12,1,1)}+Sq^{(3,4,1)}+Sq^{(0,5,1)}+Sq^{(8,0,2)}+Sq^{(2,2,2)}+Sq^{(4,1,0,1)}+Sq^{(0,0,1,1)})  3_{7}
+Sq^{(9,4)}  3_{8}
+Sq^{(17,1)}  3_{9}
+(Sq^{(15,1)}+Sq^{(1,1,2)})  3_{10}
+(Sq^{9}+Sq^{(3,2)})  3_{11}
+(Sq^{(5,1)}+Sq^{(1,0,1)})  3_{12}
+(Sq^{8}+Sq^{(2,2)})  3_{13}
+(Sq^{6}+Sq^{(0,2)})  3_{14}
+Sq^{1}  3_{17}
$

\item[43:] $d(4_{18}) =
(Sq^{40}+Sq^{(31,3)}+Sq^{(28,4)}+Sq^{(25,5)}+Sq^{(13,9)}+Sq^{(1,13)}+Sq^{(26,0,2)}+Sq^{(20,2,2)}+Sq^{(11,5,2)}+Sq^{(8,6,2)}+Sq^{(5,7,2)}+Sq^{(2,8,2)}+Sq^{(19,0,3)}+Sq^{(10,3,3)}+Sq^{(4,5,3)}+Sq^{(1,6,3)}+Sq^{(6,2,4)}+Sq^{(0,4,4)}+Sq^{(10,0,0,2)}+Sq^{(7,1,0,2)}+Sq^{(1,3,0,2)}+Sq^{(3,0,1,2)})  3_{0}
+(Sq^{37}+Sq^{(34,1)}+Sq^{(31,2)}+Sq^{(28,3)}+Sq^{(25,4)}+Sq^{(19,6)}+Sq^{(16,7)}+Sq^{(13,8)}+Sq^{(7,10)}+Sq^{(1,12)}+Sq^{(30,0,1)}+Sq^{(18,4,1)}+Sq^{(9,7,1)}+Sq^{(3,9,1)}+Sq^{(0,10,1)}+Sq^{(20,1,2)}+Sq^{(17,2,2)}+Sq^{(14,3,2)}+Sq^{(5,6,2)}+Sq^{(2,7,2)}+Sq^{(16,0,3)}+Sq^{(10,2,3)}+Sq^{(7,3,3)}+Sq^{(4,4,3)}+Sq^{(0,3,4)}+Sq^{(22,0,0,1)}+Sq^{(16,2,0,1)}+Sq^{(4,6,0,1)}+Sq^{(1,7,0,1)}+Sq^{(12,1,1,1)}+Sq^{(8,0,2,1)}+Sq^{(2,2,2,1)}+Sq^{(7,0,0,2)}+Sq^{(4,1,0,2)}+Sq^{(6,0,0,0,1)})  3_{1}
+(Sq^{33}+Sq^{(27,2)}+Sq^{(3,10)}+Sq^{(0,11)}+Sq^{(19,0,2)}+Sq^{(10,3,2)}+Sq^{(12,0,3)}+Sq^{(9,1,3)}+Sq^{(6,2,3)}+Sq^{(2,1,4)}+Sq^{(3,0,0,2)})  3_{2}
+(Sq^{32}+Sq^{(20,4)}+Sq^{(17,5)}+Sq^{(8,8)}+Sq^{(2,10)}+Sq^{(4,0,4)}+Sq^{(17,0,0,1)}+Sq^{(5,4,0,1)})  3_{3}
+(Sq^{31}+Sq^{(13,6)}+Sq^{(10,7)}+Sq^{(4,9)}+Sq^{(18,2,1)}+Sq^{(9,5,1)}+Sq^{(6,6,1)}+Sq^{(3,7,1)}+Sq^{(14,1,2)}+Sq^{(11,2,2)}+Sq^{(2,5,2)}+Sq^{(10,0,3)}+Sq^{(16,0,0,1)}+Sq^{(10,2,0,1)}+Sq^{(1,5,0,1)}+Sq^{(2,0,2,1)}+Sq^{(0,0,0,0,1)})  3_{4}
+(Sq^{26}+Sq^{(17,3)}+Sq^{(5,7)}+Sq^{(2,8)}+Sq^{(9,1,2)}+Sq^{(6,2,2)})  3_{5}
+(Sq^{21}+Sq^{(12,3)}+Sq^{(9,4)}+Sq^{(14,0,1)}+Sq^{(0,2,0,1)})  3_{9}
+(Sq^{9}+Sq^{(3,2)})  3_{13}
+Sq^{(0,0,1)}  3_{14}
+Sq^{6}  3_{15}
+(Sq^{3}+Sq^{(0,1)})  3_{16}
$

\item[44:] $d(4_{19}) =
(Sq^{41}+Sq^{(29,4)}+Sq^{(26,5)}+Sq^{(20,7)}+Sq^{(17,8)}+Sq^{(8,11)}+Sq^{(2,13)}+Sq^{(27,0,2)}+Sq^{(24,1,2)}+Sq^{(0,9,2)}+Sq^{(20,0,3)}+Sq^{(11,3,3)}+Sq^{(5,5,3)}+Sq^{(2,6,3)}+Sq^{(13,0,4)}+Sq^{(4,3,4)}+Sq^{(3,1,5)}+Sq^{(0,2,5)}+Sq^{(11,0,0,2)}+Sq^{(2,3,0,2)})  3_{0}
+(Sq^{(35,1)}+Sq^{(29,3)}+Sq^{(20,6)}+Sq^{(11,9)}+Sq^{(28,1,1)}+Sq^{(13,6,1)}+Sq^{(1,10,1)}+Sq^{(15,3,2)}+Sq^{(6,6,2)}+Sq^{(14,1,3)}+Sq^{(8,3,3)}+Sq^{(5,4,3)}+Sq^{(2,5,3)}+Sq^{(4,2,4)}+Sq^{(23,0,0,1)}+Sq^{(17,2,0,1)}+Sq^{(5,6,0,1)}+Sq^{(16,0,1,1)}+Sq^{(4,4,1,1)}+Sq^{(9,0,2,1)}+Sq^{(3,2,2,1)}+Sq^{(7,0,0,0,1)}+Sq^{(4,1,0,0,1)})  3_{1}
+(Sq^{(28,2)}+Sq^{(4,10)}+Sq^{(8,4,2)}+Sq^{(2,6,2)}+Sq^{(13,0,3)}+Sq^{(7,2,3)}+Sq^{(1,4,3)})  3_{2}
+(Sq^{(30,1)}+Sq^{(21,4)}+Sq^{(15,6)}+Sq^{(12,7)}+Sq^{(3,10)}+Sq^{(0,11)}+Sq^{(5,7,1)}+Sq^{(5,0,4)}+Sq^{(18,0,0,1)}+Sq^{(6,4,0,1)})  3_{3}
+(Sq^{(23,3)}+Sq^{(14,6)}+Sq^{(8,8)}+Sq^{(13,4,1)}+Sq^{(10,5,1)}+Sq^{(7,6,1)}+Sq^{(4,7,1)}+Sq^{(6,4,2)}+Sq^{(11,0,3)}+Sq^{(17,0,0,1)}+Sq^{(14,1,0,1)}+Sq^{(5,4,0,1)}+Sq^{(1,0,0,0,1)})  3_{4}
+(Sq^{27}+Sq^{(24,1)}+Sq^{(21,2)}+Sq^{(18,3)}+Sq^{(15,4)}+Sq^{(3,8)}+Sq^{(13,0,2)}+Sq^{(7,2,2)}+Sq^{(1,4,2)}+Sq^{(6,0,3)}+Sq^{(3,1,3)})  3_{5}
+(Sq^{26}+Sq^{(20,2)}+Sq^{(14,4)}+Sq^{(3,3,2)})  3_{6}
+(Sq^{24}+Sq^{(6,6)}+Sq^{(3,7)}+Sq^{(17,0,1)}+Sq^{(2,5,1)}+Sq^{(4,2,2)}+Sq^{(9,0,0,1)})  3_{7}
+(Sq^{23}+Sq^{(11,4)}+Sq^{(3,2,2)})  3_{8}
+(Sq^{(19,1)}+Sq^{(15,0,1)}+Sq^{(8,0,2)}+Sq^{(1,0,3)})  3_{9}
+(Sq^{(17,1)}+Sq^{(14,2)}+Sq^{(5,5)}+Sq^{(3,1,2)}+Sq^{(0,2,2)})  3_{10}
+(Sq^{11}+Sq^{(5,2)}+Sq^{(2,3)})  3_{11}
+(Sq^{10}+Sq^{(7,1)}+Sq^{(4,2)})  3_{12}
$

\item[44:] $d(4_{20}) =
(Sq^{(29,4)}+Sq^{(23,6)}+Sq^{(20,7)}+Sq^{(17,8)}+Sq^{(11,10)}+Sq^{(8,11)}+Sq^{(5,12)}+Sq^{(27,0,2)}+Sq^{(24,1,2)}+Sq^{(9,6,2)}+Sq^{(6,7,2)}+Sq^{(0,9,2)}+Sq^{(17,1,3)}+Sq^{(11,3,3)}+Sq^{(13,0,4)}+Sq^{(7,2,4)}+Sq^{(4,3,4)}+Sq^{(1,4,4)}+Sq^{(6,0,5)}+Sq^{(3,1,5)}+Sq^{(0,2,5)}+Sq^{(11,0,0,2)}+Sq^{(8,1,0,2)}+Sq^{(1,1,1,2)})  3_{0}
+(Sq^{(29,3)}+Sq^{(17,7)}+Sq^{(14,8)}+Sq^{(11,9)}+Sq^{(31,0,1)}+Sq^{(28,1,1)}+Sq^{(19,4,1)}+Sq^{(16,5,1)}+Sq^{(10,7,1)}+Sq^{(7,8,1)}+Sq^{(1,10,1)}+Sq^{(15,3,2)}+Sq^{(6,6,2)}+Sq^{(17,0,3)}+Sq^{(14,1,3)}+Sq^{(2,5,3)}+Sq^{(10,0,4)}+Sq^{(7,1,4)}+Sq^{(0,1,5)}+Sq^{(23,0,0,1)}+Sq^{(20,1,0,1)}+Sq^{(17,2,0,1)}+Sq^{(11,4,0,1)}+Sq^{(8,5,0,1)}+Sq^{(5,6,0,1)}+Sq^{(13,1,1,1)}+Sq^{(4,4,1,1)}+Sq^{(1,5,1,1)}+Sq^{(0,3,2,1)}+Sq^{(7,0,0,0,1)}+Sq^{(4,1,0,0,1)})  3_{1}
+(Sq^{34}+Sq^{(28,2)}+Sq^{(25,3)}+Sq^{(19,5)}+Sq^{(13,7)}+Sq^{(7,9)}+Sq^{(1,11)}+Sq^{(17,1,2)}+Sq^{(5,5,2)}+Sq^{(10,1,3)}+Sq^{(3,1,4)}+Sq^{(0,2,4)}+Sq^{(1,1,0,2)})  3_{2}
+(Sq^{(30,1)}+Sq^{(21,4)}+Sq^{(18,5)}+Sq^{(12,7)}+Sq^{(9,8)}+Sq^{(0,11)}+Sq^{(11,5,1)}+Sq^{(8,6,1)}+Sq^{(2,8,1)}+Sq^{(2,1,4)}+Sq^{(6,4,0,1)}+Sq^{(0,6,0,1)})  3_{3}
+(Sq^{32}+Sq^{(29,1)}+Sq^{(26,2)}+Sq^{(23,3)}+Sq^{(17,5)}+Sq^{(11,7)}+Sq^{(8,8)}+Sq^{(2,10)}+Sq^{(19,2,1)}+Sq^{(13,4,1)}+Sq^{(7,6,1)}+Sq^{(4,7,1)}+Sq^{(18,0,2)}+Sq^{(12,2,2)}+Sq^{(9,3,2)}+Sq^{(6,4,2)}+Sq^{(8,1,3)}+Sq^{(2,3,3)}+Sq^{(1,1,4)}+Sq^{(17,0,0,1)}+Sq^{(11,2,0,1)}+Sq^{(10,0,1,1)}+Sq^{(2,0,0,2)})  3_{4}
+(Sq^{27}+Sq^{(21,2)}+Sq^{(18,3)}+Sq^{(12,5)}+Sq^{(9,6)}+Sq^{(0,9)}+Sq^{(20,0,1)}+Sq^{(2,6,1)}+Sq^{(13,0,2)}+Sq^{(10,1,2)}+Sq^{(7,2,2)}+Sq^{(1,4,2)}+Sq^{(6,0,3)})  3_{5}
+(Sq^{26}+Sq^{(17,3)}+Sq^{(14,4)}+Sq^{(5,7)}+Sq^{(3,3,2)})  3_{6}
+(Sq^{24}+Sq^{(21,1)}+Sq^{(18,2)}+Sq^{(12,4)}+Sq^{(0,8)}+Sq^{(14,1,1)}+Sq^{(5,4,1)}+Sq^{(2,5,1)}+Sq^{(4,2,2)}+Sq^{(9,0,0,1)})  3_{7}
+(Sq^{(11,4)}+Sq^{(5,6)}+Sq^{(2,7)}+Sq^{(3,2,2)}+Sq^{(2,0,3)})  3_{8}
+(Sq^{22}+Sq^{(19,1)}+Sq^{(13,3)}+Sq^{(10,4)}+Sq^{(15,0,1)}+Sq^{(9,2,1)}+Sq^{(8,0,2)}+Sq^{(1,0,3)})  3_{9}
+(Sq^{20}+Sq^{(8,4)}+Sq^{(2,6)}+Sq^{(1,4,1)}+Sq^{(0,2,2)})  3_{10}
+Sq^{11}  3_{11}
+Sq^{10}  3_{13}
+(Sq^{8}+Sq^{(1,0,1)})  3_{14}
+Sq^{4}  3_{16}
+Sq^{3}  3_{17}
+Sq^{2}  3_{18}
$

\end{itemize}

\subsection{Homological degree 5}

Complete through degree $t=44$.

\begin{itemize}
\addtolength{\itemsep}{2.4ex}

\item[5:] $d(5_{0}) =
Sq^{1}  4_{0}
$

\item[14:] $d(5_{1}) =
Sq^{10}  4_{0}
+(Sq^{3}+Sq^{(0,1)})  4_{1}
$

\item[16:] $d(5_{2}) =
Sq^{12}  4_{0}
+(Sq^{5}+Sq^{(2,1)})  4_{1}
+Sq^{3}  4_{2}
$

\item[19:] $d(5_{3}) =
(Sq^{(9,2)}+Sq^{(6,3)}+Sq^{(3,4)})  4_{0}
+Sq^{(5,1)}  4_{1}
+(Sq^{6}+Sq^{(0,2)})  4_{2}
+Sq^{1}  4_{3}
$

\item[20:] $d(5_{4}) =
(Sq^{(10,2)}+Sq^{(4,4)})  4_{0}
+(Sq^{9}+Sq^{(6,1)})  4_{1}
+Sq^{(0,0,1)}  4_{2}
+Sq^{2}  4_{3}
$

\item[20:] $d(5_{5}) =
Sq^{16}  4_{0}
+(Sq^{9}+Sq^{(6,1)})  4_{1}
+Sq^{1}  4_{4}
$

\item[22:] $d(5_{6}) =
(Sq^{(9,3)}+Sq^{(6,4)}+Sq^{(3,5)}+Sq^{(0,6)})  4_{0}
+(Sq^{9}+Sq^{(6,1)})  4_{2}
+Sq^{4}  4_{3}
+Sq^{1}  4_{5}
$

\item[23:] $d(5_{7}) =
Sq^{(13,2)}  4_{0}
+Sq^{12}  4_{1}
+(Sq^{10}+Sq^{(7,1)}+Sq^{(4,2)}+Sq^{(0,1,1)})  4_{2}
+Sq^{2}  4_{5}
+Sq^{1}  4_{6}
$

\item[25:] $d(5_{8}) =
(Sq^{(12,3)}+Sq^{(9,4)}+Sq^{(6,5)}+Sq^{(3,6)}+Sq^{(0,7)})  4_{0}
+Sq^{(11,1)}  4_{1}
+Sq^{12}  4_{2}
+Sq^{7}  4_{3}
+Sq^{4}  4_{5}
+Sq^{3}  4_{6}
+Sq^{1}  4_{8}
$

\item[26:] $d(5_{9}) =
(Sq^{(13,3)}+Sq^{(4,6)})  4_{0}
+(Sq^{15}+Sq^{(12,1)})  4_{1}
+(Sq^{13}+Sq^{(7,2)}+Sq^{(6,0,1)})  4_{2}
+Sq^{4}  4_{6}
+Sq^{2}  4_{8}
$

\item[28:] $d(5_{10}) =
(Sq^{(0,8)}+Sq^{(10,0,2)}+Sq^{(4,2,2)})  4_{0}
+(Sq^{17}+Sq^{(11,2)}+Sq^{(8,3)}+Sq^{(5,4)})  4_{1}
+(Sq^{15}+Sq^{(12,1)}+Sq^{(8,0,1)}+Sq^{(5,1,1)}+Sq^{(0,0,0,1)})  4_{2}
+Sq^{10}  4_{3}
+Sq^{9}  4_{4}
+Sq^{(0,2)}  4_{6}
+Sq^{(3,1)}  4_{7}
+Sq^{4}  4_{8}
$

\item[29:] $d(5_{11}) =
(Sq^{25}+Sq^{(13,4)}+Sq^{(10,5)}+Sq^{(1,8)}+Sq^{(4,0,3)})  4_{0}
+(Sq^{(15,1)}+Sq^{(12,2)}+Sq^{(9,3)}+Sq^{(6,4)}+Sq^{(0,6)})  4_{1}
+(Sq^{16}+Sq^{(10,2)}+Sq^{(9,0,1)})  4_{2}
+(Sq^{10}+Sq^{(4,2)})  4_{4}
+(Sq^{8}+Sq^{(2,2)})  4_{5}
+Sq^{(1,2)}  4_{6}
+(Sq^{7}+Sq^{(1,2)}+Sq^{(0,0,1)})  4_{7}
+Sq^{2}  4_{10}
$

\item[35:] $d(5_{12}) =
(Sq^{31}+Sq^{(25,2)}+Sq^{(22,3)}+Sq^{(19,4)}+Sq^{(13,6)}+Sq^{(7,8)}+Sq^{(1,10)}+Sq^{(11,2,2)}+Sq^{(7,1,3)}+Sq^{(1,3,3)}+Sq^{(3,0,4)})  4_{0}
+(Sq^{(15,3)}+Sq^{(12,4)}+Sq^{(9,5)}+Sq^{(6,6)}+Sq^{(3,7)}+Sq^{(0,8)}+Sq^{(10,0,2)}+Sq^{(7,1,2)}+Sq^{(1,3,2)}+Sq^{(3,0,3)})  4_{1}
+(Sq^{22}+Sq^{(19,1)}+Sq^{(16,2)}+Sq^{(7,5)}+Sq^{(1,7)}+Sq^{(12,1,1)}+Sq^{(3,4,1)}+Sq^{(0,5,1)}+Sq^{(2,2,2)}+Sq^{(7,0,0,1)}+Sq^{(4,1,0,1)}+Sq^{(0,0,1,1)})  4_{2}
+Sq^{17}  4_{3}
+(Sq^{16}+Sq^{(7,3)}+Sq^{(4,4)}+Sq^{(1,5)})  4_{4}
+(Sq^{(2,4)}+Sq^{(0,0,2)})  4_{5}
+(Sq^{13}+Sq^{(1,4)})  4_{6}
+(Sq^{13}+Sq^{(10,1)}+Sq^{(7,2)}+Sq^{(4,3)}+Sq^{(1,4)}+Sq^{(6,0,1)}+Sq^{(0,2,1)})  4_{7}
+(Sq^{9}+Sq^{(6,1)}+Sq^{(3,2)})  4_{9}
+Sq^{1}  4_{11}
$

\item[36:] $d(5_{13}) =
(Sq^{(26,2)}+Sq^{(23,3)}+Sq^{(20,4)}+Sq^{(14,6)}+Sq^{(11,7)}+Sq^{(2,10)}+Sq^{(0,6,2)}+Sq^{(11,0,3)}+Sq^{(5,2,3)}+Sq^{(4,0,4)})  4_{0}
+(Sq^{25}+Sq^{(22,1)}+Sq^{(10,5)}+Sq^{(4,7)}+Sq^{(11,0,2)})  4_{1}
+(Sq^{23}+Sq^{(17,2)}+Sq^{(11,4)}+Sq^{(5,6)}+Sq^{(16,0,1)}+Sq^{(4,4,1)}+Sq^{(1,5,1)}+Sq^{(9,0,2)}+Sq^{(3,2,2)}+Sq^{(5,1,0,1)})  4_{2}
+Sq^{18}  4_{3}
+(Sq^{17}+Sq^{(5,4)})  4_{4}
+(Sq^{15}+Sq^{(9,2)}+Sq^{(3,4)})  4_{5}
+Sq^{(0,0,2)}  4_{6}
+(Sq^{(11,1)}+Sq^{(5,3)}+Sq^{(7,0,1)}+Sq^{(1,2,1)})  4_{7}
+Sq^{12}  4_{8}
+(Sq^{(4,2)}+Sq^{(1,3)}+Sq^{(3,0,1)}+Sq^{(0,1,1)})  4_{9}
$

\item[36:] $d(5_{14}) =
(Sq^{32}+Sq^{(14,6)}+Sq^{(8,8)}+Sq^{(5,9)})  4_{0}
+(Sq^{25}+Sq^{(22,1)}+Sq^{(10,5)}+Sq^{(4,7)}+Sq^{(12,2,1)}+Sq^{(6,4,1)}+Sq^{(0,6,1)})  4_{1}
+(Sq^{(11,4)}+Sq^{(5,6)}+Sq^{(13,1,1)}+Sq^{(9,0,2)})  4_{2}
+(Sq^{17}+Sq^{(5,4)})  4_{4}
+(Sq^{(11,1)}+Sq^{(5,3)}+Sq^{(7,0,1)}+Sq^{(1,2,1)})  4_{7}
+Sq^{1}  4_{12}
$

\item[38:] $d(5_{15}) =
(Sq^{(25,3)}+Sq^{(22,4)}+Sq^{(19,5)}+Sq^{(1,11)}+Sq^{(14,2,2)}+Sq^{(11,3,2)}+Sq^{(5,5,2)}+Sq^{(2,6,2)}+Sq^{(7,2,3)}+Sq^{(1,4,3)}+Sq^{(3,1,4)})  4_{0}
+(Sq^{(24,1)}+Sq^{(15,4)}+Sq^{(12,5)}+Sq^{(9,6)}+Sq^{(2,6,1)}+Sq^{(4,3,2)}+Sq^{(1,4,2)}+Sq^{(6,0,3)}+Sq^{(0,2,3)})  4_{1}
+(Sq^{(4,7)}+Sq^{(18,0,1)}+Sq^{(15,1,1)}+Sq^{(6,4,1)}+Sq^{(2,3,2)}+Sq^{(10,0,0,1)}+Sq^{(7,1,0,1)})  4_{2}
+Sq^{(8,4)}  4_{3}
+Sq^{(7,4)}  4_{4}
+(Sq^{(11,2)}+Sq^{(5,4)})  4_{5}
+(Sq^{(4,4)}+Sq^{(2,0,2)})  4_{6}
+(Sq^{(1,5)}+Sq^{(3,2,1)})  4_{7}
+Sq^{14}  4_{8}
+(Sq^{12}+Sq^{(9,1)}+Sq^{(3,3)}+Sq^{(0,4)}+Sq^{(2,1,1)})  4_{9}
+Sq^{2}  4_{13}
+Sq^{1}  4_{14}
$

\item[40:] $d(5_{16}) =
(Sq^{(30,2)}+Sq^{(24,4)}+Sq^{(18,6)}+Sq^{(9,9)}+Sq^{(6,10)}+Sq^{(0,12)}+Sq^{(7,5,2)}+Sq^{(1,7,2)}+Sq^{(12,1,3)}+Sq^{(6,3,3)}+Sq^{(0,5,3)})  4_{0}
+(Sq^{(17,4)}+Sq^{(8,7)}+Sq^{(5,8)}+Sq^{(2,9)}+Sq^{(10,4,1)}+Sq^{(7,5,1)}+Sq^{(12,1,2)}+Sq^{(9,2,2)}+Sq^{(0,5,2)}+Sq^{(5,1,3)}+Sq^{(2,2,3)}+Sq^{(0,0,2,1)})  4_{1}
+(Sq^{(24,1)}+Sq^{(6,7)}+Sq^{(3,8)}+Sq^{(20,0,1)}+Sq^{(17,1,1)}+Sq^{(7,2,2)}+Sq^{(9,1,0,1)})  4_{2}
+(Sq^{(10,4)}+Sq^{(4,6)}+Sq^{(1,7)})  4_{3}
+(Sq^{(15,2)}+Sq^{(12,3)}+Sq^{(3,6)}+Sq^{(0,7)}+Sq^{(4,1,2)})  4_{4}
+(Sq^{18}+Sq^{(12,2)}+Sq^{(9,3)}+Sq^{(6,4)}+Sq^{(1,1,2)})  4_{6}
+(Sq^{(6,4)}+Sq^{(0,6)}+Sq^{(11,0,1)}+Sq^{(8,1,1)}+Sq^{(2,3,1)}+Sq^{(1,1,2)}+Sq^{(0,1,0,1)})  4_{7}
+Sq^{16}  4_{8}
+(Sq^{14}+Sq^{(2,4)}+Sq^{(1,2,1)}+Sq^{(0,0,2)})  4_{9}
+Sq^{(10,1)}  4_{10}
+Sq^{4}  4_{13}
$

\item[42:] $d(5_{17}) =
(Sq^{(29,3)}+Sq^{(26,4)}+Sq^{(20,6)}+Sq^{(8,10)}+Sq^{(2,12)}+Sq^{(21,1,2)}+Sq^{(12,4,2)}+Sq^{(9,5,2)}+Sq^{(6,6,2)}+Sq^{(0,8,2)}+Sq^{(8,3,3)}+Sq^{(5,4,3)}+Sq^{(2,5,3)}+Sq^{(4,2,4)}+Sq^{(3,0,5)})  4_{0}
+(Sq^{(28,1)}+Sq^{(19,4)}+Sq^{(10,7)}+Sq^{(1,10)}+Sq^{(9,5,1)}+Sq^{(11,2,2)}+Sq^{(10,0,3)}+Sq^{(0,1,4)}+Sq^{(2,0,2,1)})  4_{1}
+(Sq^{(14,5)}+Sq^{(11,6)}+Sq^{(8,7)}+Sq^{(2,9)}+Sq^{(22,0,1)}+Sq^{(19,1,1)}+Sq^{(4,6,1)}+Sq^{(9,2,2)}+Sq^{(6,3,2)}+Sq^{(11,1,0,1)}+Sq^{(2,4,0,1)})  4_{2}
+Sq^{24}  4_{3}
+(Sq^{23}+Sq^{(11,4)}+Sq^{(2,7)}+Sq^{(9,0,2)}+Sq^{(6,1,2)}+Sq^{(0,3,2)})  4_{4}
+(Sq^{(15,2)}+Sq^{(9,4)})  4_{5}
+(Sq^{(8,4)}+Sq^{(6,0,2)}+Sq^{(0,2,2)})  4_{6}
+(Sq^{(17,1)}+Sq^{(14,2)}+Sq^{(5,5)}+Sq^{(2,6)}+Sq^{(7,2,1)}+Sq^{(6,0,2)})  4_{7}
+(Sq^{(13,1)}+Sq^{(1,5)}+Sq^{(1,0,0,1)})  4_{9}
+(Sq^{15}+Sq^{(12,1)}+Sq^{(3,4)}+Sq^{(0,5)}+Sq^{(8,0,1)}+Sq^{(0,0,0,1)})  4_{10}
+Sq^{5}  4_{14}
$

\item[43:] $d(5_{18}) =
(Sq^{(30,3)}+Sq^{(27,4)}+Sq^{(18,7)}+Sq^{(15,8)}+Sq^{(12,9)}+Sq^{(13,4,2)}+Sq^{(10,5,2)}+Sq^{(4,7,2)}+Sq^{(12,2,3)}+Sq^{(3,5,3)}+Sq^{(5,2,4)}+Sq^{(2,3,4)}+Sq^{(1,1,5)}+Sq^{(9,0,0,2)}+Sq^{(6,1,0,2)}+Sq^{(3,2,0,2)})  4_{0}
+(Sq^{32}+Sq^{(26,2)}+Sq^{(14,6)}+Sq^{(11,7)}+Sq^{(8,8)}+Sq^{(5,9)}+Sq^{(2,10)}+Sq^{(15,1,2)}+Sq^{(3,5,2)}+Sq^{(0,6,2)})  4_{1}
+(Sq^{(15,5)}+Sq^{(12,6)}+Sq^{(9,7)}+Sq^{(6,8)}+Sq^{(0,10)}+Sq^{(11,4,1)}+Sq^{(8,5,1)}+Sq^{(5,6,1)}+Sq^{(2,7,1)}+Sq^{(10,2,2)}+Sq^{(4,4,2)}+Sq^{(2,0,4)}+Sq^{(15,0,0,1)}+Sq^{(12,1,0,1)}+Sq^{(3,4,0,1)}+Sq^{(0,5,0,1)}+Sq^{(0,0,0,2)})  4_{2}
+Sq^{25}  4_{3}
+(Sq^{24}+Sq^{(15,3)}+Sq^{(12,4)}+Sq^{(9,5)}+Sq^{(3,7)}+Sq^{(0,8)}+Sq^{(7,1,2)}+Sq^{(4,2,2)}+Sq^{(1,3,2)}+Sq^{(3,0,3)})  4_{4}
+(Sq^{22}+Sq^{(16,2)}+Sq^{(10,4)}+Sq^{(8,0,2)})  4_{5}
+(Sq^{21}+Sq^{(12,3)}+Sq^{(9,4)}+Sq^{(1,2,2)}+Sq^{(0,0,3)})  4_{6}
+(Sq^{21}+Sq^{(18,1)}+Sq^{(15,2)}+Sq^{(12,3)}+Sq^{(9,4)}+Sq^{(6,5)}+Sq^{(3,6)}+Sq^{(0,7)}+Sq^{(14,0,1)}+Sq^{(8,2,1)}+Sq^{(5,3,1)}+Sq^{(2,4,1)}+Sq^{(7,0,2)}+Sq^{(4,1,2)}+Sq^{(1,2,2)}+Sq^{(0,0,3)}+Sq^{(6,0,0,1)}+Sq^{(0,2,0,1)})  4_{7}
+Sq^{8}  4_{12}
+(Sq^{5}+Sq^{(2,1)})  4_{15}
+Sq^{1}  4_{17}
$

\item[44:] $d(5_{19}) =
(Sq^{(31,3)}+Sq^{(22,6)}+Sq^{(19,7)}+Sq^{(10,10)}+Sq^{(23,1,2)}+Sq^{(11,5,2)}+Sq^{(8,6,2)}+Sq^{(5,7,2)}+Sq^{(7,4,3)}+Sq^{(4,5,3)}+Sq^{(1,6,3)}+Sq^{(9,1,4)}+Sq^{(6,2,4)}+Sq^{(0,4,4)}+Sq^{(5,0,5)}+Sq^{(10,0,0,2)}+Sq^{(4,2,0,2)})  4_{0}
+(Sq^{(30,1)}+Sq^{(27,2)}+Sq^{(24,3)}+Sq^{(18,5)}+Sq^{(15,6)}+Sq^{(12,7)}+Sq^{(6,9)}+Sq^{(0,11)}+Sq^{(14,4,1)}+Sq^{(11,5,1)}+Sq^{(1,6,2)}+Sq^{(9,1,3)}+Sq^{(6,2,3)}+Sq^{(5,0,4)})  4_{1}
+(Sq^{(25,2)}+Sq^{(19,4)}+Sq^{(10,7)}+Sq^{(1,10)}+Sq^{(21,1,1)}+Sq^{(9,5,1)}+Sq^{(6,6,1)}+Sq^{(3,7,1)}+Sq^{(17,0,2)}+Sq^{(5,4,2)}+Sq^{(13,1,0,1)}+Sq^{(0,0,0,0,1)})  4_{2}
+(Sq^{26}+Sq^{(8,6)}+Sq^{(5,7)}+Sq^{(2,8)})  4_{3}
+(Sq^{25}+Sq^{(10,5)}+Sq^{(7,6)}+Sq^{(4,7)}+Sq^{(4,0,3)})  4_{4}
+(Sq^{23}+Sq^{(17,2)}+Sq^{(11,4)}+Sq^{(5,6)}+Sq^{(9,0,2)})  4_{5}
+(Sq^{22}+Sq^{(13,3)}+Sq^{(10,4)})  4_{6}
+(Sq^{(19,1)}+Sq^{(13,3)}+Sq^{(4,6)}+Sq^{(1,7)}+Sq^{(15,0,1)}+Sq^{(12,1,1)}+Sq^{(3,4,1)}+Sq^{(1,0,3)}+Sq^{(0,0,1,1)})  4_{7}
+Sq^{20}  4_{8}
+(Sq^{(12,2)}+Sq^{(9,3)}+Sq^{(6,4)}+Sq^{(11,0,1)}+Sq^{(8,1,1)}+Sq^{(4,0,2)}+Sq^{(3,0,0,1)})  4_{9}
+Sq^{(10,0,1)}  4_{10}
+Sq^{10}  4_{11}
+Sq^{8}  4_{13}
+Sq^{7}  4_{14}
+Sq^{2}  4_{16}
$

\end{itemize}

\subsection{Homological degree 6}

Complete through degree $t=44$.

\begin{itemize}
\addtolength{\itemsep}{2.4ex}

\item[6:] $d(6_{0}) =
Sq^{1}  5_{0}
$

\item[16:] $d(6_{1}) =
Sq^{11}  5_{0}
+Sq^{2}  5_{1}
$

\item[17:] $d(6_{2}) =
Sq^{12}  5_{0}
+Sq^{(0,1)}  5_{1}
+Sq^{1}  5_{2}
$

\item[20:] $d(6_{3}) =
Sq^{15}  5_{0}
+Sq^{6}  5_{1}
+Sq^{4}  5_{2}
+Sq^{1}  5_{3}
$

\item[21:] $d(6_{4}) =
(Sq^{16}+Sq^{(10,2)})  5_{0}
+(Sq^{7}+Sq^{(4,1)}+Sq^{(0,0,1)})  5_{1}
+Sq^{1}  5_{5}
$

\item[22:] $d(6_{5}) =
Sq^{(11,2)}  5_{0}
+(Sq^{8}+Sq^{(5,1)})  5_{1}
+Sq^{(0,2)}  5_{2}
+Sq^{3}  5_{3}
+Sq^{2}  5_{4}
$

\item[23:] $d(6_{6}) =
(Sq^{(9,3)}+Sq^{(6,4)}+Sq^{(3,5)})  5_{0}
+Sq^{9}  5_{1}
+Sq^{4}  5_{3}
+Sq^{(0,1)}  5_{4}
+Sq^{1}  5_{6}
$

\item[26:] $d(6_{7}) =
(Sq^{(9,4)}+Sq^{(3,6)}+Sq^{(0,7)})  5_{0}
+Sq^{7}  5_{3}
+Sq^{6}  5_{4}
+Sq^{4}  5_{6}
+(Sq^{3}+Sq^{(0,1)})  5_{7}
+Sq^{1}  5_{8}
$

\item[29:] $d(6_{8}) =
(Sq^{(15,3)}+Sq^{(9,5)}+Sq^{(3,7)}+Sq^{(0,8)}+Sq^{(10,0,2)})  5_{0}
+(Sq^{(12,1)}+Sq^{(3,4)}+Sq^{(0,5)}+Sq^{(8,0,1)}+Sq^{(0,0,0,1)})  5_{1}
+Sq^{9}  5_{4}
+Sq^{7}  5_{6}
+Sq^{6}  5_{7}
+Sq^{4}  5_{8}
+Sq^{(0,1)}  5_{9}
+Sq^{1}  5_{10}
$

\item[32:] $d(6_{9}) =
(Sq^{27}+Sq^{(15,4)}+Sq^{(6,7)}+Sq^{(3,8)}+Sq^{(10,1,2)})  5_{0}
+(Sq^{(15,1)}+Sq^{(3,5)}+Sq^{(0,6)}+Sq^{(11,0,1)})  5_{1}
+(Sq^{16}+Sq^{(4,4)}+Sq^{(2,0,2)})  5_{2}
+Sq^{(7,2)}  5_{3}
+Sq^{12}  5_{4}
+Sq^{12}  5_{5}
+Sq^{(1,3)}  5_{6}
+(Sq^{9}+Sq^{(6,1)}+Sq^{(0,3)})  5_{7}
+Sq^{(1,2)}  5_{8}
+Sq^{(0,2)}  5_{9}
+Sq^{4}  5_{10}
+Sq^{3}  5_{11}
$

\item[36:] $d(6_{10}) =
(Sq^{31}+Sq^{(13,6)}+Sq^{(4,9)}+Sq^{(11,2,2)}+Sq^{(10,0,3)})  5_{0}
+(Sq^{22}+Sq^{(19,1)}+Sq^{(4,6)}+Sq^{(15,0,1)})  5_{1}
+(Sq^{(11,3)}+Sq^{(2,6)}+Sq^{(0,2,2)})  5_{2}
+Sq^{(8,3)}  5_{3}
+Sq^{(9,0,1)}  5_{4}
+(Sq^{(10,2)}+Sq^{(7,3)})  5_{5}
+Sq^{(5,2)}  5_{8}
+(Sq^{(7,1)}+Sq^{(4,2)})  5_{9}
+Sq^{(2,2)}  5_{10}
$

\item[37:] $d(6_{11}) =
(Sq^{32}+Sq^{(26,2)}+Sq^{(11,7)}+Sq^{(8,8)}+Sq^{(2,10)}+Sq^{(15,1,2)}+Sq^{(12,2,2)}+Sq^{(6,4,2)})  5_{0}
+(Sq^{23}+Sq^{(20,1)}+Sq^{(11,4)}+Sq^{(8,5)}+Sq^{(5,6)}+Sq^{(2,7)}+Sq^{(16,0,1)}+Sq^{(4,4,1)}+Sq^{(8,0,0,1)})  5_{1}
+Sq^{1}  5_{14}
$

\item[38:] $d(6_{12}) =
(Sq^{(27,2)}+Sq^{(15,6)}+Sq^{(12,7)}+Sq^{(3,10)}+Sq^{(10,3,2)}+Sq^{(1,6,2)})  5_{0}
+(Sq^{(12,4)}+Sq^{(6,6)}+Sq^{(0,8)}+Sq^{(2,5,1)})  5_{1}
+(Sq^{22}+Sq^{(16,2)}+Sq^{(13,3)}+Sq^{(10,4)}+Sq^{(8,0,2)}+Sq^{(1,0,3)})  5_{2}
+(Sq^{19}+Sq^{(13,2)}+Sq^{(10,3)})  5_{3}
+(Sq^{18}+Sq^{(15,1)}+Sq^{(11,0,1)})  5_{4}
+(Sq^{18}+Sq^{(12,2)}+Sq^{(6,4)})  5_{5}
+Sq^{(7,3)}  5_{6}
+(Sq^{15}+Sq^{(0,5)})  5_{7}
+(Sq^{13}+Sq^{(7,2)}+Sq^{(1,4)})  5_{8}
+(Sq^{(9,1)}+Sq^{(3,3)})  5_{9}
+Sq^{(4,2)}  5_{10}
+(Sq^{3}+Sq^{(0,1)})  5_{12}
$

\item[40:] $d(6_{13}) =
(Sq^{(11,8)}+Sq^{(15,2,2)}+Sq^{(12,3,2)}+Sq^{(9,4,2)}+Sq^{(6,5,2)}+Sq^{(3,6,2)}+Sq^{(2,4,3)}+Sq^{(7,0,4)})  5_{0}
+(Sq^{(14,4)}+Sq^{(5,7)}+Sq^{(2,8)}+Sq^{(19,0,1)}+Sq^{(4,5,1)})  5_{1}
+(Sq^{(15,3)}+Sq^{(10,0,2)}+Sq^{(7,1,2)})  5_{2}
+(Sq^{(15,2)}+Sq^{(9,4)}+Sq^{(3,6)}+Sq^{(0,7)}+Sq^{(7,0,2)})  5_{3}
+(Sq^{20}+Sq^{(5,5)}+Sq^{(2,6)}+Sq^{(13,0,1)})  5_{4}
+(Sq^{(6,4)}+Sq^{(3,5)})  5_{6}
+(Sq^{(11,2)}+Sq^{(8,3)}+Sq^{(10,0,1)})  5_{7}
+(Sq^{15}+Sq^{(0,5)}+Sq^{(1,0,2)})  5_{8}
+(Sq^{14}+Sq^{(2,4)}+Sq^{(7,0,1)}+Sq^{(0,0,2)})  5_{9}
+(Sq^{5}+Sq^{(2,1)})  5_{12}
+Sq^{4}  5_{13}
$

\item[42:] $d(6_{14}) =
(Sq^{(31,2)}+Sq^{(25,4)}+Sq^{(7,10)}+Sq^{(1,12)}+Sq^{(14,3,2)}+Sq^{(11,4,2)}+Sq^{(8,5,2)}+Sq^{(2,7,2)}+Sq^{(13,1,3)}+Sq^{(10,2,3)}+Sq^{(4,4,3)}+Sq^{(9,0,4)})  5_{0}
+(Sq^{28}+Sq^{(16,4)}+Sq^{(13,5)}+Sq^{(10,6)}+Sq^{(7,7)}+Sq^{(4,8)}+Sq^{(6,5,1)}+Sq^{(3,6,1)}+Sq^{(0,7,1)})  5_{1}
+(Sq^{(17,3)}+Sq^{(5,7)}+Sq^{(9,1,2)}+Sq^{(0,4,2)}+Sq^{(5,0,3)})  5_{2}
+(Sq^{23}+Sq^{(17,2)}+Sq^{(14,3)}+Sq^{(8,5)}+Sq^{(5,6)}+Sq^{(9,0,2)}+Sq^{(6,1,2)})  5_{3}
+(Sq^{22}+Sq^{(15,0,1)}+Sq^{(3,4,1)}+Sq^{(0,5,1)})  5_{4}
+(Sq^{(16,2)}+Sq^{(10,4)}+Sq^{(8,0,2)})  5_{5}
+(Sq^{(11,3)}+Sq^{(5,5)}+Sq^{(2,6)})  5_{6}
+(Sq^{(13,2)}+Sq^{(10,3)}+Sq^{(4,5)}+Sq^{(12,0,1)}+Sq^{(2,1,2)})  5_{7}
+(Sq^{(2,5)}+Sq^{(0,1,2)})  5_{8}
+(Sq^{(4,4)}+Sq^{(1,5)}+Sq^{(3,2,1)}+Sq^{(2,0,2)})  5_{9}
+(Sq^{(8,2)}+Sq^{(2,4)}+Sq^{(0,0,2)})  5_{10}
+Sq^{(0,2)}  5_{13}
$

\item[43:] $d(6_{15}) =
(Sq^{(29,3)}+Sq^{(26,4)}+Sq^{(14,8)}+Sq^{(8,10)}+Sq^{(5,11)}+Sq^{(2,12)}+Sq^{(24,0,2)}+Sq^{(9,5,2)}+Sq^{(0,8,2)}+Sq^{(11,2,3)}+Sq^{(5,4,3)}+Sq^{(4,2,4)}+Sq^{(1,3,4)}+Sq^{(3,0,5)})  5_{0}
+(Sq^{29}+Sq^{(5,8)}+Sq^{(10,4,1)})  5_{1}
+(Sq^{27}+Sq^{(15,4)}+Sq^{(9,6)}+Sq^{(6,7)}+Sq^{(3,8)}+Sq^{(13,0,2)}+Sq^{(10,1,2)}+Sq^{(6,0,3)})  5_{2}
+(Sq^{24}+Sq^{(18,2)}+Sq^{(15,3)}+Sq^{(12,4)}+Sq^{(9,5)}+Sq^{(10,0,2)}+Sq^{(7,1,2)}+Sq^{(4,2,2)}+Sq^{(1,3,2)}+Sq^{(3,0,3)})  5_{3}
+(Sq^{(20,1)}+Sq^{(11,4)}+Sq^{(8,5)}+Sq^{(5,6)}+Sq^{(16,0,1)}+Sq^{(8,0,0,1)})  5_{4}
+(Sq^{(14,3)}+Sq^{(11,4)}+Sq^{(5,6)}+Sq^{(2,7)}+Sq^{(9,0,2)}+Sq^{(6,1,2)}+Sq^{(3,2,2)})  5_{5}
+(Sq^{21}+Sq^{(15,2)}+Sq^{(3,6)}+Sq^{(0,7)}+Sq^{(7,0,2)}+Sq^{(0,0,3)})  5_{6}
+(Sq^{(17,1)}+Sq^{(14,2)}+Sq^{(11,3)}+Sq^{(8,4)}+Sq^{(2,6)}+Sq^{(7,2,1)}+Sq^{(1,4,1)}+Sq^{(3,1,2)}+Sq^{(0,2,2)})  5_{7}
+(Sq^{(9,3)}+Sq^{(1,1,2)})  5_{8}
+(Sq^{(2,5)}+Sq^{(10,0,1)}+Sq^{(4,2,1)}+Sq^{(0,1,2)})  5_{9}
+(Sq^{14}+Sq^{(11,1)}+Sq^{(8,2)}+Sq^{(2,4)}+Sq^{(7,0,1)}+Sq^{(4,1,1)}+Sq^{(0,0,2)})  5_{11}
+Sq^{1}  5_{17}
$

\item[44:] $d(6_{16}) =
(Sq^{(30,3)}+Sq^{(27,4)}+Sq^{(12,9)}+Sq^{(9,10)}+Sq^{(6,11)}+Sq^{(3,12)}+Sq^{(13,4,2)}+Sq^{(10,5,2)}+Sq^{(15,1,3)}+Sq^{(12,2,3)}+Sq^{(9,3,3)}+Sq^{(3,5,3)}+Sq^{(5,2,4)})  5_{0}
+(Sq^{(18,4)}+Sq^{(15,5)}+Sq^{(12,6)}+Sq^{(6,8)}+Sq^{(0,10)}+Sq^{(11,4,1)}+Sq^{(8,5,1)}+Sq^{(15,0,0,1)})  5_{1}
+(Sq^{28}+Sq^{(22,2)}+Sq^{(16,4)}+Sq^{(7,7)}+Sq^{(14,0,2)}+Sq^{(11,1,2)}+Sq^{(8,2,2)})  5_{2}
+(Sq^{(13,4)}+Sq^{(5,2,2)})  5_{3}
+(Sq^{(12,4)}+Sq^{(6,6)}+Sq^{(17,0,1)})  5_{4}
+(Sq^{24}+Sq^{(15,3)}+Sq^{(12,4)}+Sq^{(6,6)}+Sq^{(3,7)})  5_{5}
+(Sq^{(7,5)}+Sq^{(1,0,3)})  5_{6}
+(Sq^{(9,4)}+Sq^{(6,5)}+Sq^{(3,6)}+Sq^{(0,7)}+Sq^{(7,0,2)}+Sq^{(1,2,2)}+Sq^{(6,0,0,1)}+Sq^{(0,2,0,1)})  5_{7}
+(Sq^{(7,4)}+Sq^{(4,5)}+Sq^{(5,0,2)}+Sq^{(2,1,2)})  5_{8}
+(Sq^{(12,2)}+Sq^{(3,5)}+Sq^{(11,0,1)}+Sq^{(4,0,2)}+Sq^{(1,1,2)})  5_{9}
+(Sq^{(10,2)}+Sq^{(4,4)})  5_{10}
+(Sq^{9}+Sq^{(6,1)})  5_{12}
+Sq^{(1,0,1)}  5_{13}
+Sq^{(3,1)}  5_{15}
$

\item[44:] $d(6_{17}) =
(Sq^{(27,4)}+Sq^{(21,6)}+Sq^{(15,8)}+Sq^{(12,9)}+Sq^{(3,12)}+Sq^{(25,0,2)}+Sq^{(10,5,2)}+Sq^{(1,8,2)}+Sq^{(15,1,3)}+Sq^{(3,5,3)}+Sq^{(11,0,4)}+Sq^{(8,1,4)})  5_{0}
+(Sq^{30}+Sq^{(18,4)}+Sq^{(6,8)}+Sq^{(3,9)}+Sq^{(23,0,1)}+Sq^{(11,4,1)}+Sq^{(2,7,1)}+Sq^{(2,0,4)}+Sq^{(15,0,0,1)}+Sq^{(3,4,0,1)})  5_{1}
+(Sq^{(22,2)}+Sq^{(19,3)}+Sq^{(16,4)}+Sq^{(10,6)}+Sq^{(1,9)}+Sq^{(11,1,2)}+Sq^{(8,2,2)}+Sq^{(2,4,2)}+Sq^{(0,0,4)})  5_{2}
+(Sq^{25}+Sq^{(1,8)}+Sq^{(11,0,2)}+Sq^{(5,2,2)}+Sq^{(1,1,3)})  5_{3}
+(Sq^{24}+Sq^{(12,4)}+Sq^{(6,6)}+Sq^{(0,8)}+Sq^{(2,5,1)})  5_{4}
+(Sq^{24}+Sq^{(15,3)}+Sq^{(12,4)}+Sq^{(0,8)}+Sq^{(10,0,2)}+Sq^{(4,2,2)})  5_{5}
+(Sq^{(13,3)}+Sq^{(10,4)}+Sq^{(7,5)}+Sq^{(8,0,2)}+Sq^{(2,2,2)}+Sq^{(1,0,3)})  5_{6}
+(Sq^{(18,1)}+Sq^{(9,4)}+Sq^{(6,5)}+Sq^{(8,2,1)}+Sq^{(1,2,2)}+Sq^{(0,0,3)})  5_{7}
+(Sq^{(13,2)}+Sq^{(10,3)})  5_{8}
+(Sq^{18}+Sq^{(15,1)}+Sq^{(12,2)}+Sq^{(9,3)}+Sq^{(3,5)}+Sq^{(1,1,2)})  5_{9}
+Sq^{16}  5_{10}
+(Sq^{15}+Sq^{(12,1)}+Sq^{(9,2)}+Sq^{(0,0,0,1)})  5_{11}
+(Sq^{9}+Sq^{(6,1)})  5_{12}
+Sq^{8}  5_{13}
+Sq^{6}  5_{15}
+(Sq^{4}+Sq^{(1,1)})  5_{16}
+Sq^{2}  5_{17}
$

\end{itemize}

\subsection{Homological degree 7}

Complete through degree $t=44$.

\begin{itemize}
\addtolength{\itemsep}{2.4ex}

\item[7:] $d(7_{0}) =
Sq^{1}  6_{0}
$

\item[18:] $d(7_{1}) =
Sq^{12}  6_{0}
+Sq^{2}  6_{1}
+Sq^{1}  6_{2}
$

\item[22:] $d(7_{2}) =
(Sq^{16}+Sq^{(10,2)})  6_{0}
+(Sq^{6}+Sq^{(0,2)})  6_{1}
+Sq^{1}  6_{4}
$

\item[23:] $d(7_{3}) =
(Sq^{17}+Sq^{(11,2)})  6_{0}
+Sq^{(0,0,1)}  6_{1}
+Sq^{6}  6_{2}
+(Sq^{3}+Sq^{(0,1)})  6_{3}
$

\item[24:] $d(7_{4}) =
Sq^{8}  6_{1}
+Sq^{4}  6_{3}
+Sq^{2}  6_{5}
+Sq^{1}  6_{6}
$

\item[30:] $d(7_{5}) =
(Sq^{(15,3)}+Sq^{(12,4)}+Sq^{(6,6)}+Sq^{(3,7)}+Sq^{(10,0,2)})  6_{0}
+(Sq^{14}+Sq^{(8,2)}+Sq^{(2,4)}+Sq^{(7,0,1)}+Sq^{(4,1,1)}+Sq^{(0,0,2)})  6_{1}
+Sq^{(1,4)}  6_{2}
+Sq^{(1,0,1)}  6_{5}
+Sq^{7}  6_{6}
+Sq^{(1,1)}  6_{7}
$

\item[33:] $d(7_{6}) =
(Sq^{(15,4)}+Sq^{(6,7)}+Sq^{(0,9)})  6_{0}
+(Sq^{(14,1)}+Sq^{(10,0,1)}+Sq^{(3,0,2)})  6_{1}
+Sq^{(7,3)}  6_{2}
+(Sq^{(7,2)}+Sq^{(0,2,1)})  6_{3}
+(Sq^{(5,2)}+Sq^{(4,0,1)})  6_{5}
+Sq^{(4,2)}  6_{6}
+(Sq^{(4,1)}+Sq^{(1,2)}+Sq^{(0,0,1)})  6_{7}
+Sq^{(1,1)}  6_{8}
$

\item[36:] $d(7_{7}) =
(Sq^{(15,5)}+Sq^{(12,6)}+Sq^{(6,8)}+Sq^{(0,10)}+Sq^{(10,2,2)}+Sq^{(9,0,3)})  6_{0}
+(Sq^{(17,1)}+Sq^{(8,4)}+Sq^{(5,5)}+Sq^{(2,6)}+Sq^{(13,0,1)}+Sq^{(1,4,1)}+Sq^{(6,0,2)}+Sq^{(0,2,2)})  6_{1}
+(Sq^{19}+Sq^{(7,4)}+Sq^{(4,5)})  6_{2}
+(Sq^{(4,4)}+Sq^{(1,5)}+Sq^{(3,2,1)}+Sq^{(2,0,2)})  6_{3}
+Sq^{15}  6_{4}
+(Sq^{(2,4)}+Sq^{(7,0,1)}+Sq^{(4,1,1)}+Sq^{(0,0,2)})  6_{5}
+(Sq^{(7,2)}+Sq^{(1,4)})  6_{6}
+(Sq^{(7,1)}+Sq^{(4,2)})  6_{7}
+(Sq^{(4,1)}+Sq^{(1,2)}+Sq^{(0,0,1)})  6_{8}
+Sq^{(1,1)}  6_{9}
$

\item[37:] $d(7_{8}) =
(Sq^{(8,3,2)}+Sq^{(2,5,2)}+Sq^{(10,0,3)})  6_{0}
+(Sq^{(6,5)}+Sq^{(0,7)}+Sq^{(14,0,1)}+Sq^{(7,0,2)})  6_{1}
+(Sq^{(8,4)}+Sq^{(3,1,2)}+Sq^{(0,2,2)})  6_{2}
+(Sq^{17}+Sq^{(14,1)}+Sq^{(11,2)}+Sq^{(5,4)}+Sq^{(3,0,2)})  6_{3}
+(Sq^{15}+Sq^{(8,0,1)}+Sq^{(0,0,0,1)})  6_{5}
+(Sq^{14}+Sq^{(2,4)})  6_{6}
+(Sq^{11}+Sq^{(8,1)}+Sq^{(5,2)}+Sq^{(2,3)}+Sq^{(4,0,1)})  6_{7}
+Sq^{(5,1)}  6_{8}
+Sq^{5}  6_{9}
+Sq^{1}  6_{10}
$

\item[38:] $d(7_{9}) =
(Sq^{32}+Sq^{(26,2)}+Sq^{(8,8)}+Sq^{(5,9)}+Sq^{(2,10)}+Sq^{(15,1,2)}+Sq^{(12,2,2)}+Sq^{(9,3,2)}+Sq^{(6,4,2)}+Sq^{(8,1,3)})  6_{0}
+(Sq^{22}+Sq^{(19,1)}+Sq^{(16,2)}+Sq^{(10,4)}+Sq^{(7,5)}+Sq^{(4,6)}+Sq^{(1,7)}+Sq^{(15,0,1)}+Sq^{(12,1,1)}+Sq^{(3,4,1)}+Sq^{(0,5,1)}+Sq^{(8,0,2)}+Sq^{(2,2,2)}+Sq^{(7,0,0,1)}+Sq^{(4,1,0,1)}+Sq^{(0,0,1,1)})  6_{1}
+Sq^{1}  6_{11}
$

\item[39:] $d(7_{10}) =
(Sq^{(27,2)}+Sq^{(15,6)}+Sq^{(12,7)}+Sq^{(9,8)}+Sq^{(3,10)}+Sq^{(10,3,2)})  6_{0}
+(Sq^{(20,1)}+Sq^{(11,4)}+Sq^{(5,6)}+Sq^{(2,7)}+Sq^{(13,1,1)}+Sq^{(9,0,2)}+Sq^{(8,0,0,1)})  6_{1}
+(Sq^{(16,2)}+Sq^{(10,4)}+Sq^{(7,5)}+Sq^{(4,6)})  6_{2}
+(Sq^{19}+Sq^{(16,1)}+Sq^{(10,3)}+Sq^{(7,4)}+Sq^{(4,5)})  6_{3}
+Sq^{(12,2)}  6_{4}
+(Sq^{(10,0,1)}+Sq^{(7,1,1)})  6_{5}
+Sq^{(7,3)}  6_{6}
+(Sq^{(4,3)}+Sq^{(6,0,1)}+Sq^{(0,2,1)})  6_{7}
+(Sq^{(7,1)}+Sq^{(4,2)})  6_{8}
+(Sq^{7}+Sq^{(4,1)}+Sq^{(1,2)}+Sq^{(0,0,1)})  6_{9}
+Sq^{3}  6_{10}
$

\item[40:] $d(7_{11}) =
(Sq^{(13,7)}+Sq^{(10,8)}+Sq^{(4,10)}+Sq^{(11,3,2)}+Sq^{(10,1,3)})  6_{0}
+(Sq^{24}+Sq^{(21,1)}+Sq^{(17,0,1)}+Sq^{(5,4,1)}+Sq^{(9,0,0,1)}+Sq^{(6,1,0,1)})  6_{1}
+(Sq^{23}+Sq^{(5,6)}+Sq^{(2,7)})  6_{2}
+(Sq^{(17,1)}+Sq^{(14,2)}+Sq^{(11,3)}+Sq^{(5,5)}+Sq^{(1,4,1)})  6_{3}
+(Sq^{19}+Sq^{(10,3)}+Sq^{(7,4)})  6_{4}
+(Sq^{18}+Sq^{(15,1)}+Sq^{(12,2)})  6_{5}
+Sq^{(7,0,1)}  6_{7}
+(Sq^{11}+Sq^{(8,1)}+Sq^{(2,3)})  6_{8}
+Sq^{(5,1)}  6_{9}
+Sq^{4}  6_{10}
+Sq^{2}  6_{12}
$

\item[42:] $d(7_{12}) =
(Sq^{(30,2)}+Sq^{(27,3)}+Sq^{(12,8)}+Sq^{(9,9)}+Sq^{(6,10)}+Sq^{(0,12)}+Sq^{(1,7,2)}+Sq^{(15,0,3)}+Sq^{(3,4,3)})  6_{0}
+(Sq^{26}+Sq^{(23,1)}+Sq^{(11,5)}+Sq^{(5,7)}+Sq^{(16,1,1)}+Sq^{(3,3,2)}+Sq^{(11,0,0,1)}+Sq^{(8,1,0,1)}+Sq^{(4,0,1,1)})  6_{1}
+(Sq^{25}+Sq^{(10,5)}+Sq^{(5,2,2)}+Sq^{(1,1,3)})  6_{2}
+(Sq^{(19,1)}+Sq^{(7,5)}+Sq^{(9,2,1)})  6_{3}
+(Sq^{(9,4)}+Sq^{(6,5)})  6_{4}
+(Sq^{(13,0,1)}+Sq^{(5,0,0,1)})  6_{5}
+(Sq^{(13,2)}+Sq^{(7,4)}+Sq^{(1,6)})  6_{6}
+(Sq^{(13,1)}+Sq^{(9,0,1)}+Sq^{(3,2,1)})  6_{7}
+(Sq^{(7,2)}+Sq^{(4,3)}+Sq^{(6,0,1)}+Sq^{(0,2,1)})  6_{8}
+Sq^{(0,2)}  6_{10}
$

\item[44:] $d(7_{13}) =
(Sq^{(14,8)}+Sq^{(2,12)}+Sq^{(12,4,2)}+Sq^{(9,5,2)}+Sq^{(3,7,2)}+Sq^{(0,8,2)}+Sq^{(14,1,3)}+Sq^{(11,2,3)}+Sq^{(5,4,3)}+Sq^{(7,1,4)})  6_{0}
+(Sq^{(25,1)}+Sq^{(13,5)}+Sq^{(10,6)}+Sq^{(7,7)}+Sq^{(4,8)}+Sq^{(1,9)}+Sq^{(6,5,1)}+Sq^{(0,7,1)}+Sq^{(14,0,2)}+Sq^{(5,3,2)}+Sq^{(2,4,2)}+Sq^{(13,0,0,1)}+Sq^{(1,4,0,1)}+Sq^{(6,0,1,1)})  6_{1}
+(Sq^{(18,3)}+Sq^{(3,8)}+Sq^{(0,9)}+Sq^{(7,2,2)}+Sq^{(4,3,2)}+Sq^{(6,0,3)}+Sq^{(3,1,3)}+Sq^{(0,2,3)})  6_{2}
+(Sq^{(21,1)}+Sq^{(12,4)}+Sq^{(3,7)}+Sq^{(11,2,1)}+Sq^{(5,4,1)}+Sq^{(10,0,2)}+Sq^{(1,3,2)}+Sq^{(3,0,3)}+Sq^{(0,1,3)})  6_{3}
+(Sq^{23}+Sq^{(8,5)})  6_{4}
+(Sq^{(19,1)}+Sq^{(10,4)}+Sq^{(7,5)}+Sq^{(4,6)}+Sq^{(15,0,1)}+Sq^{(12,1,1)}+Sq^{(3,4,1)}+Sq^{(0,5,1)}+Sq^{(8,0,2)}+Sq^{(2,2,2)}+Sq^{(4,1,0,1)}+Sq^{(0,0,1,1)})  6_{5}
+(Sq^{(6,5)}+Sq^{(7,0,2)}+Sq^{(1,2,2)})  6_{6}
+(Sq^{(3,5)}+Sq^{(0,6)}+Sq^{(5,2,1)}+Sq^{(4,0,2)})  6_{7}
+(Sq^{15}+Sq^{(12,1)}+Sq^{(0,5)}+Sq^{(1,0,2)}+Sq^{(0,0,0,1)})  6_{8}
+Sq^{(3,3)}  6_{9}
+Sq^{8}  6_{10}
+Sq^{6}  6_{12}
+Sq^{4}  6_{13}
+Sq^{2}  6_{14}
$

\item[44:] $d(7_{14}) =
(Sq^{(29,3)}+Sq^{(11,9)}+Sq^{(5,11)}+Sq^{(2,12)}+Sq^{(15,3,2)}+Sq^{(12,4,2)}+Sq^{(3,7,2)}+Sq^{(0,8,2)}+Sq^{(14,1,3)}+Sq^{(11,2,3)}+Sq^{(8,3,3)}+Sq^{(2,5,3)}+Sq^{(7,1,4)})  6_{0}
+(Sq^{28}+Sq^{(22,2)}+Sq^{(13,5)}+Sq^{(7,7)}+Sq^{(4,8)}+Sq^{(1,9)}+Sq^{(21,0,1)}+Sq^{(6,5,1)}+Sq^{(3,6,1)}+Sq^{(14,0,2)}+Sq^{(5,3,2)}+Sq^{(13,0,0,1)}+Sq^{(10,1,0,1)}+Sq^{(1,4,0,1)}+Sq^{(6,0,1,1)})  6_{1}
+(Sq^{(15,4)}+Sq^{(12,5)}+Sq^{(10,1,2)}+Sq^{(4,3,2)}+Sq^{(6,0,3)}+Sq^{(0,2,3)})  6_{2}
+(Sq^{24}+Sq^{(21,1)}+Sq^{(18,2)}+Sq^{(9,5)}+Sq^{(10,0,2)}+Sq^{(7,1,2)}+Sq^{(0,1,3)})  6_{3}
+(Sq^{22}+Sq^{(16,2)}+Sq^{(7,5)}+Sq^{(15,0,1)}+Sq^{(12,1,1)}+Sq^{(0,5,1)}+Sq^{(8,0,2)}+Sq^{(2,2,2)}+Sq^{(7,0,0,1)}+Sq^{(4,1,0,1)}+Sq^{(0,0,1,1)})  6_{5}
+(Sq^{(15,2)}+Sq^{(9,4)}+Sq^{(6,5)}+Sq^{(3,6)}+Sq^{(7,0,2)})  6_{6}
+(Sq^{(15,1)}+Sq^{(6,4)}+Sq^{(11,0,1)}+Sq^{(5,2,1)}+Sq^{(4,0,2)}+Sq^{(1,1,2)})  6_{7}
+(Sq^{15}+Sq^{(12,1)}+Sq^{(1,0,2)}+Sq^{(0,0,0,1)})  6_{8}
+(Sq^{12}+Sq^{(6,2)}+Sq^{(3,3)}+Sq^{(0,4)})  6_{9}
+Sq^{8}  6_{10}
+Sq^{1}  6_{15}
$

\end{itemize}

\subsection{Homological degree 8}

Complete through degree $t=44$.

\begin{itemize}
\addtolength{\itemsep}{2.4ex}

\item[8:] $d(8_{0}) =
Sq^{1}  7_{0}
$

\item[23:] $d(8_{1}) =
Sq^{16}  7_{0}
+(Sq^{5}+Sq^{(2,1)})  7_{1}
+Sq^{1}  7_{2}
$

\item[25:] $d(8_{2}) =
(Sq^{18}+Sq^{(12,2)})  7_{0}
+(Sq^{7}+Sq^{(1,2)}+Sq^{(0,0,1)})  7_{1}
+Sq^{2}  7_{3}
$

\item[30:] $d(8_{3}) =
(Sq^{(17,2)}+Sq^{(14,3)}+Sq^{(8,5)})  7_{0}
+(Sq^{(9,1)}+Sq^{(6,2)})  7_{1}
+(Sq^{7}+Sq^{(4,1)}+Sq^{(0,0,1)})  7_{3}
+Sq^{(3,1)}  7_{4}
$

\item[31:] $d(8_{4}) =
(Sq^{(15,3)}+Sq^{(12,4)}+Sq^{(9,5)})  7_{0}
+(Sq^{13}+Sq^{(10,1)}+Sq^{(7,2)}+Sq^{(4,3)}+Sq^{(1,4)}+Sq^{(6,0,1)}+Sq^{(0,2,1)})  7_{1}
+Sq^{1}  7_{5}
$

\item[33:] $d(8_{5}) =
(Sq^{26}+Sq^{(11,5)}+Sq^{(12,0,2)})  7_{0}
+(Sq^{(12,1)}+Sq^{(1,0,2)}+Sq^{(0,0,0,1)})  7_{1}
+(Sq^{10}+Sq^{(7,1)})  7_{3}
+(Sq^{(6,1)}+Sq^{(3,2)}+Sq^{(0,3)}+Sq^{(2,0,1)})  7_{4}
+Sq^{3}  7_{5}
$

\item[34:] $d(8_{6}) =
(Sq^{27}+Sq^{(21,2)}+Sq^{(18,3)}+Sq^{(15,4)}+Sq^{(12,5)}+Sq^{(6,7)}+Sq^{(3,8)}+Sq^{(13,0,2)})  7_{0}
+(Sq^{(13,1)}+Sq^{(4,4)}+Sq^{(6,1,1)}+Sq^{(3,2,1)}+Sq^{(1,0,0,1)})  7_{1}
+(Sq^{11}+Sq^{(4,0,1)})  7_{3}
+(Sq^{(7,1)}+Sq^{(1,3)}+Sq^{(0,1,1)})  7_{4}
+Sq^{4}  7_{5}
+Sq^{1}  7_{6}
$

\item[36:] $d(8_{7}) =
(Sq^{29}+Sq^{(23,2)}+Sq^{(20,3)}+Sq^{(5,8)}+Sq^{(12,1,2)})  7_{0}
+(Sq^{(6,4)}+Sq^{(0,6)}+Sq^{(1,1,2)}+Sq^{(0,1,0,1)})  7_{1}
+(Sq^{(10,1)}+Sq^{(6,0,1)})  7_{3}
+(Sq^{(6,2)}+Sq^{(3,3)}+Sq^{(5,0,1)})  7_{4}
+(Sq^{6}+Sq^{(0,2)})  7_{5}
+Sq^{3}  7_{6}
$

\item[37:] $d(8_{8}) =
(Sq^{(21,3)}+Sq^{(15,5)}+Sq^{(6,8)}+Sq^{(3,9)}+Sq^{(0,10)})  7_{0}
+Sq^{(7,0,1)}  7_{3}
+(Sq^{(4,3)}+Sq^{(1,4)}+Sq^{(0,2,1)})  7_{4}
+Sq^{7}  7_{5}
+Sq^{4}  7_{6}
+Sq^{1}  7_{7}
$

\item[38:] $d(8_{9}) =
(Sq^{(19,4)}+Sq^{(16,5)}+Sq^{(17,0,2)}+Sq^{(14,1,2)}+Sq^{(8,3,2)}+Sq^{(10,0,3)})  7_{0}
+(Sq^{(5,5)}+Sq^{(2,6)}+Sq^{(6,0,2)})  7_{1}
+(Sq^{(3,4)}+Sq^{(0,5)}+Sq^{(0,0,0,1)})  7_{3}
+(Sq^{(11,1)}+Sq^{(5,3)}+Sq^{(7,0,1)}+Sq^{(4,1,1)})  7_{4}
+(Sq^{8}+Sq^{(2,2)})  7_{5}
+Sq^{1}  7_{8}
$

\item[39:] $d(8_{10}) =
(Sq^{(23,3)}+Sq^{(14,6)}+Sq^{(8,8)}+Sq^{(2,10)}+Sq^{(11,0,3)})  7_{0}
+(Sq^{21}+Sq^{(12,3)}+Sq^{(9,4)}+Sq^{(6,5)}+Sq^{(0,7)}+Sq^{(14,0,1)}+Sq^{(8,2,1)}+Sq^{(2,4,1)})  7_{1}
+Sq^{17}  7_{2}
+(Sq^{(13,1)}+Sq^{(9,0,1)})  7_{3}
+(Sq^{(6,3)}+Sq^{(3,4)})  7_{4}
+Sq^{9}  7_{5}
+(Sq^{6}+Sq^{(0,2)})  7_{6}
+Sq^{3}  7_{7}
$

\item[39:] $d(8_{11}) =
(Sq^{32}+Sq^{(8,8)}+Sq^{(5,9)}+Sq^{(15,1,2)}+Sq^{(12,2,2)}+Sq^{(9,3,2)}+Sq^{(11,0,3)})  7_{0}
+(Sq^{21}+Sq^{(18,1)}+Sq^{(15,2)}+Sq^{(12,3)}+Sq^{(9,4)}+Sq^{(6,5)}+Sq^{(3,6)}+Sq^{(0,7)}+Sq^{(14,0,1)}+Sq^{(11,1,1)}+Sq^{(8,2,1)}+Sq^{(2,4,1)}+Sq^{(7,0,2)}+Sq^{(4,1,2)}+Sq^{(1,2,2)}+Sq^{(0,0,3)}+Sq^{(6,0,0,1)}+Sq^{(0,2,0,1)})  7_{1}
+Sq^{1}  7_{9}
$

\item[40:] $d(8_{12}) =
(Sq^{33}+Sq^{(27,2)}+Sq^{(24,3)}+Sq^{(12,7)}+Sq^{(9,8)}+Sq^{(6,9)}+Sq^{(3,10)}+Sq^{(0,11)}+Sq^{(12,0,3)})  7_{0}
+(Sq^{(13,3)}+Sq^{(7,5)}+Sq^{(1,7)}+Sq^{(3,4,1)}+Sq^{(0,5,1)}+Sq^{(1,0,3)}+Sq^{(0,0,1,1)})  7_{1}
+(Sq^{17}+Sq^{(10,0,1)})  7_{3}
+(Sq^{(4,4)}+Sq^{(9,0,1)}+Sq^{(3,2,1)})  7_{4}
+Sq^{(4,2)}  7_{5}
+Sq^{7}  7_{6}
+Sq^{4}  7_{7}
+Sq^{1}  7_{10}
$

\item[42:] $d(8_{13}) =
(Sq^{(29,2)}+Sq^{(17,6)}+Sq^{(11,8)}+Sq^{(12,3,2)}+Sq^{(6,5,2)}+Sq^{(3,6,2)}+Sq^{(11,1,3)})  7_{0}
+(Sq^{(18,2)}+Sq^{(12,4)}+Sq^{(9,5)}+Sq^{(14,1,1)}+Sq^{(5,4,1)}+Sq^{(7,1,2)}+Sq^{(3,0,3)}+Sq^{(6,1,0,1)}+Sq^{(3,2,0,1)})  7_{1}
+Sq^{(14,2)}  7_{2}
+(Sq^{(16,1)}+Sq^{(1,6)}+Sq^{(12,0,1)})  7_{3}
+(Sq^{(6,4)}+Sq^{(0,6)})  7_{4}
+(Sq^{(6,2)}+Sq^{(0,4)})  7_{5}
+Sq^{9}  7_{6}
+(Sq^{6}+Sq^{(0,2)})  7_{7}
+Sq^{5}  7_{8}
+Sq^{3}  7_{10}
$

\item[43:] $d(8_{14}) =
(Sq^{(21,5)}+Sq^{(15,7)}+Sq^{(3,11)}+Sq^{(0,12)}+Sq^{(12,1,3)})  7_{0}
+(Sq^{(16,3)}+Sq^{(10,5)}+Sq^{(4,7)}+Sq^{(3,5,1)}+Sq^{(5,2,2)}+Sq^{(2,3,2)}+Sq^{(4,0,3)}+Sq^{(1,1,3)}+Sq^{(4,2,0,1)}+Sq^{(0,1,1,1)})  7_{1}
+(Sq^{21}+Sq^{(12,3)})  7_{2}
+(Sq^{20}+Sq^{(17,1)}+Sq^{(5,5)})  7_{3}
+(Sq^{(13,2)}+Sq^{(10,3)}+Sq^{(7,4)}+Sq^{(12,0,1)}+Sq^{(9,1,1)}+Sq^{(6,2,1)}+Sq^{(1,1,0,1)})  7_{4}
+Sq^{13}  7_{5}
+Sq^{(4,2)}  7_{6}
+Sq^{7}  7_{7}
+Sq^{4}  7_{10}
+Sq^{3}  7_{11}
+Sq^{1}  7_{12}
$

\end{itemize}


\begin{bibdiv}
\begin{biblist}


\bib{Hinf}{book}{
   author={Bruner, R. R.},
   author={May, J. P.},
   author={McClure, J. E.},
   author={Steinberger, M.},
   title={$H\sb \infty $ ring spectra and their applications},
   series={Lecture Notes in Mathematics},
   volume={1176},
   publisher={Springer-Verlag, Berlin},
   date={1986},
   pages={viii+388},
   isbn={3-540-16434-0},
   review={\MR{836132}},
   doi={10.1007/BFb0075405},
}
\bib{V168}{article}{
   author={May, J. Peter},
   title={A general algebraic approach to Steenrod operations},
   conference={
      title={The Steenrod Algebra and its Applications (Proc. Conf. to
      Celebrate N. E. Steenrod's Sixtieth Birthday, Battelle Memorial Inst.,
      Columbus, Ohio, 1970)},
   },
   book={
      series={Lecture Notes in Mathematics, Vol. 168},
      publisher={Springer, Berlin},
   },
   date={1970},
   pages={153--231},
   review={\MR{0281196}},
}

\bib{Nas97}{misc}{
   author={Nassau, Christian},
   date={1997},
   note={private communication},
}



\end{biblist}
\end{bibdiv}
\end{document}